\pdfoutput=1   
\documentclass[12pt,twoside]{article}

    \usepackage{graphicx}
    \usepackage{styleset}
    \usepackage{macros}
    \let\subsubsection\subparagraph

    \title  {Tait colorings, and an instanton homology for webs and foams}

    \author {P. B. Kronheimer and T. S. Mrowka%
      \thanks{%
        The work of the first author was supported by the National
        Science Foundation through NSF grants DMS-0904589 and
        DMS-1405652. The work of the second author was supported by
        NSF grants DMS-0805841 and DMS-1406348.}}

    \address {Harvard University, Cambridge MA 02138 \\
              Massachusetts Institute of Technology, Cambridge MA 02139}

\begin{document}

\maketitle

\begin{abstract}
   We use $\SO(3)$ gauge theory to define a functor from a category of
   unoriented webs and foams to the category of finite-dimensional
   vector spaces over the field of two elements.  We prove a
   non-vanishing theorem for this $\SO(3)$ instanton homology of
   webs, using Gabai's sutured manifold theory. It is hoped that the
   non-vanishing theorem may support a program to provide a new proof
   of the four-color theorem.
\end{abstract}

\clearpage

{\small \tableofcontents}

\clearpage

\section{Introduction}

\subsection{Statement of results}

By a \emph{web in $\R^{3}$} we shall mean an embedded trivalent graph.
More specifically, a web will be a  compact subset
$K \subset\R^{3}$ with a finite set $V$ of distinguished points, the
\emph{vertices}, such that $K\sminus V$ is a smooth
1-dimensional submanifold of $\R^{3}$, and such that each vertex has a
neighborhood in which $K$ is diffeomorphic to three distinct, coplanar rays in
$\R^{3}$. We shall refer to the components of $K\sminus V$ as the
\emph{edges}. Note that our definition allows an edge to be a
(possibly knotted) circle. The vertex set may be empty.

An edge $e$ of $K$ is an \emph{embedded bridge} if there is a smoothly
embedded $2$-sphere in $\R^{3}$ which meets $e$ transversely in a single
point and is otherwise disjoint from $K$. We use the term ``embedded
bridge'' to distinguish this from the more general notion of a
\emph{bridge}, which is an edge whose removal increases the number of
connected components of $K$.  

A \emph{Tait coloring} of $K$ is a
function from the edges of $K$ to a $3$-element set of ``colors''
$\{\,1,2,3\,\}$ such that edges of three different colors are incident at
each vertex. Tait colorings which differ only by a permutation of the
colors will still be regarded as distinct; so if $K$ is a single
circle, for example, then it has three Tait colorings.

In this paper, we will show how to associate a finite-dimensional
$\F$-vector space $\Jsharp(K)$ to any web $K$ in $\R^{3}$, using a
variant of the instanton homology for knots from \cite{Floer-Durham}
and \cite{KM-unknot}, but with gauge group $\SO(3)$ replacing the
group $\SU(2)$ as it appeared in \cite{KM-unknot}. (Here $\F$ is the
field of $2$ elements.) We will establish the following non-vanishing
property for $\Jsharp$.

\begin{theorem}\label{thm:non-vanishing}
    For a web $K\subset\R^{3}$, 
    the vector space $\Jsharp(K)$ is zero if and only if $K$ has an embedded bridge.
\end{theorem}

Based on the evidence of small examples, and some general properties
of $\Jsharp$, we also make the following conjecture:

\begin{conjecture}\label{thm:Tait-count}
    If the web $K$ lies in the plane, so $K\subset
    \R^{2}\subset \R^{3}$, then the dimension of $\Jsharp(K)$ is equal to
    the number of Tait colorings of $K$.
\end{conjecture}

The conjecture is true for planar webs which are bipartite. Slightly
more generally, we know that a minimal counterexample cannot contain
any squares, triangles, bigons or circles.  If the conjecture is true
in general, then, together with the preceding theorem, it would
establish that every bridgeless, planar trivalent graph admits a Tait
coloring.  The existence of Tait colorings for bridgeless trivalent
graphs in the plane is equivalent to the four-color theorem (see
\cite{Tait}), so the above conjecture would provide an alternative
proof of the theorem of Appel and Haken
\cite{Appel-Haken-book}: every planar graph admits a four-coloring.

\subsection{Functoriality}

The terminology of graphs as `webs' goes back to Kuperberg
\cite{Kuperberg-spiders} and was used by Khovanov in
\cite{Khovanov-sl3}, where Khovanov's $\sl_{3}$ homology, $F(K)$, was
defined for \emph{oriented, bipartite} webs $K$.  The results of
\cite{Khovanov-sl3} have motivated a lot of the constructions in this
paper, but our webs are more general, being unoriented and without the
bipartite restriction.  At least for planar webs, one can repeat
Khovanov's combinatorial definitions without the bipartite condition,
but only by passing to $\F$ coefficients. (We will take this up
briefly in section~\ref{sec:jflat}.) In this way one can sketch 
a possible combinatorial counterpart to $\Jsharp(K)$, though it
is not clear to the authors how to calculate it, or even that this
combinatorial version is finite-dimensional in general.

A key property of
$\Jsharp(K)$ is its functoriality for certain singular cobordisms
between webs. Following \cite{Khovanov-sl3}, these singular cobordisms
will be called \emph{foams}, though the foams that are appropriate
here are more general than those of \cite{Khovanov-sl3}, in line with
our more general webs (see also \cite{Vaz-et-al}).
A (closed) \emph{foam} $\Sigma$ in $\R^{4}$ will be a
compact $2$-dimensional subcomplex decorated with ``dots''. The
subcomplex is required to have one of the following models at each
point $x\in \Sigma$:
\begin{enumerate}
\item a smoothly embedded $2$-manifold in a neighborhood of $x$; 
\item a neighborhood  modeled on $\R \times K_{3}$ where
$K_{3}\subset \R^{3}$ is the union of three distinct rays meeting at
the origin; or
\item a cone in $\R^{3}\subset \R^{4}$ whose vertex is $x$ and whose base is the union of $4$
    distinct points in $S^{2}$ joined pairwise by $6$ non-intersecting great-circle
    arcs, each of length at most $\pi$.
\end{enumerate}

Points of the third type will be called \emph{tetrahedral points}, because
a neighborhood has the combinatorics of a cone on the $1$-skeleton of a
tetrahedron. (Singular points of this type appear in the foams of
\cite{Vaz-et-al}.) The points of the second type form a union of arcs called
the \emph{seams}. The seams and tetrahedral points together form an
embedded graph in $\R^{4}$ whose vertices have valence $4$. The
complement of the tetrahedral points and seams is a smoothly embedded
2-manifold whose components are the \emph{facets} of the foam. Each
facet is decorated with a number of \emph{dots}:  a possibly empty
collection of marked points.
No orientability is required of
the facets of $\Sigma$. 

Given webs $K$ and $K'$ in $\R^{3}$, we
can also consider \emph{foams with boundary},
$\Sigma\subset [a,b]\times \R^{3}$, modeled on $[a,a+\epsilon) \times
K$ at one end and $(b-\epsilon, b] \times K'$ at the other. The seams
and tetrahedral points comprise a graph with interior vertices of
valence $4$ and vertices of valence $1$ at the boundary.
are either circles or arcs whose endpoints lie on the boundary. We will
refer to such a $\Sigma$ as a (foam)-cobordism from $K$ to $K'$. A cobordism
gives rise to linear map,
\[
          \Jsharp(\Sigma) : \Jsharp(K) \to \Jsharp(K').
\]
Cobordisms can be composed in the obvious way, and composite
cobordisms give rise to composite maps, so we can regard $\Jsharp$ as
a functor from a suitable category of webs and foams to the
category of vector spaces over $\F$.

\subsection{Remarks about the proof of Theorem~\ref{thm:non-vanishing}}

 Theorem~\ref{thm:non-vanishing} belongs to the same family as
other non-vanishing theorems for instanton homology proved by the
authors in \cite{KM-sutures} and \cite{KM-unknot}. As with the earlier
results, the proof rests on the Gabai's existence theorem for sutured
manifold hierarchies, for taut sutured manifolds \cite{Gabai}.

Given a web $K\subset \R^{3}$, let us remove a small open ball around each
of the vertices, to obtain a manifold with boundary, each boundary
component being  sphere with $3$ marked points. Now identify these
boundary components in pairs (the number of vertices in a web 
is always even) to obtain a link $K^{+}$ in the connected sum of
$\R^{3}$ with a number of copies of $S^{1}\times S^{2}$. Add a point
at infinity to replace $\R^{3}$ by $S^{3}$. Let $M$ be the complement
of a tubular neighborhood of $K^{+}$. Form a sutured manifold
$(M,\gamma)$ by putting two meridional sutures on each of the torus
boundary components of $M$. By a sequence of applications of an
excision principle, we will show that the non-vanishing of
$\Jsharp(K)$ is implied by the non-vanishing of the \emph{sutured
  instanton homology}, $\SHI(M,\gamma)$, as defined in
\cite{KM-sutures}. Furthermore, $\SHI(M,\gamma)$ will be non-zero
provided that $(M,\gamma)$ is a \emph{taut} sutured manifold in the
sense of \cite{Gabai}.  Finally, by an elementary topological
argument, we show that the tautness of $(M,\gamma)$ is equivalent to
the conditions that $K\subset \R^{3}$ is not split (i.e. admits no
separating embedded $2$-sphere) and has no embedded
bridge. This establishes Theorem~\ref{thm:non-vanishing} for the case
of non-split webs, which is sufficient because of a multiplicative
property of $\Jsharp$ for split unions.

\subsection{Outline of the paper}

Section~\ref{sec:Preliminaries} deals with the gauge theory that is
used in the definition of $\Jsharp$. Section~\ref{sec:J} defines the functor
$\Jsharp$ and establishes its basic properties. We define $\Jsharp$ more generally
than in this introduction, considering webs in
arbitrary $3$-manifolds. In section~\ref{sec:relations}, we establish
properties of $\Jsharp$ that are sufficient to calculate $\Jsharp(K)$
at least when $K$ is a bipartite planar graph, as well as some other
simple cases. The arguments in section~\ref{sec:relations} depend on
applications of an excision principle which is discussed in section~\ref{sec:excision}.
The proof of Theorem~\ref{thm:non-vanishing} is given 
in
section~\ref{sec:proof-nonvanish}. Some further questions are
discussed in section~\ref{sec:further}.

\subparagraph{Acknowledgement.}  The authors are very grateful for the
support of the Radcliffe Institute for Advanced Study, which provided
them with the opportunity to pursue this project together as Fellows
of the Institute during the academic year 2013--2014.

\section{Preliminaries}
\label{sec:Preliminaries}

\subsection{Orbifolds}

We will consider $3$- and $4$-dimensional orbifolds with the following
local models. In dimension $3$, we require that every singular point of
the orbifold has a neighborhood which is modeled either on $\R \times
(\R^{2} / \{\pm 1\})$ or on $\R^{3} / V_{4}$, where $V_{4}$ is a
standard Klein $4$-group in $\SO(3)$. In the first of these cases, the
singular set is locally a smooth $1$-manifold, and in the second case
the singular set is three half-lines meeting at a single
vertex.  In
dimension $4$, we allow local models which are either products of one of the
above two $3$-dimensional models with $\R$, or one additional case, the
quotient of $\R^{4}$ by the group $V_{8}\subset \SO(4)$ consisting of
all diagonal matrices with entries $\pm 1$ and determinant $1$.

Every point in such an orbifold
has a neighborhood $U$ which is the codomain of an orbifold chart,
\[
         \phi : \tilde U \to U.
\]
The map $\phi$ is a quotient map for an action of a group $H$ which is
either trivial, $\F$,
$V_{4}$ or $V_{8}$. If $x$ is a point in $U$ and $\tilde x$ a preimage point, we
write $H_{x}$ for the stabilizer of $\tilde x$.

We introduce a term for this restricted class of orbifolds:

\begin{definition}
    A \emph{bifold} will mean a $3$- or $4$-dimensional orbifold whose
    local models belong to one of the particular types described above.
\end{definition}

 Our bifolds will be equipped
with an orbifold Riemannian metric, which we can regard as a smooth
metric $\check g$ on the complement of the singular set with the
property that the pull-back of $\check g$ via any orbifold chart
$\phi$ extends to a smooth metric on the domain $\tilde U$ of $\phi$.
In the $3$-dimensional case, the singular set is an embedded
web and the total space of the bifold is topologically a manifold. The
Riemannian metric will have cone-angle $\pi$ along each edge of the
graph. In the $4$-dimensional case, the singular set is an embedded
$2$-dimensional complex $\Sigma$ with the same combinatorial structure
as a foam, and cone-angle $\pi$ along the facets of $\Sigma$.

\subsection{Bifolds from embedded webs and foams}

Given a smooth Riemannian $3$-manifold $Y$ containing a web
$K$, smoothly embedded in the sense described in the introduction, 
we can construct a Riemannian
bifold $\check Y$ with a copy of $K$ as its singular set by
modifying the Riemannian metric in a neighborhood of $K$ to introduce
the required cone angles. To do
this, we first adjust  $K$ so that, at each vertex the incident edges
arrive with tangent directions which are coplanar and at equal angles
$2\pi /3$. This can be done in a standard way, because the space of
triples of distinct points in $S^{2}$ deformation-retracts to the
subspace of planar equilateral triangles. Once this is done, the
neighborhoods of the vertices are standard and we can use the smooth
Riemannian metric $g$ on $Y$ to introduce coordinates in which to
modify the Riemannian metric, making it isometric to $\R^{3}/V_{4}$ in
a neighborhood of the singular set near the vertex. Along each edge,
we continue the modification, using exponential normal coordinates to
modify the metric and introduce a cone-angle $\pi$. 

In dimension $4$, we can similarly start with a smooth Riemannian
$4$-manifold $(X, g)$ and an embedded foam $\Sigma$, and modify
the metric near $\Sigma$ to produce a bifold $(\check X, \check
g)$. Again, the first step is to modify $\Sigma$ near the tetrahedral
points and seams so
that, at each point $x$ on a seam, the tangent planes of the three incident branches
of $\Sigma$  lie in a single $3$-dimensional subspace of $T_{x}\R^{4}$
and are separated by angles $2\pi/3$. The structure of $\Sigma$ is
then locally standard, and we can use normal coordinates as before.

 We will often pass freely
from a pair $(Y,K)$ or $(X,\Sigma)$, consisting of a Riemannian $3$- or $4$-manifold
with an embedded web or foam, to the corresponding
Riemannian bifold $\check Y$ or $\check X$.

\subsection{Orbifold connections}
\label{sec:connections}

Let $\check Y$ be a closed, connected $3$-dimensional bifold, and $K\subset \check Y$
the singular part (which we may regard as a web). A $C^{\infty}$ orbifold $\SO(3)$-connection over
$\check Y$ means an oriented $\R^{3}$-vector bundle $E$ over $\check Y
\sminus K$
with an $\SO(3)$ connection $A$ having the property that the pull-back
of $(E,A)$ via any orbifold chart $\phi: \tilde U \to U$ extends to a
smooth pair $(\tilde E, \tilde A)$ on $\tilde U$. (It may be that the
bundle $E$ cannot be extended to all of $\check Y$ as a topological
vector bundle.)

If $U$ is the codomain of an orbifold chart around $x$ and $\tilde x$
is a preimage in $\tilde U$, then the stabilizer $H_{x}$ acts on the
fiber $\tilde E_{\tilde x}$. We will require that this action is
\emph{non-trivial} at all points where $H_{x}$ has order $2$, i.e.~at
all points on the edges of the web $K$. 
We introduce a name for orbifold
connections of this type:

\begin{definition}
    A \emph{bifold connection} on a bifold $\check Y$ will mean an
    $\SO(3)$ connection for which all the order-$2$ stabilizer groups
    $H_{x}$ act non-trivially on the corresponding fibers $\tilde
    E_{\tilde x}$.
\end{definition}

This determines the action of $H_{x}$ on $\tilde
E_{\tilde x} \cong \R^{3}$ also at the vertices of $K$, up to conjugacy: we  have a standard
action of $V_{4}$ on the fiber.
 If $\gamma$ is a
standard meridional loop of diameter $\epsilon$ about an edge of $K$,
then the holonomy of $A$ will approach an element of order $2$ in
$\SO(3)$ as $\epsilon$ goes to zero. 

Given an orbifold connection $(E,A)$, we can use the connection $A$
and the Levi-Civita connection of the Riemannian metric $\check g$ to
define Sobolov norms on the spaces of sections of $E$ and its
associated vector bundles. We fix a sufficiently large Sobolev
exponent $l$ ($l\ge 3$ suffices) and we consider orbifold
$\SO(3)$-connections $A$ of class $L^{2}_{l}$. By this we mean that
there exists a $C^{\infty}$ bifold connection $(E, A_{0})$ so that
$A$ can be written as $A_{0} + a$, where $a$ belongs to the Sobolev
space
\[
          L^{2}_{l, A_{0}}(\check Y\sminus K, \Lambda^{1} Y \otimes E).
\]
Under these circumstances, the Sobolev norms $L^{2}_{j,A}$ and
$L^{2}_{j,A_{0}}$ are equivalent for $j\le l+1$.

If we have two $\SO(3)$ bifold connections of class $L^{2}_{l}$, say $(E,A)$
and $(E', A')$, then an isomorphism between them is a bundle map
$\tau : E\to E'$ over $\check Y \sminus K$ of class $L^{2}_{l+1}$
such that $\tau^{*}(A') = A$. The group $\Gamma_{E,A}$ of
automorphisms of $(E, A)$ can be identified as usual with the group of
parallel sections of the associated bundle with fiber $\SO(3)$. It is
isomorphic to either the trivial group, the group of order $2$, 
the group $V_{4}$, the group
$O(2)$ or all of $\SO(3)$. If $K$ has at least one vertex, then
$\Gamma_{E,A}$ is no larger than $V_{4}$.

We write $\bonf_{l}(\check Y)$ for the space of all isomorphism classes of
bifold connections of class $L^{2}_{l}$. By the usual application of
a Coulomb slice, the neighborhood of an isomorphism class 
$[E,A]$ in $\bonf_{l}(\check Y)$ can be given the structure $S/\Gamma_{E,A}$,
where $S$ is a neighborhood of $0$ in a Banach space and the group
acts linearly. In particular, if $K$ has at least one vertex, then
$\bonf_{l}(\check Y)$ is a Banach orbifold.

All the content of this subsection carries over without essential
change to the case of a $4$-dimensional bifold $\check X$, where we also
have a space $\bonf_{l}(\check X)$ of isomorphism classes of 
$\SO(3)$ bifold connections of class $L^{2}_{l}$. We still require
that the order-$2$ stabilizers $H_{x}$ act non-trivially on the fibers,
and this condition determines the action of the order-$8$ stabilizers
$V_{8}$ at the tetrahedral points, up to conjugacy.

In the
$4$-dimensional case (and sometimes also in the $3$-dimensional case,
if $Y\setminus K$ has non trivial homology) the space
$\bonf_{l}(\check X)$ or $\bonf_{l}(\check Y)$ will 
have than one component, because the $\SO(3)$ bifold connections
belong to different topological types.

\subsection{Marked connections} 

Because $\SO(3)$ bifold connections may have non-trivial automorphism groups,
we introduce marked connections. By \emph{marking data} $\mu$ on a
$3$-dimensional bifold $\check Y$ we mean a pair $(U_{\mu}, E_{\mu})$, where:
\begin{itemize}
\item  $U_{\mu}\subset \check Y$ is any subset; and
\item  $E_{\mu} \to U_{\mu}\sminus K$ is an $\SO(3)$ bundle (where
    $K$ denotes the singular set).
\end{itemize}
 An $\SO(3)$ connection
\emph{marked by $\mu$} will mean a triple
$(E,A,\sigma)$, where 
\begin{itemize}
\item $(E,A)$ is a bifold $\SO(3)$ connection as
before; and
\item $\sigma : E_{\mu} \to E|_{U_{\mu}\sminus K}$ is an
isomorphism of $\SO(3)$ bundles. 
\end{itemize}
An \emph{isomorphism} between marked connections $(E,A,\sigma)$ and $(E',
A', \sigma')$ is a bundle map
$\tau:E' \to E$ such that
\begin{itemize}
\item $\tau$ respects the connections, so $\tau^{*}(A) = A'$; and
\item the automorphism 
\[
    \sigma^{-1}\tau\sigma' : E_{\mu} \to E_{\mu}
\]
lifts to the ``determinant-$1$ gauge group''. That is, when we regard
$\sigma^{-1}\tau\sigma'$ as a section of the bundle $\SO(3)_{E_{\mu}}$
with fiber $\SO(3)$ (associated to $E_{\mu}$ by the adjoint action),
we require that this section lift to a section of the bundle
$\SU(2)_{E_{\mu}}$ associated to the adjoint action of $\SO(3)$ on $\SU(2)=\Spin(3)$.
\end{itemize}

We will say that the marking data $\mu$ is \emph{strong} if the
automorphism group of every $\mu$-marked bifold connection on
$\check Y$ is trivial.

\begin{lemma}\label{lem:marking}
    The marking data $\mu$ on a $3$-dimensional bifold $\check Y$ 
    is strong if either of the following hold.
    \begin{enumerate}
    \item \label{it:vetex-strong} The set $U_{\mu}$ contains a neighborhood of a vertex of
        $K$.
    \item\label{it:Hopf-strong} The set $U_{\mu}$ contains a $3$-ball $B$ which meets $K$ in a
        Hopf link $H$ contained in the interior of $B$, and
        $w_{2}(E_{\mu})|_{B}$ is Poincar\'e dual to an arc joining the
        two components of $H$.
    \end{enumerate}
\end{lemma}

\begin{proof}
    From a $\mu$-marked bundle $(E,A,\tau)$, we obtain by pull-back a
    connection $A_{\mu}$ in the bundle $E_{\mu}$. The automorphisms of
    $(E,A,\tau)$ are a subgroup of the group of $A_{\mu}$-parallel
    sections of the associated bundle $\SU(2)_{\mu}$ with fiber
    $\SU(2)$.

     Pick a point $y$ in $U_{\mu}\sminus K$ and consider the holonomy
     group of $A_{\mu}$ at this point, as a subgroup of $\SO(3)$ (the
     automorphisms of the fiber of $E_{\mu}$ at $p$. In case
     \ref{it:vetex-strong}, the closure of the holonomy group contains
     $V_{4}$. We can find a subset of $U_{\mu}$ with trivial second
     homology over which the closure of the holonomy group still
     contains $V_{4}$. Over this subset, we can lift $E_{\mu}$ to an
     $\SU(2)$ bundle and lift $A_{\mu}$ to an $\SU(2)$ connection
     $\hat A_{\mu}$. The
     holonomy group of $\hat A_{\mu}$ contains lifts of the generators
     of $V_{4}$, which implies that its commutant in $\SU(2)$ is
     trivial.

     In case \ref{it:Hopf-strong}, pick one component of the Hopf link
     and a point $q$ on this component. Using loops based at $p$
     running around small meridional loops near $q$, we see that
     the closure of the holonomy group contains an element $h_{q}$ of
     order $2$ in $\SO(3)$. As $q$ varies, this element of order $2$
     cannot be constant, for otherwise $w_{2}$ would be zero. So the
     holonomy group contains two distinct involutions $h_{q}$ and
     $h_{q'}$. We can now lift to $\SU(2)$ as in the previous case to
     see that the commutant in $\SU(2)$ is trivial.
\end{proof}

We write $\bonf_{l}(\check Y; \mu)$ for the space of isomorphism
classes of $\mu$-marked bifold connections of class $L^{2}_{l}$ on
the bifold $\check Y$. If the marking data is strong, then
$\bonf_{l}(\check Y; \mu)$ is a Banach manifold modeled locally on
the Coulomb slices. There is a map that forgets the marking,
\[
      \bonf_{l}(\check Y; \mu) \to \bonf_{l}(\check Y).
\]
Its image consists of a union of some
connected components of $\bonf_{l}(\check Y)$, namely the components
comprised of isomorphism classes of connections $(E,A)$ for which the
restriction of $E$ to $U_{\mu}\sminus K$ is isomorphic to $E_{\mu}$.

Slightly more generally, we can consider the case that we have two
different marking data, $\mu$ and $\mu'$ with $U_{\mu} \subset
U_{\mu'}$ and $E_{\mu} = E_{\mu'}|_{U_{\mu}\sminus K}$. In this
case, there is a forgetful map
\begin{equation}\label{eq:covering}
 r : \bonf_{l}(\check Y; \mu') \to \bonf_{l}(\check Y; \mu).
\end{equation}

\begin{lemma}\label{lem:covering}
    The image of the map \eqref{eq:covering} consists of a union of
    connected 
    components of $\bonf_{l}(\check Y; \mu)$. Over these components, the map
    $r$ is a covering map with covering group an elementary
    abelian $2$-group, namely the group which is the kernel of the
    restriction map
    \[
         H^{1}(U_{\mu'}\sminus K ; \F) \to
         H^{1}(U_{\mu}\sminus K ; \F).
    \]
\end{lemma}

\begin{proof}
   The image consists of isomorphism classes of triples $(E,A,\tau)$
   for which the map $\tau: E_{\mu} \to E$ can be extended to some $\tau' :
   E_{\mu'} \to E$ over $U_{\mu'}\sminus K$. This set
   is open and closed in $\bonf_{l}(\check Y; \mu)$, so it is a union
   of components of this Banach manifold.

   If $(E,A,\tau')$ and $(E,A,\tau'')$ are two elements of a fiber of $r$, then
   the difference of $\tau'$ and $\tau''$ is an automorphism
   $\sigma : E_{\mu'} \to E_{\mu'}$ whose restriction to $E_{\mu}$
   lifts to determinant $1$. The fiber consists of such automorphisms
   $\sigma$ modulo those that lift to determinant $1$ on the whole of
   $E_{\mu'}$. These form a group isomorphic to the kernel of the
   restriction map
    \[
         H^{1}(U_{\mu'}\sminus K ; \F) \to
         H^{1}(U_{\mu}\sminus K ; \F).
    \]
   Over the components that form its image, the map $r$ is a
   quotient map for this elementary abelian $2$-group.
\end{proof}

All of the above definitions can be formulated
in the $4$-dimensional case. So for a compact, connected
4-dimensional bifold $\check X$ we can talk about marking data $\mu$
in the same way. We have a space $\bonf_{l}(\check X; \mu)$
parametrizing isomorphism classes of $\mu$-marked $\SO(3)$
connections, and this space is a Banach manifold if $\mu$ is strong.

\subsection{ASD connections and the index formula}

Let $\check X$ be a closed, oriented Riemannian bifold of
dimension $4$, and let $(X,\Sigma)$ be the associated pair. We can, as
usual, consider the anti-self-duality condition,
\[
        F^{+}_{A}=0
\]
in the bifold setting. We write
\[
           M(\check X) \subset \bonf_{l}(\check X)
\]
for the moduli space of anti-self-dual bifold connections. It is independent
of the choice of Sobolev exponent $l\ge 3$. We can also introduce
marking data $\mu$, and consider the moduli space
\[
         M(\check X ; \mu) \subset \bonf_{l}(\check X ; \mu).
\]

We write $\kappa(E,A)$ for the orbifold version of the characteristic
class $-(1/4) p_{1}(E)$, which we can compute as the Chern-Weil integral,
       \begin{equation}\label{eq:kappa}
                 \kappa( E,A) = \frac{1}{32\pi^{2}}\int_{\check X} \tr
                     (F_{ A} \wedge F_{ A}) .
\end{equation}
Our normalization means that $\kappa$ coincides with $c_{2}(\tilde
E)[X]$ if $E$ lifts to an $\SU(2)$ bundle $\tilde E$ and $\Sigma$ is
absent. In general $\kappa$ may be non-integral. We refer to $\kappa$
as the (topological) \emph{action}.

The formal dimension of
the moduli space in the neighborhood of $[E,A]$ 
is given by the index of the linearized equations
with gauge fixing, which we write as
 \begin{equation}\label{eq:d-def}
      d( E, A) = \ind (-d^{*}_{A} \oplus d^{+}_{A}).
\end{equation}
This definition does not require $A$ to be anti-self-dual and defines
a function which is constant on the components of $\bonf_{l}(\check
X)$. If the marking is strong, then for a generic metric $\check g$ on
$\check X$, the marked moduli space $M(\check X ; \mu)$ is smooth and
of dimension $d(E,A)$ in the neighborhood of any anti-self-dual
connection $(E,A)$, unless the connection is flat.

We shall give a formula for the formal dimension $d(E,A)$ in terms of
$\kappa$ and the topology of $(X,\Sigma)$. To do so, we must digress
to say more about foams $\Sigma$.

We shall define a self-intersection number $\Sigma\cdot\Sigma$ which
coincides with the usual self-intersection number of a (not
necessarily orientable) surface in $X$ if there are no seams. We can
regard $\Sigma$ as the image of an immersion of a surface with corners,
$i:\Sigma^{+}\to X$, which is injective except at $\partial \Sigma^{+}$,
which is mapped to the seams as a $3$-fold covering. The corners of
$\Sigma^{+}$ are mapped to the tetrahedral points of $\Sigma$. The 
tangent spaces to the three branches of $\Sigma$ span a
3-dimensional subspace $V_{q}$ at each point $q$ on the seams and vertices. This
determines a $3$-dimensional subbundle
\[
            V \subset i^{*}(TX)|_{\partial\Sigma^{+}}.
\]
Since $V$ contains the directions tangent to $i(\Sigma^{+})$, it
determines a $1$-dimensional subbundle $W$ of the normal bundle $N\to\Sigma^{+}$ to the
immersion:
\[
           W \subset N|_{\partial\Sigma^{+}}.
\]
Although the $2$-plane bundle $N \to \Sigma^{+}$ is not necessarily
orientable, it has a well-defined ``square'', $N^{[2]}$. (Topologically, this is
the bundle obtained by identifying $n$ with $-n$ everywhere.) The
orientation bundles of both $N$ and $N^{[2]}$ are canonically
identified with the orientation bundle of $\Sigma^{+}$, using the
orientation of $X$.
The subbundle $W$ determines a section $w$ of
$N^{[2]}|_{\partial\Sigma^{+}}$, and there is a relative Euler number,
\begin{equation}\label{eq:euler-number}
         e(N^{[2]}, w)[\Sigma^{+},\partial\Sigma^{+}]
\end{equation}
obtained from the pairing in (co)homology with coefficients in the
orientation bundle.

\begin{definition}\label{def:self-intersection}
    We define $\Sigma\cdot\Sigma$ to be half of the relative Euler
    number \eqref{eq:euler-number}.
\end{definition}
 
If $V$ is orientable, then so is
$W$. In this case, we may choose a trivialization of $W$ and
obtain a section of $N$ along $\partial \Sigma^{+}$. This removes the
need to pass to the square of $N$ and also shows that
$\Sigma\cdot\Sigma$ is an integer. If $V$ is non-orientable, then
$\Sigma\cdot\Sigma$ may be a half-integer.

With this definition out of the way, we can state the index theorem
for the formal dimension $d(E,A)$.

\begin{proposition}\label{prop:dimension}
    The index $d(E,A)$ is given by
\begin{equation}\label{eq:dimension}
        d(E,A) = 8 \kappa - 3 (1 -  b^{1}(X) + b^{+}(X)) +
        \chi(\Sigma) + \frac{1}{2} \Sigma\cdot\Sigma - \frac{1}{2}|\tau|,
\end{equation}
where $\chi(\Sigma)$ is the ordinary Euler number of the $2$-dimensional
complex underlying the foam $\Sigma$, the integer $|\tau|$ is the
number of tetrahedral points of $\Sigma$, and $\Sigma\cdot\Sigma$ is
the self-intersection number from
Definition~\ref{def:self-intersection}.
\end{proposition}

\begin{proof}
    In the case that $\Sigma$ has no seams, the result coincides with
    the index formula in \cite{KM-unknot}. (If $\Sigma$ is also
    orientable then the formula appears in \cite{KM-gtes-I}.) 
    So the result is known in this case.
    
    Consider next the case that $\Sigma$ has no tetrahedral points.
    The seams of $\Sigma$ form circles, and the neighborhood of each
    circle has one of three possible types: the three branches of the
    foam are permuted by the monodromy around the circle, and
    permutation may be trivial, an involution of two of the branches,
    or a cyclic permutation of the three.
    So (in the absence of tetrahedral points) an excision
    argument shows that it is sufficient to verify the index
    formula for just three examples, one containing seams of each of
    three types. A standard model for a neighborhood of a seam $s
    \cong S^{1}$ is a foam in $S^{1}\times B^{3}$. By
    doubling this standard model, we obtain a foam in $X
    = S^{1}\times S^{3}$ with two seams of the same type. If the
    branches are permuted non-trivially by the monodromy, 
    we can now pass to a $2$-fold or $3$-fold
    cyclic cover of $X$, and so reduce to the case of seams with
    trivial monodromy. (Both
    sides of the formula in the Proposition are multiplicative under
    finite covers.) When the monodromy permutation of the branches is
    trivial, 
    the pair $(X,\Sigma)$ that we have described
   carries a circle action along the $S^{1}$ factor and there is a
   flat bundle $(E,A)$ with holonomy group $V_{4}$, also acted on by
   $S^{1}$.  The circle action on the bundle means that the index
   $d(E,A)$ is zero, as is the right-hand side. This completes the
   proof in the absence of tetrahedral points.

   Consider finally the case that $\Sigma$ has tetrahedral points. By
   taking two copies if $(X,\Sigma)$ if necessary, we may assume that
   the number of tetrahedral points is even. By an excision argument,
   it is then enough to verify the formula in the case of a standard
   bifold $\check X$ with two tetrahedral points, namely the quotient
   of $S^{4}$ by the action of \[ V_{8}\subset \SO(4) \subset \SO(5)
   . \]
   In this model case, there is again a flat bifold connection with
   holonomy $V_{4}$, obtained as a global quotient of the trivial
   bundle on $S^{4}$. For this example we have $d(E,A)=0$, as one can
   see by looking at corresponding operator on the cover $S^{4}$. The
   pair $(X,\Sigma)$ is topologically a $4$-sphere containing a foam
   which is the suspension of the $1$-skeleton of a tetrahedron. The
   Euler number $\chi(\Sigma)$ is $4$, and the terms $b_{1}(X)$,
   $b^{+}(X)$ and $\Sigma\cdot\Sigma$ in \eqref{eq:dimension} are
   zero. So the formula is verified in this case:
   \[
               0 = -3 + 4 - \frac{1}{2} \times 2.
   \]
   This completes the proof.
\end{proof}

From the formula, one can see that $\kappa(E,A)$ belongs to
$(1/32)\Z$, because $2 ( \Sigma\cdot\Sigma )$ is an integer.

\subsection{Bubbling}

Uhlenbeck's theorem \cite{Uhlenbeck} applies in the orbifold
setting. So if $[E_{i}, A_{i}] \in M(\check X)$ 
is a sequence of anti-self-dual
bifold connections on the closed oriented bifold $\check
X$ and if their actions $\kappa(E_{i},A_{i})$ are bounded, 
then there exists a subsequence $\{i'\} \subset \{i\}$, 
an anti-self-dual connection $(E,A)$, a finite
set of points $Z\subset\ \check X$ and isomorphisms 
\[
    \tau_{i'} : (E,A)|_{\check X\sminus (\Sigma \cup Z)} \to 
                       (E_{i'},A_{i'})|_{\check X\sminus (\Sigma
                         \cup Z)}                     
\]
such that the connections $(E, \tau^{*}(A_{i'})$ converge on compact
subsets of $\check X \sminus (K\cup Z)$ to $(E,A)$ in the
$C^{\infty}$ topology. Furthermore,
\begin{equation}\label{eq:energy-loss}
       \kappa(E,A) \le \liminf \kappa(E_{i'}, A_{i'}),
\end{equation}
and if equality holds then $Z$ can be taken to be empty and the
convergence is strong (i.e. the convergence is in the topology of
$M(\check X)$).

When the inequality is strict, the difference is accounted for by
bubbling at the points $Z$. The difference is therefore equal to
the sum of the actions of some finite-action solutions on bifold
quotients of $\R^{4}$ by either the trivial group, the group of order
$2$, the group $V_{4}$, or the group $V_{8}$. 
Any solution on such a quotient of $\R^{4}$ pulls back
to a solution on $\R^{4}$ with integer action; so on the bifold, the
action of any bubble lies in $(1/8)\Z$. If there are no tetrahedral
points, then the action lies in $(1/4)\Z$. Thus we have:

\begin{lemma}\label{lem:Uhlenbeck}
    In the Uhlenbeck limit, either the action inequality
    \eqref{eq:energy-loss} is an equality and the convergence is
    strong, or the difference is at least $1/8$. If the difference is
    exactly $1/8$, then the set of bubble-points $Z$ consists of a
    single point which is a tetrahedral point of $\Sigma$.

    If there are no tetrahedral points, and the energy-loss is
    non-zero, then the difference is at least $1/4$. If the difference is
    exactly $1/4$, then the set of bubble-points $Z$ consists of a
    single point which lies on a seam of $\Sigma$.
\end{lemma}

Suppose that $\Sigma$ has no tetrahedral points. 
From the dimension formula, Proposition~\ref{prop:dimension}, we see
that when bubbling occurs, 
the dimension of the moduli space drops by
at least $2$. Let us write $M_{d}(\check X ; \mu)$ for the moduli
space of $\mu$-marked instantons $(E,A,\tau)$ whose action $\kappa$ is such that
the formal dimension at $[E,A,\tau]$ is $d$. Suppose 
that the marking data $\mu$ is strong, that $M_{0}(\check X; \mu)$
is regular, and that all moduli space of negative formal dimension are
empty. Uhlenbeck's theorem and the Lemma tell us that $M_{0}(\check X; \mu)$ is
a finite set and that $M_{2}(\check X; \mu)$ has a compactification,
\[
             \bar M_{2} = M_{2} \cup s \times M_{0},
\]
where $s$ is the union of the seams of $\Sigma$.

\begin{proposition}\label{prop:cone-on-V}
    In the above situation, when $\Sigma$ has no tetrahedral point, 
    for each $q\in s$ and $\alpha\in M_{0}$,
    an open neighborhood of $(q,\alpha)$ in $\bar M_{2}$ is
    homeomorphic to 
     \[
           N(q) \times T
      \]
     where $N(q) \subset s$ is a neighborhood of $q$ in the seam $s$, and $T$ is a cone on
     $4$ points (i.e. the union of $4$ half-open intervals $[0,1)$
     with their endpoints ${0}$ identified). In particular, if $M_{2}$
     is regular, then $\bar M_{2}$ is homeomorphic to an identification space of
     a compact $2$-manifold with boundary, $S$, by an identification
     which maps $\partial S$ to $s$ by a $4$-sheeted covering.
\end{proposition}

\begin{proof}
  This proposition is an orbifold adaptation of the more familiar
  non-orbifold version \cite{Donaldson-connections} which describes
  the neighborhood of the stratum $X\times M_{0}$ in the Uhlenbeck
  compactification of $M_{8}$. In the non-orbifold version, the local
  model near $(x, [A])$ is $N(x) \times \mathrm{Cone}(\SO(3))$, where
 $N(x)$ is a $4$-dimensional neighborhood of $x\in X$ and $\SO(3)$ arises
 as the gluing parameter. The above proposition is similar, except
 that the cone on $\SO(3)$ has been replaced by $T$, which is a cone
 on the Klein $4$-group $V_{4}$. 

  Let $\R^{4}$ be Euclidean $4$-space and let $\check\R^{4}$ be its
  bifold quotient by $V_{4}$ acting on the last three coordinates. Let
  $s\subset\R^{4}$ be the seam, i.e. the line $\R\times 0$ fixed by
  $V_{4}$. The standard $1$-instanton moduli space on $\R^{4}$ is the
  $5$-dimensional space with center and scale coordinates: 
 \[
  M_{5}(\R^{4}) = \R^{4} \times \R^{+}.
  \]
  The moduli space of solutions with $\kappa=1/4$ on $\check \R^{4}$
  is the set of fixed points of the $V_{4}$ action on $M_{5}$. This
  space is $2$-dimensional, with a center coordinate constrained to
  lie on $s$ and scale coordinate as before:
  \[
    M_{2}(\check \R^{4}) =s \times \R^{+}.
  \]
   On $\R^{4}$, we may pass to the $8$-dimensional framed moduli space
   $\tilde M_{8}(\R^{4})$ parametrizing isomorphism classes of
   instantons $A$ together with an identification of the limiting
   flat connection on the sphere at infinity $S^{3}_{\infty}$ with the
   trivial connection in $\R^{3}\times S^{3}_{\infty}$. This is the
   product,
\[
    \tilde M_{8}(\R^{4}) = \R^{4} \times \R^{+} \times \SO(3).
\]
   In the bifold case, the framing data is an identification of the
   flat bifold connection on $S^{3}_{\infty}/V_{4}$ with a standard
   bifold connection whose holonomy group is $V_{4}$. The choice of
   identification is now the commutant of $V_{4}$ in $\SO(3)$, namely
   $V_{4}$ itself. So the framed moduli space is
 \[
      \tilde M_{2}(\check \R^{4}) =s \times \R^{+} \times V_{4}.
  \]
  With this moduli space understood, the usual proof from the
  non-orbifold case carries over without change.
\end{proof}

When tetrahedral points are present, bubbling is a codimension-$1$
phenomenon, meaning that even $1$-dimensional moduli spaces may be
non-compact. 
We have the following counterpart of the previous
proposition. Consider a $1$-dimensional moduli space $M_{1}$ of bifold
connections on $(X,\Sigma)$. Its Uhlenbeck compactification is the space
\[
             \bar M_{1} = M_{1} \cup \tau \times M_{0},
\]
where $\tau$ is the finite set of tetrahedral points of $\Sigma$.

\begin{proposition}\label{prop:tetrahedral-cone-on-V}
    In the above situation, 
    for each  tetrahedral point $q$ and $\alpha\in M_{0}$,
    an open neighborhood of $(q,\alpha)$ in $\bar M_{1}$ is
    homeomorphic to $T$, i.e.~again the union of $4$ half-open intervals $[0,1)$
     with their endpoints ${0}$ identified. In particular, if $M_{1}$
     is regular, then $\bar M_{1}$ is homeomorphic to an identification space of
     a compact $1$-manifold with boundary, by an identification
     which identifies the boundary points in sets of four.
\end{proposition}

\begin{proof}
    The proof is much the same as the proof of the previous proposition.
\end{proof}

\subsection{The Chern-Simons functional}

Let $\check Y$ be a closed, oriented, $3$-dimensional bifold
and let $\mu$ be strong marking data on $\check Y$. The tangent space
to $\bonf_{l}(\check Y ; \mu)$ at a point represented by a connection
$(E,A)$ is isomorphic to the kernel of $d^{*}_{A}$ acting on
$L^{2}_{l}(\check Y; \Lambda^{1}\check Y \otimes E)$. As usual, there
is a locally-defined smooth function, the Chern-Simons function, 
on $\bonf_{l}(\check Y; \mu)$
whose formal $L^{2}$ gradient is $*F_{A}$. On the universal cover of
each component of
$\bonf_{l}(\check Y; \mu)$, the Chern-Simons function is single-valued
and well-defined up to the addition of a constant. 

If we have a closed loop $z$ in $\bonf_{l}(\check Y; \mu)$, then we
have a $1$-parameter family of marked connections, \[
 \zeta(t) = (E(t), A(t),
\tau(t))
\] 
parametrized by $[0,1]$ together with an isomorphism $\sigma$ from
$\zeta(0)$ to $\zeta(1)$. The isomorphism
$\sigma$ is determined uniquely by the data, because the marking is
strong, which means that $\zeta(0)$ has no automorphisms. Identifying
the two ends, we have a connection $(E_{z}, A_{z})$ over the bifold $S^{1}\times
\check Y$. 

This $4$-dimensional connection has an index $d(E_{z},
A_{z})$ and an action $\kappa(E_{z}, A_{z})$. We can interpret both
as usual, in terms of the Chern-Simons function. Up to a constant factor,
$\kappa(E_{z}, A_{z})$ is equal to minus the change in $\CS$ along the
path $\zeta(t)$:
\[
               \kappa(E_{z}, A_{z}) = \frac{1}{32 \pi^{2}}
               (\CS(\zeta(0)) - \CS(\zeta(1))).
\]
The index $d(E_{z}, A_{z})$ is equal to the spectral flow
of a family of elliptic operators related to the formal Hessian
of $\CS$, along the path $\zeta$: that is,
\[
       d(E_{z}, A_{z}) = \sflow_{\zeta}(D_{A})
\]
where $D_{A}$ is the ``extended Hessian'' operator on $\check Y$,
\[
          D_{A} = \begin{bmatrix} 
                0 & - d^{*}_{A} \\ -d_{A} & *d_{A} \end{bmatrix}
\]
acting on $L^{2}_{l}$ sections of $(\Lambda^{0} \oplus
\Lambda^{1})\otimes E$.

An important consequence of the index formula,
Proposition~\ref{prop:dimension}, is the proportionality of these of
two quantities. For any path $\zeta$ corresponding to a closed loop $z$ in
$\bonf_{l}(\check Y; \mu)$, we have
\begin{equation}
    \label{eq:monotone}
                  \CS(\zeta(0)) - \CS(\zeta(1)) = 4 \pi^{2}\sflow_{\zeta}(D_{A}).
\end{equation}
Since the spectral flow is always an integer, we also see that $\CS$
defines a single-valued function with values in the circle $\R/(4\pi^{2}\Z)$.

\subsection{Representation varieties}\label{subsec:RepVarieties}

The critical points of $\CS : \bonf_{l}(\check Y; \mu) \to
\R/(4\pi^{2}\Z)$ are the isomorphism classes of $\mu$-marked connections
$(E, A, \tau)$ for which $A$ is flat. We will refer to the critical
point set as the \emph{representation variety} and write
\[
    \Rep(\check Y; \mu) \subset \bonf_{l}(\check Y; \mu).
\]
We may also write this as $\Rep(Y,K ;\mu)$.
In the absence of the marking, this would coincide with a space of
homomorphisms from the orbifold fundamental group  $\pi_{1}(\check Y ,
y_{0})$ to $\SO(3)$, modulo the action of $\SO(3)$ by conjugation.
In terms of the pair $(Y, K)$, we are looking at conjugacy classes of
homomorphisms
\[
         \rho : \pi_{1}(Y\sminus K, y_{0}) \to \SO(3)
\]
satisfying the constraint that, for each edge $e$ and any
representative $m_{e}$ for the conjugacy class of the meridian of $e$,
the element $\rho(m_{e})$ has order $2$ in $\SO(3)$.

We can adapt this description to incorporate the marking data
$\mu=(U_{\mu},E_{\mu})$ as follows. Let $w\subset U_{\mu}$ be a closed
1-dimensional submanifold dual to $w_{2}(E_{\mu})$. One
$U_{\mu}\sminus w$, we can lift $E_{\mu}$ to an $\SU(2)$ bundle
$\tilde E_{\mu}$, and we fix an isomorphism $p:\mathrm{ad}(\tilde
E_{\mu}) \to E_{\mu}|_{U_{\mu}\sminus w}$. Via $p$, any flat $\SO(3)$
connection $A_{\mu}$ on $E_{\mu}$ gives rise to a flat $\SU(2)$ connection
$\tilde A_{\mu}$ on $\tilde E_{\mu}$ with the property that, for each
component $v\subset w$, the holonomy the meridian $m_{v}$
of $v$ is $-1 \in \SU(2)$. If we pick a basepoint $y_{0}$ in
$U_{\mu}\sminus (K\cup w)$, then we have the following
description. The representation variety $\Rep(\check Y; \mu)$
corresponds to $\SO(3)$-conjugacy classes of pairs $(\rho,
\tilde\rho_{\mu})$, where
\begin{enumerate}
\item $\rho :\pi_{1}(Y\sminus K, y_{0}) \to \SO(3)$ is a
    homomorphism with the property that $\rho(m_{e})$ has order $2$,
    for every edge $e$ of $K$;
\item $\tilde\rho_{\mu} : \pi_{1} (U_{\mu} \sminus (K\cup w),
    y_{0})\to \SU(2)$ is a  homomorphism with the property that
    $\tilde\rho_{\mu}(m_{v})$ is $-1$, for every component $v$ of $w$;
\item the diagram formed from $\rho$, $\tilde\rho_{\mu}$, the adjoint
    homomorphism $\SU(2)\to\SO(3)$ and the inclusion 
     \[
                 U_{\mu} \sminus (K\cup w) \to Y\sminus K
     \]
     commutes.
\end{enumerate}
Note that these conditions imply that $\tilde\rho_{\mu}(m_{e})$ has
order $4$ in $\SU(2)$, for every meridian of $K$ contained in $U_{\mu}$.
We give two basic examples.

\subparagraph{Example: the theta graph.} Take $Y=S^{3}$ and $K=\Theta$ an
unknotted theta graph: two vertices joined by three coplanar arcs. 
Take $U_{\mu}$ to be a ball containing $\Theta$ and
take $E_{\mu}$ to be the trivial $\SO(3)$
bundle. Lemma~\ref{lem:marking} tells us that $\mu$ is strong. Using the above
description, we see that $\Rep(S^{3},\Theta; \mu)$ is the space of
conjugacy classes of homomorphism $\tilde\rho : \pi_{1}(B^{3}\sminus
\Theta) \to \SU(2)$ mapping meridians to elements of order $4$. Up to
conjugacy, there is exactly one such homomorphism. Its image is the
quaternion group of order $8$. So 
$\Rep(S^{3},\Theta; \mu)$ consists of a single point.

\subparagraph{Example: a Hopf link.}
Take $Y=S^{3}$ again and take $K=H$, a Hopf link. Take $U_{\mu}$ to
be a ball containing $H$ and take $w_{2}(E_{\mu})$ to be dual to an
arc $w$ joining the two components of $H$. This example is similar to
the previous one, and was exploited in \cite{KM-unknot}. The
representation variety $\Rep(S^{3}, H; \mu)$ is again a single point,
corresponding to a unique homomorphism with image the quaternion
group in $\SU(2)$.

As an application of these basic examples, we have the following
observation.

\begin{lemma}\label{lem:Hopf-link-Rep}
    Let $K\subset Y$ be an embedded web. Let $B$ be a ball
    in $Y$ disjoint from $K$, and let $K^{\sharp}$ be the disjoint
    union of $K$ with either a theta graph $\Theta\subset B$ or a Hopf
    link $H\subset B$. Let $\mu$ be marking data for $(Y,K^{\sharp})$, 
    with $U_{\mu} =
    B$. In the Hopf link case, take $w_{2}(E_{\mu})$ to be dual to an
    arc joining the two components. Then the representation variety
    \[
                \Rep(Y,K^{\sharp} ; \mu)
      \]
     can be identified with the space of homomorphisms
     \[
                  \rho : \pi_{1} (Y\sminus K, y_{0}) \to \SO(3)
      \]
     which map meridians of $K$ to elements of order $2$.
\end{lemma}

Note that the space described in the conclusion of the lemma is not
the quotient by the action of $\SO(3)$ by conjugation, but simply the
space of homomorphisms. We introduce abbreviated terminology for the
versions of these representation varieties that we use most:

\begin{definition}\label{def:Rsharp}
    For web $K\subset\R^{3}$, we denote by $\Rep^{\sharp}(K)$ the
    space of homomorphisms
   \[
          \Rep^{\sharp}(K) = \{ \, \rho : \pi_{1} (\R^{3}\sminus K) \to \SO(3) \mid
           \text{$\rho(m_{e})$ has order $2$ for all edges $e$ } \, \},
    \]
    a space that we can identify with $\Rep(S^{3},K^{\sharp} ; \mu)$
    by the above lemma. We write $\Rep(K)$ for the quotient by the
    action of conjugation,
    \[
            \Rep(K) = \Rep^{\sharp}(K)  / \SO(3).
    \]
\end{definition}

\subsection{Examples of representation varieties}

\label{sec:examples-Rep}

We examine the representation varieties $\Rep^{\sharp}(K)$ and
$\Rep(K)$ for some webs $K$ in $\R^{3}$. First, we make some simple
observations. As in Section~\ref{sec:connections}, as long
as $K$ is non-empty, a flat
$\SO(3)$ bifold connection $(E,A)$ has automorphism group
$\Gamma$ described by one of the following cases, according to
the image of the corresponding representation $\rho$ of the
fundamental group:
\begin{enumerate}
\item the image of $\rho$ is a $2$-element group and the automorphism
    group $\Gamma$ is $O(2)$;
\item the image of $\rho$ is $V_{4}$ and the
    $\Gamma$ is also $V_{4}$;
\item the image of $\rho$ is contained in $O(2)$ and strictly contains
    $V_{4}$, so that the automorphism group $\Gamma(E,A)$ is $\Z/2$;
\item the image of $\rho$ is not contained in a conjugate of $O(2)$
    and the automorphism group $\Gamma$ is trivial.
\end{enumerate}
The first case arises only if  $K$ has no vertices. 
We refer to the remaining three cases as \emph{$V_{4}$ connections}, \emph{fully $O(2)$
  connections}, and \emph{fully irreducible connections} respectively.
By an \emph{$O(2)$ connection}, we mean either a fully $O(2)$
connection or a $V_{4}$ connection.

For a conjugacy class of representations $\rho$ in $\Rep(K)$, the
preimage in $\Rep^{\sharp}(K)$ is a copy of $\SO(3)/\Gamma$, where
$\Gamma$ is described by the above cases. The quotient $\SO(3)/O(2)$
is $\RP^{2}$, the quotient $\SO(3)/V_{4}$ is the flag manifold of
real, unoriented flags in $\R^{3}$, while $\SO(3)/(\Z/2)$ is a lens
space $L(4,1)$.

The following is a simple observation. (In the language of graph theory,
it is the statement that Tait colorings of a cubic graph are the same
as nowhere-zero $4$-flows.)

\begin{lemma}\label{lem:V4-is-Tait}
    If $K\subset \R^{3}$ is any web, then the $V_{4}$ connections in
    $\Rep(K)$ correspond bijectively to the set of all Tait colorings
    of $K$ modulo permutations of the three colors.
\end{lemma}

\begin{proof}
    Since $V_{4}$ is abelian, a homomorphism from
    $\pi_{1}(\R^{3}\setminus K)$ to $V_{4}$ which maps each meridian
    to an element of order $2$ is just the same as an edge-coloring of
    $K$ by the set of three non-trivial elements of $V_{4}$.
\end{proof}

So the $V_{4}$ connections contribute one copy of the flag manifold
$\SO(3)/V_{4}$ to $\Rep^{\sharp}(K)$ for each permutation-class of
Tait coloring of $K$.

\subparagraph{Example: the unknot.}
Let $K$ be an unknotted circle in $S^{3}$. The fundamental group of
the complement is $\Z$, and $\Rep(K)$ contains a single point
corresponding to the representation with image $\Z/2$. The stabilizer
is $O(2)$ and $\Rep^{\sharp}(K)$ is homeomorphic to $\RP^{2}$.

\subparagraph{The theta graph.}
For the theta graph $K$, the representation variety $\Rep(K)$ consists of
a single $V_{4}$ representation, corresponding to the unique class of
Tait colorings for $K$.

\subparagraph{The tetrahedron graph.}
Let $K$ be the $1$-skeleton of a standard simplex in $\R^{3}$, viewed as
a web. All representations in $\Rep(K)$ are again $V_{4}$ connections,
and there is only one of these, because there is only way to
Tait-color this graph, up to permutation of the colors. So
$\Rep^{\sharp}(K)$ is a copy of the flag manifold again.

\subparagraph{The Hopf link with a bar.}
Let $K$ be formed by taking the two circles of the standard Hopf link 
and joining the two
components by an unknotted arc. The complement deformation-retracts to
a punctured torus in such a way that the two generators for the torus
are representatives for the meridians of the Hopf link, and the
puncture is a meridian for the extra arc. If $a$, $b$ and $c$ are
these classes, then a presentation of $\pi_{1}$ is $[a,b]=c$.
For any $\rho\in \Rep^{\sharp}(K)$, the $\rho(a)$ and $\rho(b)$ are
rotations about axes in $\R^{3}$ whose angle is $\pi/4$, because the
commutator needs to be of order $2$. The stabilizer is $\Z/2$ and
$\Rep^{\sharp}(K)$ is therefore a copy of $L(4,1)$.

\subparagraph{The dodecahedron graph.}

Let $K$ be the $1$-skeleton of a regular dodecahedron in $\R^{3}$. The
representation variety  $\Rep^{\sharp}(K)$ is described in
\cite{KM-dodecahedron-rep}. We summarize the results here. The
dodecahedron has $60$ different Tait colorings which form $10$ orbits
under the permutation group of the colors. These contribute $10$
$V_{4}$-connections to $\Rep(K)$ and thence $10$ copies of the flag
manifold to $\Rep^{\sharp}(K)$. There also exist exactly two fully
irreducible representations in $\Rep(K)$. These make a  contribution of
$2$ copies of $\SO(3)$ to $\Rep^{\sharp}(K)$. The representation
variety $\Rep^{\sharp}(K)$ consists therefore of $10$ copies of the
flag manifold and two copies of $\SO(3)$.

\section{Construction of the functor $J$}
\label{sec:J}

\subsection{Discussion}

In this section we will define the instanton Floer homology $J(\check
Y; \mu)$ for any closed, oriented bifold $\check Y$ with
strong marking data $\mu$. The approach is, by now, very standard. We
start with the circle-valued Chern-Simons function on
$\bonf_{l}(\check Y ; \mu)$ and add a suitable real-valued function
$f$ as a perturbation. The perturbation is chosen so as to make $\CS + f$
formally Morse-Smale. We then construct the Morse complex $(C, d)$ for
this function, with $\F$ coefficients, and we define $J(\check
Y;\mu)$ to be the homology of the complex. 
Nearly all the steps are already laid out elsewhere in
the literature: the closest model is the exposition in
\cite{KM-unknot}, which we will follow closely. That paper in turn
draws on the more complete expositions in \cite{KM-singular} and
\cite{KM-book}, to which we eventually refer for details of the proofs.

However, there is one important issue that arises here that is new. The
codimension-$2$ bubbling phenomenon detailed in
Lemma~\ref{lem:Uhlenbeck} and Proposition~\ref{prop:cone-on-V} means
that there are extra considerations in the proof that $d^{2}=0$ in the
Morse complex.

\subsection{Holonomy perturbations}

The necessary material on holonomy perturbations can be drawn directly
from \cite{KM-unknot}, with slight modifications to deal with the
marking data. Fix $\check Y$ with marking data $\mu = (U_{\mu},
E_{\mu})$. Choose a lift of $U_{\mu}$ to a $U(2)$ bundle $\tilde
E_{\mu}$, and fix a connection $\theta$ on its determinant. A $\mu$-marked
connection $(E,A,\tau)$ on $\check Y$ determines a $U(2)$ connection
$\tilde A_{\mu}$ on $E_{\mu}$, with determinant $\theta$. Let $q$ be a
smooth loop in $Y\sminus K$ based at $y$. If we identify the fiber
$E_{y}$ with $\R^{3}$, then the holonomy of $A$ around the loop
becomes an element of $\SO(3)$. If $y$ belongs to $U_{\mu}$, let us
identify the fiber of $\tilde E_{\mu}$ with $\C^{2}$. 
If the loop is entirely contained in
$U_{\mu}$, then using $\tau$ we can interpret the holonomy as an
element of $U(2)$.

As in \cite{KM-unknot}, we now consider a collection
$\mathbf{q} = (q_{1}, \dots, q_{k})$ where each $q_{i}$ is an
immersion
\[
         q_{i} :   S^{1} \times D^{2} \to Y\sminus K.
\]
We suppose that all the $q_{i}$ agree on $p\times D^{2}$, where $p$ is
a basepoint. Let us suppose also that the image of $q_{i}$ is contained in
$U_{\mu}$ for $1\le i \le j$. Choose a trivialization of $\tilde
E^{\mu}$ on the image of $p\times D^{2}$. Then for each $x\in D^{2}$,
the holonomy of $A$ around the $k$ loops determined by $x$ gives a
well-defined element of
\begin{equation}\label{eq:group-product}
          U(2)^{j} \times\SO(3)^{k-j}.
\end{equation}
Fix any conjugation-invariant function
\[
           h :  U(2)^{j} \times\SO(3)^{k-j} \to \R
\]
Denoting by $\mathrm{Hol}_{x}$ the holonomy element in \ref{eq:group-product},
we obtain for each $x$ a well-defined function
$
      H_{x} = h \comp \mathrm{Hol}_{x}: \bonf_{l}(\check Y; \mu) \to \R.
$
We then integrate this function over $D^{2}$ with respect to a
compactly supported bump-form $\nu$ with integral $1$, to obtain a
function
\[
\begin{gathered}
    f_{\mathbf{q}} : \bonf_{l}(\check Y; \mu) \to \R \\
    f_{\mathbf{q}}(E,A,\tau) =\int_{D^2} H_x(A)\nu.
\end{gathered}
\]
We refer to such functions as \emph{cylinder functions}, as in
\cite{KM-unknot}.

Next, following \cite{KM-unknot}, we introduce a suitable Banach
space $\Pert$ of real sequences $\pi = \{\pi_{i}\}$, and a suitable infinite
collection of cylinder functions $\{f_{i}\}$ so that for each $\pi \in
\Pert$, the sum
\[
         f_{\pi} = \sum_{i} \pi_{i} f_{i}
\]
is convergent and defines a smooth function on $\bonf_{l}(\check Y;
\mu)$. Furthermore, we arrange that the formal $L^{2}$ gradient of
$f_{\pi}$ defines a smooth vector field on the Banach manifold
$\bonf_{l}(\check Y;\mu)$, which we denote by $V_{\pi}$. With suitable
choices for $\Pert$ and the $f_{i}$, we arrange that the analytic
conditions of \cite[Proposition 3.7]{KM-singular} hold. In order to have a
large enough space of perturbations, we form the countable collection
$f_{i}$ by taking, for every $k$, a countable  $C^{\infty}$-dense set
in the space of $k$-tuples of immersions $(q_{1},\dots,q_{k})$; and
for each of these $k$-tuples of immersions, we take a
$C^{\infty}$-dense collection of conjugation-invariant functions $h$. 
Cylinder functions separate points and tangent vectors in
$\bonf_{l}(\check Y; \mu)$, because the $\SO(3)$ holonomies already do
so, up to finite ambiguity, and the extra data from the loops in
$U_{\mu}$ is enough to resolve the remaining ambiguity. So we can
achieve the transversality that  we need. We state these consequences
now, referring to \cite{KM-singular,KM-book} for proofs.

Let $\Crit_{\pi} \subset \bonf_{l}(\check Y; \mu)$ be the set of
critical points of $\CS+f_{\pi}$. The space $\Crit_{0}$ coincides with
the representation variety $\Rep(\check Y; \mu)$. For any $\pi$, the
space $\Crit_{\pi}$ is compact, and for a residual set of
perturbations $\pi$ all the critical points are non-degenerate, in the
sense that the Hessian operator has no kernel. In this case,
$\Crit_{\pi}$ is a finite set.

Now fix such a perturbation $\pi_{0}$ with the property that $\Crit_{\pi}$ is
non-degenerate. For any pair of critical points $\alpha$, $\beta$ in
$\Crit_{\pi}$, we can then form the moduli space $M(\alpha,\beta)$ of
formal gradient-flow lines of $\CS+f_{\pi}$ from $\alpha$ to
$\beta$. This can be interpreted as a moduli space of solutions to a
the perturbed anti-self-duality equations on the $4$-dimensional
bifold $\R\times \check Y$. Note that the (non-compact) foam $\R\times
K$ in $\R\times Y$ has no tetrahedral points. Each component of $M(\alpha,
\beta)$ has a formal dimension, given by the spectral flow of the
Hessian. We write $M_{d}(\alpha,\beta)$ for the component of formal
dimension $d$. The group $\R$ acts by translations, and we write
\[
       M'_{d}(\alpha,\beta) = M_{d}(\alpha,\beta)/R.
\]
The relation \eqref{eq:monotone} implies
a bound on the action $\kappa$ depending only on $d$, across moduli
spaces $M_{d'}(\alpha,\beta)$ for all $\alpha$ and $\beta$ and all
$d'\le d$. This is the ``monotone'' condition of \cite{KM-singular}.

We can now find a perturbation
\[
       \pi = \pi_{0} + \pi_{1}
\]
such that $f_{\pi_{1}}$ vanishes in a neighborhood of the critical
point set $\Crit_{\pi_{0}} = \Crit_{\pi}$, and such that the moduli
spaces $M(\alpha,\beta)$ are regular for all $\alpha$ and $\beta$. 

We shall call a perturbation $\pi$ a \emph{good} perturbation if the
critical set $\Crit_{\pi}$ is non-degenerate and all moduli spaces
$M(\alpha,\beta)$ are regular. We shall call $\pi$ \emph{$d$-good} if
it satisfies the weaker condition that the moduli spaces
$M_{d'}(\alpha,\beta)$ are regular for all $d'\le d$.

Suppose $\pi$ is $2$-good. Then the moduli space
$M'_{1}(\alpha,\beta)$ is a finite set of
points. Write
\[
             n_{\alpha,\beta} = \# M'_{1}(\alpha,\beta) \bmod 2.
\]
 Let $C$ be the $\F$ vector space with basis $\Crit_{\pi}$
and let $d: C \to C$ be the linear map whose matrix entry from
$\alpha$ to $\beta$ is $n_{\alpha,\beta}$.

\subsection{Proof that $d^{2}=0$ mod $2$}

We wish to show that $(C,d)$ is a complex whenever $\pi$ is good. Let
us first summarize the usual argument, in order to isolate where a new
issue arises. To show that $d^{2}$ is zero is to show, for all
$\alpha$ and $\beta$ in $\Crit_{\pi}$,
\[
          \sum_{\gamma} n_{\alpha,\gamma}n_{\gamma,\beta} = 0 \pmod 2.
\]
One considers the $1$-dimensional moduli space
$M'_{2}(\alpha,\beta)=M_{2}(\alpha,\beta)/\R$, of trajectories from
$\alpha$ to $\beta$, and one proves that it has a compactification
obtained by adding broken trajectories. The broken trajectories
correspond to elements of $M'_{1}(\alpha,\gamma)\times M'_{1}(\gamma, \beta)$
for some critical point $\gamma$, and the number of broken
trajectories is the sum above. Finally, one must show that the
compactification of $M'_{2}(\alpha,\beta)$ has the structure of a
manifold with boundary, and uses the fact that the number of boundary
points of a compact $1$-manifold is even.

The extra issue in the bifold case is that $2$-dimensional
moduli spaces such as $M_{2}(\alpha,\beta)$ may have non-compactness
due to bubbling, just as in the case of closed manifolds, as described
in Lemma~\ref{lem:Uhlenbeck}. There are no tetrahedral points, so 
the codimension-$2$ bubbles can only arise
on the seams, which in the cylindrical case means that lines $\R\times
V$, where $V\subset \check Y$ is the set of vertices of the web $K$.
 After dividing by translations, we have the following
description of a compactification of $M'_{2}(\alpha,\beta)$:
\[
          \bar M'_{2}(\alpha,\beta) = M'_{2}(\alpha,\beta) \cup 
               \left(\bigcup_{\gamma}M'_{1}(\alpha,\gamma) \times
              M'_{1}(\gamma,\beta) \right) \cup \left( 
           V \times M'_{0}(\alpha,\beta) \right ).
\]
Note that $M'_{0}(\alpha,\beta)$ is empty unless $\alpha=\beta$,
because it consists only of constant trajectories. So the matrix entry
of $d^{2}$ from $\alpha$ to $\beta$ can be shown to be zero for
$\alpha\ne\beta$ as usual. For $\alpha=\beta$, we have
\[
          \bar M'_{2}(\alpha,\alpha) = M'_{2}(\alpha,\alpha) \cup 
               \left(\bigcup_{\gamma}M'_{1}(\alpha,\gamma) \times
              M'_{1}(\gamma,\alpha) \right) \cup V.
\]

We would like to be able to apply Proposition~\ref{prop:cone-on-V} to
the cylinder $\R\times\check Y$, to conclude that a neighborhood of
each vertex $v \in V$ in $\bar M'_{2}(\alpha,\alpha)$ is homeomorphic
to $4$ intervals joined at one endpoint. This would tell us that $4$
ends of the $1$-manifold $M'_{2}(\alpha,\alpha)$ are incident at each
point of $V$ in the compactification. Since $4$ is even (and indeed,
since the number of vertices is also even), it would follow that
$d^{2}=0$.

However, Proposition~\ref{prop:cone-on-V} requires gluing theory which
has not been carried out in the situation where holonomy perturbations
are present at the point where the bubble occurs. For a holonomy
perturbation $f_{\pi}$, the $\check Y$-support of $f_{\pi}$ will mean
the closure in $\check Y$ of the union of the images  of all the 
immersions  $q : (S^{1}\times D^{2}) \to \check Y\sminus K$ used in
the cylinder functions for $f_{\pi}$. If $f_{\pi}$ is a \emph{finite}
sum of cylinder functions, then the $\check Y$-support is a compact
subset of $\check Y \subset K$ and is therefore disjoint from some
open neighborhoods of the vertex set $V$. In this case, the gluing
theory for bubbles at $V$ is standard, and we can apply
Proposition~\ref{prop:cone-on-V} to conclude that $d^{2}=0$.

It may not be the case that there is a good perturbation which is a
finite of cylinder functions. However, the proof that $d^{2}=0$ only
involves moduli spaces with $d\le 2$. So the following Proposition is
the key.

\begin{proposition}
    There exists a $2$-good perturbation $f_{\pi}$ which is a finite
    sum of cylinder functions.
\end{proposition}

\begin{proof}
     We can approximate any $\pi$ by a finite sum, so the issue is
     openness of the required regularity conditions. Recall that we
     construct good perturbations $\pi$ as $\pi_{0} + \pi_{1}$, where
     $\pi_{0}$ is chosen first to make $\Crit_{\pi_{0}}$
     non-degenerate. The non-degeneracy of the compact critical set
     $\Crit_{\pi_{0}}$ is an open condition, so we certainly arrange
     that $f_{\pi_{0}}$ is a finite sum by truncating $\pi_{0}$ after
     finitely many terms.

     With $\pi_{0}$ fixed, choose a good $\pi=\pi_{0} + \pi_{1}$ as
     before, with $f_{\pi} = f_{\pi_{0}}$ in a neighborhood of
     $\Crit_{\pi}$. Choose an approximating sequence
      \[
                 \pi^{m} \to \pi
      \]
      in $\Pert$, with the property that each $f_{\pi^{m}}$ is a
      finite sum and $f_{\pi^{m}}=f_{\pi_{0}}$ on the same
      neighborhood of $\Crit_{\pi}$. From the fact that $\pi$ is
      $2$-good, we wish to conclude that $\pi^{m}$ is $2$-good for
      sufficiently large $m$.

      Let $M(\alpha,\beta)_{m}$ denote the
      moduli spaces for the perturbation $\pi_{m}$. 
      For $d \le 1 $, there is a straightforward induction argument to show that
      $\pi_{m}$ is $d$-good for $m$ sufficiently large. This starts
      with the observation that, for $d$ sufficiently negative, the
      moduli spaces $M_{d}(\alpha,\beta)_{m}$ are empty because they have
      negative action $\kappa$, as a consequence of equation
      \eqref{eq:monotone}. If $d_{0} \le 0$ is the first dimension for
      which any the moduli spaces $M'_{d_{0}}(\alpha,\beta)_{m}$ is
      non-empty, then these moduli spaces are compact, and their
      regularity is therefore an open condition. It follows that these
      moduli spaces are regular (i.e. empty if $d_{0}<0$) for $m$
      sufficiently large.

      Compactness also holds for $M'_{1}(\alpha,\beta)$ once the lower
      moduli spaces are regular. So we may assume that all the
      $\pi^{m}$ are $1$-regular. We must prove that they are then
      $2$-regular for large $m$.

      Suppose the contrary. 
      Passing to a subsequence, we may assume
      that there is a critical point $\alpha$ and for each $m$ a
      perturbed anti-self-dual solution $A_{m}$ representing a point
      of $M_{2}(\alpha,\alpha)_{m}$ which is not regular. This means that
      the cokernel of the linearized operator is not surjective at $A_{m}$, and
      we can choose a  non-zero solution 
      \[      
      \omega_{m} \in L^{2}_{l}(\R\times
      \check Y; \Lambda^{+}\otimes E)
      \]
       of the formal adjoint equation (i.e. the operator
       $d^{*}_{A_{m}}$ plus a perturbation determined by the holonomy
       of $A_{m}$).   

        In seeking a contradiction, the interesting cases are when
        $A_{m}$ is either converging to a broken trajectory or is
        approaching a point in the Uhlenbeck compactification with a
        bubble at a point on a seam. The case of a broken trajectory
        is standard: if a sequence $A_{m}$ converges to a broken
        trajectory and the components of the broken trajectory are
        regular, then $A_{m}$ is regular for large $m$. This is a standard
        consequence of the proof of the gluing theorem for broken
        trajectories.

\begin{figure}
    \begin{center}
        \includegraphics[scale=0.85]{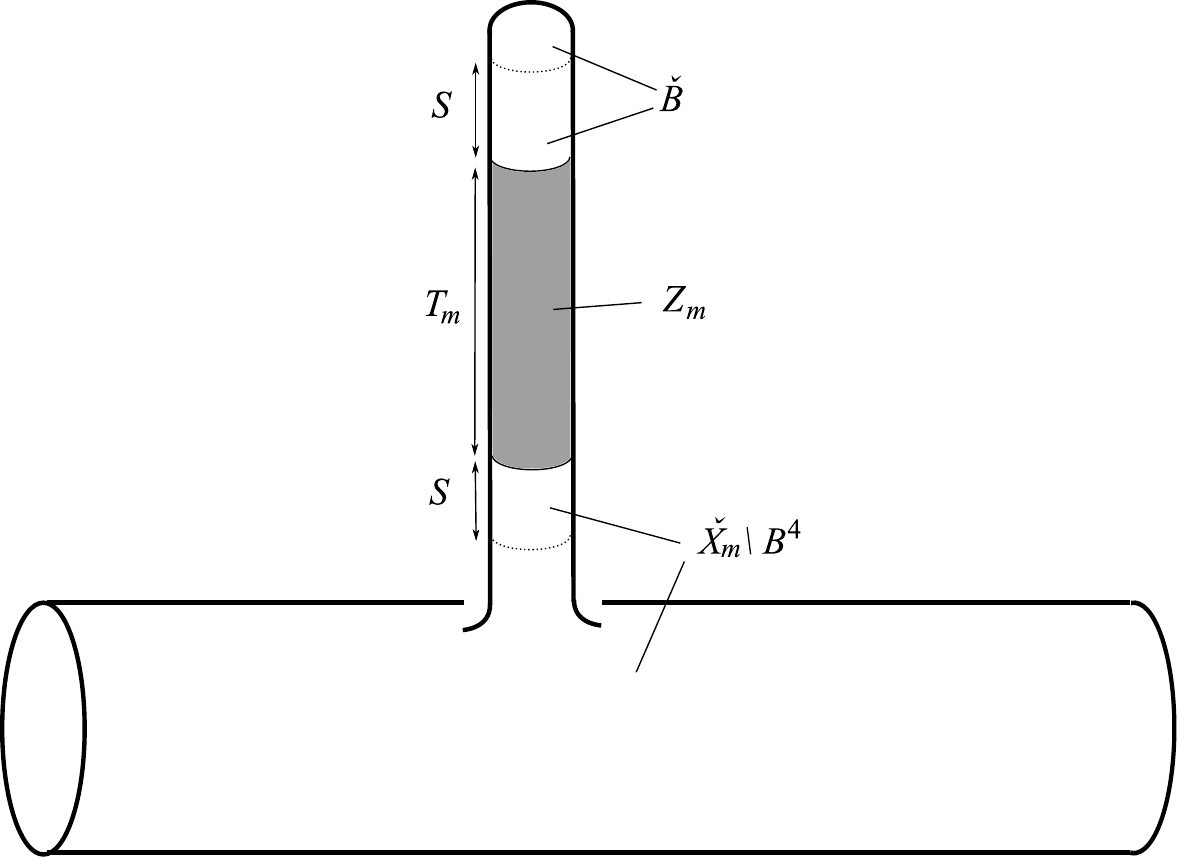}
    \end{center}
    \caption{\label{fig:conformal}
   Making a conformal change at the bubble.}
\end{figure}

        So let us suppose that $A_{m}$ converges to a point $(p,A)$ in
        the Uhlenbeck compactification, where $p$ is a point on a seam
        $\R\times \{v\}$ and $A \in M_{0}(\alpha,\alpha)$ is the
        constant solution. Let us apply a conformal change to
        $\R\times \check Y$, as indicated in
        Figure~\ref{fig:conformal}, so that the new manifold  $\check
        X_{m}$ contains
        a cylinder on $S^{3}/V_{4}$ of length $T_{m}+2S$. The length
        $S$ will be fixed and large. We can choose the lengths $T_{m}$
        and the centers of the conformal transformations so that the
        $A_{m}$ converge in the sense of broken trajectories as
        $T_{m}\to\infty$. In the limit, we obtain two pieces. We
        obtain 
        an orbifold
        instanton on $\check B=B^{4}/V_{4}$ equipped with a cylindrical end
        (or equivalently $S^{4}/V_{4}$), with action $\kappa=1/4$; and
        we obtain the conformal transformation of the 
         constant solution  $A$ on $(\R\times\check
        Y)\sminus\{p\}$.  

        On $\check X_{m}$, the equations satisfied by $A_{m}$ are of
        the form $F^{+}_{A_{m}} = W_{m}$, where $W_{m}$ is obtained
        from the holonomy perturbation. In the original metric on
        $\R\times \check Y$, these perturbing terms satisfy a uniform
        $L^{\infty}$ bound. So in the conformally equivalent metric of $\check X_{m}$, there
        is a bound of the form
         \[
                      \| W_{m} \|_{L^{\infty}(Z_{m})} \le c_{1}
                      e^{-c_{2} S}
         \]
          on the cylinder $Z_{m}$ (see
          Figure~\ref{fig:conformal}). For any $\epsilon$, we
           may arrange, by increasing $S$, that the connections
           $A_{m}$ are all $\epsilon$-close to the constant flat
           $V_{4}$-connection in $L^{p}_{1}$ norm on $Z_{m}$, for any $p$.
             
          Consider now the cokernel elements $\omega_{m}$. Scale these
          so that their $L^{2}$ norm is $1$ on $\check X_{m}$. Since
          the $A_{m}$ are converging locally in $L^{p}_{1}$ for all
          $p$, we have a uniform $C^{0}$ bound on $\omega_{m}$ in the
          metric of $\check X_{m}$, and (after passing to a
          subsequence) convergence to a limit $\omega$.  We may assume
          that the $\omega_{m}$ have uniform exponential decay on the
          two ends of $\R\times \check Y$, for otherwise we are in the
          standard case where the bubble plays no role. Because
          $M_{0}(\alpha,\alpha)$ is regular, the limit $\omega$ must
          be zero on $\check X_{m}\sminus \check B^{4}$, and the uniform
          exponential decay on the two ends means that
          \[
                 \int_{\check X_{m}\sminus \check B^{4}}
                 |\omega_{m}|^{2} \to 0.
          \]
          Because all orbifold instantons on $\R^{4}/V^{4}$ are also
          regular, we have
          \[
              \int_{\check B}
                 |\omega_{m}|^{2} \to 0
          \]
          also. It follows that
           \[
       \int_{Z_{m}}
                 |\omega_{m}|^{2} \to 1.
           \]
           However, $\omega_{m}$ is converging to zero on the two
           boundary components of the cylinder $Z_{m} = [0,T_m]\times S^{3}/V_{4}$, and satisfies
           an equation which is schematically of the shape
            \[
                d^{*}\omega_{m} + A_{m}\cdot \omega_{m}= Q_{m},
            \]
            where $Q_{m}$ is the contribution coming from the
            non-local perturbation (and is therefore dependent on the
            value of $\omega_{,}$ on parts of $\check X_{m}$ that are
            not in $Z_{m}$).  The $L^{p}_{1}$ convergence of the terms
            $A_{m}$ mean that the corresponding multiplication
            operators that appear on the left can be taken to be
            uniformly small in operator norm, as operators either from
            $L^{2}_{1}$ to $L^{2}$, or as operators on weighted
            Sobolev spaces, $L^{2}_{1,\delta}\to L^{2}_{\delta}$,
            weighted by $e^{\delta T}$. The terms $Q_{m}$ are also
            uniformly bounded in weighted $L^{2}$ norms. So standard
            arguments (see \cite{KM-book} for example) establish that
            the $L^{2}$ norm of $\omega_{m}$ on length-$1$ cylinders
            $[t,t+1]\times (S^{3}/V_{4})$ is uniformly controlled by a
            function that decays exponentially towards the middle of
            the cylinder.  It follows that the integral of
            $|\omega_{m}|^{2}$ on $Z_{m}$ is also going to zero, which
            is a contradiction.
\end{proof}

The Proposition shows that we can construct $2$-good perturbations for
which $d^{2}=0$. It follows that, in fact, $(C,d)$ is a complex for
all $2$-good perturbations, and therefore for all good
perturbations. This is because, up to isomorphism, the complex $C$ and
boundary map $d$ are stable under small changes in the perturbation,
by the compactness of $\Crit_{\pi}$ and $M'_{1}(\alpha,\beta)$. So we
can always approximate a given good perturbation by one for which the
proof that $d^{2}=0$ applies, without actually changing $d$. Thus, we
have:
\begin{proposition}
    The map $d : C\to C$ satisfies $d^{2}=0$, for all good
    perturbations $\pi$.
\end{proposition}

We can now define the instanton Floer homology:

\begin{definition}
    For any closed, oriented bifold $Y$ with strong marking data $\mu$, we
    define $J(\check Y; \mu)$ to be the homology of the complex
    $(C,d)$ constructed above, for any choice of orbifold Riemannian
    metric $\check g$ and any good perturbation $\pi$. When $\check Y$
    is obtained from a $3$-manifold $Y$ with an embedded web
    $K$, we may write $J(Y, K ; \mu)$.
\end{definition}

\subsection{Functoriality of $J$}
\label{sec:functoriality}

In the above definition, the group $J(\check Y; \mu)$ depends on a
choice of metric and perturbation. As usual with Floer theories, the
fact that $J(\check Y; \mu)$ is a topological invariant is a formal
consequence of the more general functorial property of instanton
homology. 

Let $\check Y_{1}$ and $\check Y_{0}$ be two $3$-dimensional
bifolds with orbifold metrics $\check g_{1}$ and $\check
g_{0}$. Let $\check X$ be an oriented cobordism from $\check
Y_{1}$ to $\check Y_{0}$. Equip $\check X$ with an orbifold metric
that is a product in a collar of each end. Let $\mu_{1}$ and $\mu_{0}$
be marking data for the two ends, and let $\nu$ be marking data for
$\check X$ that restricts to $\mu_{1}$ and $\mu_{0}$ at the
boundary. Let $\pi_{1}$ and $\pi_{0}$ be good perturbations on $\check
Y_{1}$ and $\check Y_{0}$.  If we attach cylindrical ends to $\check
X$ and choose auxiliary perturbations in the neighborhoods of the two
boundary components \cite{KM-unknot, KM-singular, KM-book}, then we
have moduli spaces of solutions to the perturbed anti-self-duality
equations on the bifold:
\begin{equation}\label{eq:cobordism-moduli-space}
                  M(\check X, \nu ; \alpha,\beta).
\end{equation}
Here $\alpha$ and $\beta$ are critical points for the perturbed
Chern-Simons functional, i.e.~generators for the chain complexes
$(C_{1}, d_{1})$ and $(C_{0}, d_{0})$ associated with $(\check Y_{1},
\mu_{1})$ and $(\check Y_{0}, \mu_{0})$. Thus we
can use $(\check X; \nu)$ to define a linear map
\[
               (C_{1}, d_{1}) \to (C_{0}, d_{0})
\]
between the respective complexes, by counting solutions to the
perturbed equations in zero-dimensional moduli spaces. The usual proof
that the linear map defined in this way is a \emph{chain} map involves
moduli spaces only of dimension $0$ and $1$. In particular, the
codimension-$2$ bubbling phenomenon of Proposition~\ref{prop:cone-on-V}
does not enter into the argument (unlike the proof that
$d^{2}=0$). However, if the foam $\Sigma$ has tetrahedral points, then
the codimension-$1$ bubbling phenomenon of
Proposition~\ref{prop:tetrahedral-cone-on-V} comes into play.  That
proposition tells us that
the number of ends of each $1$-dimensional moduli space that are
accounted for by bubbling in the Uhlenbeck compactification is a
multiple of $4$, and in particular even. So the
usual proof carries over as long as we are using coefficients $\F$.

In this way, $(\check X; \nu)$ defines a map on homology,
\[
   J(\check X; \nu) :   J(\check Y_{1}; \mu_{1}) \to J(\check Y_{0} ; \mu_{0}).
\]
Furthermore, a chain-homotopy argument shows that the map is
independent of the choice of metric on the interior of $\check X$ and
is independent also of the auxiliary perturbations on the cobordism.

Consider now the composition of two cobordisms. Suppose we have two
cobordisms $(\check X' ; \nu')$ and $(\check X'' ; \nu'')$, where the
first is a cobordism to $(\check Y; \mu)$ and the second is a
cobordism from $(\check Y; \mu)$. This means that the restrictions of
$E_{\nu'}$ and $E_{\nu''}$ to the common subset $E_{\mu}\subset \check
Y$ are the same bundle (not just isomorphic), so together they define
a bundle $E_{\nu'}\cup E_{\nu''}$. We can therefore
join the cobordisms canonically to  make a composite cobordism with marking data, 
\[
             (X; \nu) = (\check X' \cup \check X'' ; \nu' \cup \nu'').
\]

In the above situation it is \emph{not} always the case that
\begin{equation}\label{eq:compositon-law}
                J(\check X; \nu) =  J(\check X''; \nu'') \comp  J(\check X'; \nu').
\end{equation}
To see the reason for this, consider the problem of joining together
marked connections. Suppose we have marked connection $(E', A',
\sigma')$ and $(E'', A'', \sigma'')$ on $(\check X' ; \nu')$ and
$(\check  X'' ; \nu'')$ respectively. By restriction, we obtain two
$\mu$-marked connections on $(\check Y; \mu)$, say $(F' , B', \tau')$
and $(F'', B'', \tau'')$. Let us suppose these are isomorphic: they
define the same point in $\bonf(\check Y ; \mu)$. To avoid the issue
of differentiability, let us also suppose that  $(E', A',
\sigma')$ coincides with the pull-back of $(F' , B', \tau')$ in a
collar of $\check Y$, and similarly with $E''$. Because the marking is
strong, there is a unique isomorphism, which is a bundle isomorphism
\[
            \phi : F' \to F''
\]
such that $\phi^{*}(B'') = B'$ and such that the map 
\[
     \psi =  (\tau'')^{-1} \phi \tau' : E_{\mu} \to E_{\mu}
\] 
lifts to determinant $1$.  

These conditions on $\phi$ are not enough to allow us
to construct a $\nu$-marked connection $(E, A, \sigma)$ on the union
$\check X$. Certainly we can use $\phi$ to glue $E'$ to $E''$ along
the common boundary, and the union of the connections $A'$ and $A''$
will give us a connection $A$. So we do have an unmarked $\SO(3)$
connection $(E,A)$ on $\check X$. However, $(E,A)$ does not have a
$\nu$-marking. Instead of having a bundle isomorphism
\[
              \sigma :  E_{\nu} \to E|_{U_{\nu}\sminus \Sigma}
\]
we have a pair of bundle isomorphisms $\bar\sigma'$ and $\bar\sigma''$
(essentially the maps $\sigma'$ and $\sigma''$) defined over
$U_{\nu'}$ and $U_{\nu''}$. To form $\sigma$, we would want
$\bar\sigma'$ and $\bar\sigma''$ to be equal on their common domain
$E_{\mu}$. Instead, we are only given that
$(\bar\sigma'')^{-1}\sigma'|_{E_{\mu}}$ lifts to determinant $1$.

For any $\SO(3)$ bundle $F$, let $\delta(F)$ denote the group of
bundle automorphisms of $F$ modulo those that lift to determinant
$1$. From the above discussion, we see that we need the following
diagram to be a fiber product:
\[
  \xymatrix{  & \delta(E_{\nu'}) \ar[dr] & \\ \delta(E_{\nu})  \ar[ur]
    \ar[dr] & & \delta(E_{\mu})  \\ & \delta(E_{\nu''}) \ar[ur]&}               
\]
Since $\delta(F)$ is isomorphic to $H^{1}$ of the base with $\F$
coefficients, this is equivalent to the exactness of the sequence
\begin{equation}\label{eq:exact-on-H1}
       0 \to H^{1}(W;\F) \to H^{1}(W' ; \F) \oplus H^{1}(W'' ;
       \F)
                 \to H^{1}(W' \cap W'' ; \F) \to 0
\end{equation}
where $W= U_{\nu} \sminus \Sigma \subset \check X$ and so on. To
summarize this, we have:

\begin{proposition}
    The composition law \eqref{eq:compositon-law} holds provided that
    the sequence \eqref{eq:exact-on-H1} is exact.
\end{proposition}

Using the Mayer-Vietoris sequence, we see that one way to ensure
exactness of \eqref{eq:exact-on-H1} is to require that the maps
\[
               H^{i}(W'; \F) \to H^{i}(W' \cap W''; \F)
\]
are surjective for $i=0$ and $1$.  So we make the following
definition.

\begin{definition}
    If $(\check X , \nu)$ is a marked cobordism from $(\check Y_{1},
    \mu_{1})$ to $(\check Y_{0}, \mu_{0})$, we say that the marking
    data $\nu$ is \emph{right-proper} if the map
      \[
                H^{i}(U_{\nu} \sminus \Sigma ;\F) \to
                H^{i}(U_{\mu_{0}} \sminus K_{0};\F) 
       \]
     is surjective for $i=0$ and $1$. We define left-proper similarly.
\end{definition}

Thus the composition law \eqref{eq:compositon-law} 
always holds if the marking data $\nu'$ on $(\check X' ; \nu')$ is
right-proper, or if the marking data $\nu''$ on $(\check X'' ; \nu'')$
is left-proper. Furthermore, if we compose two cobordisms with
right-proper marking, then the composite marking is right-proper also
(and similarly with left-proper). 

Passing to the language of smooth $3$-manifolds $Y$ with embedded
webs $K$, and smooth $4$-manifolds $X$ with embedded
foams $\Sigma$, we have the following description. According to the
definition, a foam comes with marked points (dots) on the faces, but
we consider first the case that \emph{there are no dots}. There
is a category in which an object is a quadruple $(Y, K, \mu, a)$,
where $Y$ is a closed, oriented, connected $3$-manifold, $K$ is an
embedded web, $\mu$ is strong marking data, and $a$ is
auxiliary data consisting of a choice of orbifold metric $\check g$ on
a corresponding orbifold $\check Y$ and a choice of good perturbation
$\pi \in \Pert$. A morphism in this category from $(Y_{1}, K_{1},
\mu_{1}, a_{1})$ to $(Y_{0}, K_{0}, \mu_{0}, a_{0})$ 
is an isomorphism class of compact, oriented,
4-manifold-with-boundary, 
$X$, containing a foam $\Sigma$ and right-proper marking
data $\nu$, together
with orientation-preserving identifications $\iota$ of $(\partial X, \partial
\Sigma, \partial \nu)$ with 
\[ (-Y_{1}, K_{1}, \mu_{1}) \cup (Y_{0},
K_{0}, \mu_{0}).
\]
The construction $J$ defines a functor from this category to vector
spaces over $\F$.

If only the auxiliary data $a_{i}$ differ, then there is a canonical
product morphism from $(Y, K , \mu, a_{1})$ to $(Y, K, \mu, a_{0})$,
which is an isomorphism. So we can drop the auxiliary data $a$ and
take an object in our category to be a triple $(Y, K, \mu)$ and an
morphism to be (still) an isomorphism class of triples
$(X,\Sigma,\nu)$.

\begin{proposition}\label{prop:Cato-functor}
    Let $\Cat_{o}$ be the category in which an object is a triple $(Y,
    K,\mu)$, consisting a  closed,
    connected, oriented $3$-manifold $Y$, an embedded web
    $K$, and strong marking data $\mu$, and in which the morphisms are
    isomorphism classes of foam cobordisms $(X,\Sigma,\nu)$ with
    right-proper marking data $\nu$. (Here again $X$ is an oriented
    $4$-manifold with boundary and $\Sigma$ is an embedded foam, without
    dots.) Then
    $J$ defines a functor from $\Cat_{o}$ to vector spaces over $\F$.
\end{proposition}

We use the terminology $\Cat_{o}$ rather than $\Cat$ because we
reserve $\Cat$ for the category in which the foams are allowed to have
\emph{dots}. We turn to this matter next.

\subsection{Foams with dots}
\label{subsec:dots}

Until now, the foam  $\Sigma$ has been dotless. We now explain how to
extend the definition of the linear maps defined by cobordisms to the
case of foams with dots. For now, we work in the general setting of
cobordisms between $3$-dimensional bifolds equipped with strong marking data.

So consider a $4$-manifold $X$ with a foam
$\Sigma\subset X$ and strong marking data $\nu$. Let $\check X$ be a
corresponding $4$-dimensional Riemannian bifold. Let $\delta \in
\Sigma$ be a point lying on a face (a ``dot''). Associated with
$\delta$ is a real line bundle over the space of connections,
\[
            L_{\delta} \to \bonf_{l}(\check X; \nu),
\]
defined as follows. Let $B_{\delta}\subset X$ be a ball neighborhood
of $\delta$, and consider marking data $\mu$ given by the trivial
bundle over $U_{\mu}=B_{\delta}\setminus \Sigma$. There is a double covering,
\[
               \bonf_{l}(\check X; \nu, \mu) \to  \bonf_{l}(\check X; \nu)
\]
where the left-hand side is the space of isomorphism classes of bifold
connections $(E,A)$ equipped with \emph{two} markings: a $\nu$-marking
$\tau$, and a $\mu$-marking $\sigma$. The covering map is the map that
forgets $\sigma$:
\[
              (E,A,\tau,\sigma) \mapsto (E, A, \tau).
\]
The fact that this is a double cover is essentially the same point as
in Lemma~\ref{lem:covering}.
 We take $L_{\delta}$
to be the real line bundle associated to this double cover.

The line bundle $L_{\delta}$ can be regarded as being pulled back from
the space of connections over a suitable open subset of $X$. Specifically,
suppose $X_{\delta}\subset X$ is a connected open set such that
the two restriction maps
\[
             H^{1}( U_{\nu} ; \F)  \to H^{1}(X_{\delta} \cap
             U_{\nu} ; \F)
\]
and
\[
             H^{1}( U_{\mu} ; \F)  \to H^{1}(X_{\delta} \cap
             U_{\mu} ; \F)
\]
are injective. On $X_{\delta}$ we have marking data $\nu_{\delta}$ and
$\mu_{\delta}$ by restriction; and the line bundle $L_{\delta}$ can be
regarded as pulled back via the restriction map
\[
 \bonf_{l}(\check X; \nu) \to  \bonf_{l}(\check X_{\delta};
 \nu_{\delta}).
\]
We will choose $X_{\delta}$ so that it is disjoint from the foam and
the boundary of $X$. (Think of a regular neighborhood of a collection
of loops in $X$.)

In this way, given a collection of points $\delta_{i}$ ($i=1,\dots,n$), we obtain a
collection of real line bundles $L_{\delta_{i}}$ which we may regard as pulled
back from spaces of marked connections over \emph{disjoint} open sets
$X_{\delta_{i}}\subset X$. Let $s_{i}$ be a section of
$L_{\delta_{i}}$ that is also pulled back from $X_{\delta_{i}}$, and
let
\[
        V(\delta_{i})\subset  \bonf_{l}(\check X; \nu)
\]
be its zero-set. If $\check X$ is a bifold cobordism to which we
attach cylindrical ends, we can then consider the moduli spaces
of solutions to the perturbed anti-self-duality equations
\eqref{eq:cobordism-moduli-space} \emph{cut
  down} by the $V(\delta_{i})$:
\[
                M(\check X, \nu ; \alpha,\beta) \cap V(\delta_{1})
                \cap\dots\cap V(\delta_{n}).
\]
The sections $s_{\delta_{i}}$ can be chosen so that all such
intersections are transverse. The disjointness of the $X_{\delta_{i}}$
ensures the necessary compactness properties when bubbles occur and
these moduli spaces define chain maps in the usual way. The resulting
maps on homology are the required maps defined by the cobordism with
dots $\delta_{1},\dots,\delta_{n}$ on the foam. We have:

\begin{proposition}\label{prop:Cat-functor}
    The above construction extends $J$ from the category $\Cat_{o}$ of
    dotless foams to the category $\Cat$ of foams with dots, in which an object is a triple $(Y,
    K,\mu)$, consisting a  closed,
    connected, oriented $3$-manifold $Y$, an embedded web
    $K$, and strong marking data $\mu$, and in which the morphisms are
    isomorphism classes of foam cobordisms $(X,\Sigma,\nu)$ with
    right-proper marking data $\nu$.
\end{proposition}

As a special case of the construction, given a web $K\subset Y$ we can define a collection of
operators on $J(Y,K,\mu)$, one for each edge of $K$:

\begin{definition}\label{def:u-map}
    For each edge $e$ of a web $K$, we write
\[
            u_{e} : J(Y,K,\mu) \to J(Y,K,\mu)
\]
    for the operator corresponding to cylindrical cobordism with a
    single dot on the face corresponding to the edge $e$.
\end{definition}

The operators are $u_{e}$ are not independent. Consider a morphism
$(X,\Sigma, \nu)$ in the foam category $\Cat$, let $s$ be a seam in
$\Sigma$, and let $\delta_{1}$, $\delta_{2}$ and $\delta_{3}$ be
points on the three facets of $\Sigma$ that are locally incident along
$s$. Let $\Sigma(\delta_{i})$ be the foam obtained from
$\Sigma$ be placing one additional dot at $\delta_{i}$.
\begin{figure}
    \begin{center}
        \includegraphics[scale=0.4]{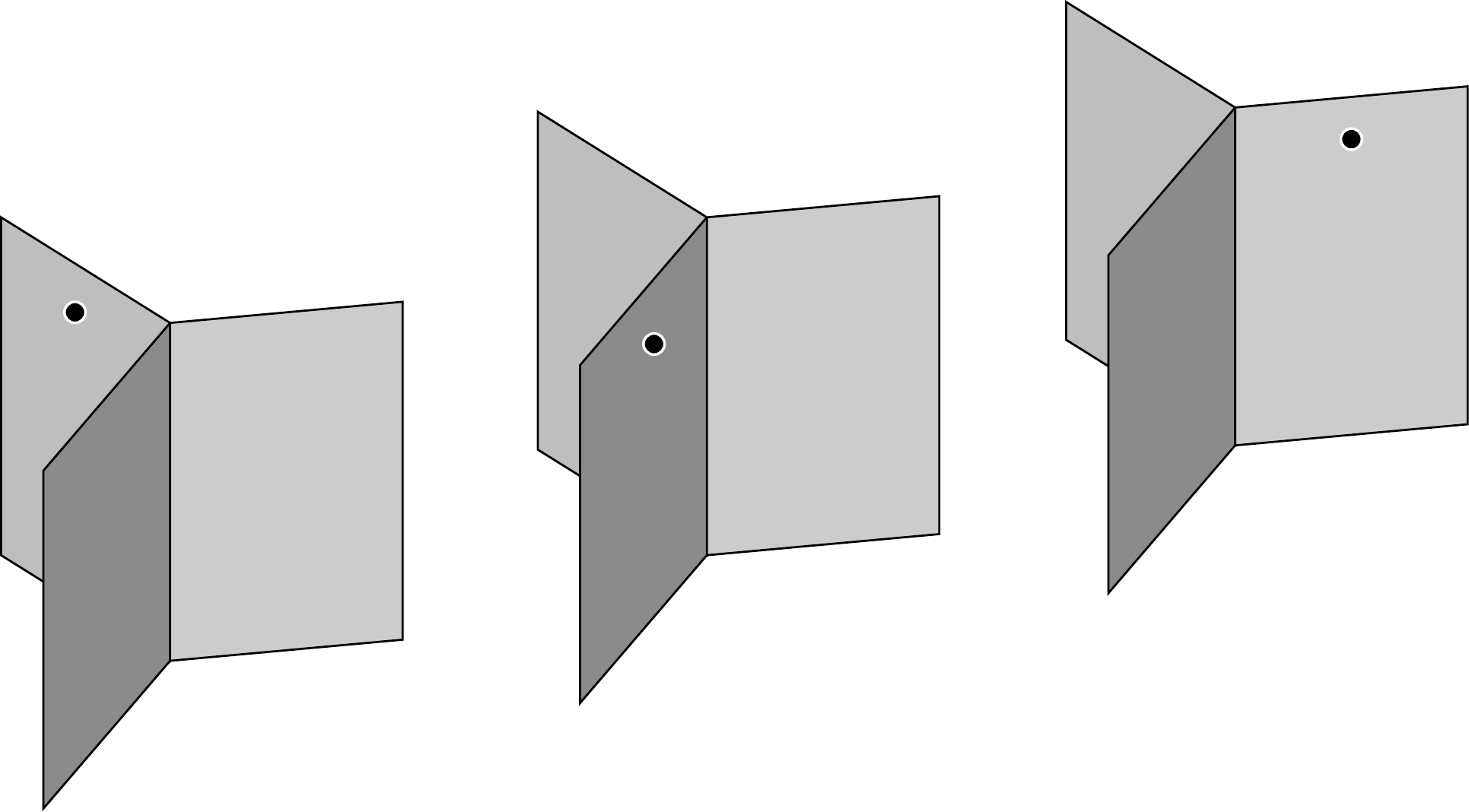}
    \end{center}
    \caption{\label{fig:first-dot-rule}
   The three foams $\Sigma(\delta_{i})$ ($i=1,2,3$) that appear in Proposition~\ref{prop:dot-migration}.}
\end{figure}
 See
Figure~\ref{fig:first-dot-rule}. Then we have relation (which is also
part of the set-up of the $\sl_{3}$ homology of \cite{Khovanov-sl3}):

\begin{proposition}\label{prop:dot-migration}
    The linear maps $J(\delta_{i})=J(X,\Sigma(\delta_{i}),\nu)$
    satisfy the relation
     \[
               J(\delta_{1}) + J(\delta_{2}) + J(\delta_{3})=0.
     \]
     In particular (and equivalently), if $e_{1}$, $e_{2}$ and $e_{3}$
     are the edges incident at a vertex of $K$,  then the three
     operators $u_{e_{i}}$ from Definition~\ref{def:u-map} satisfy
     $u_{e_{1}}+u_{e_{2}}+ u_{e_{3}}=0$.
\end{proposition}

\begin{proof}
    The corresponding real line bundles $L_{i}$ satisfy
     \[
                     L_{1}\otimes L_{2} \otimes L_{3}=\underline{\R}
     \]
     so the homology classes dual to the zero-sets $V_{i}$ of sections
     of these line bundles add up to zero. Furthermore, this is
     already true in the space of connections on a connected open set $Z\subset
     \check X$ that includes the domain of the marking data $\nu$
     together with a neighborhood of a $2$-sphere meeting $\Sigma$ in
     the three points $\delta_{i}$ (i.e. a $2$-sphere which is a link
     of the seam). We can use $Z$ in place of $X_{\delta_{i}}$ in defining
     the maps corresponding to the foams with dots, and the result follows.
\end{proof}

There is an alternative way to incorporate dots into our
definitions. The authors do \emph{not} know in general whether this
definition is equivalent to our standard definition (the one
above). Given a foam with dots, $\Sigma \subset X$, we can
construct a new foam $\Sigma' \subset X$ as follows. We choose 
standard $4$-balls $B_{i}$ centered on each dot $\delta_{i}$ and meeting $\Sigma$ in
standard $2$-disks $D_{i}$. We define $\Sigma'$ to be the union of $\Sigma$ with
collection of genus-$1$ surfaces with boundary, $T_{i}$. Each $T_{i}$ is
contained in the corresponding ball $B_{i}$,
where it is isotopic to a standardly-embedded surface in
$B^{3}$, with $\partial T_{i}$ meeting $D_{i}$ in a circle so as to
form a new circular seam in $\Sigma'$.
\begin{figure}
    \begin{center}
        \includegraphics[scale=0.4]{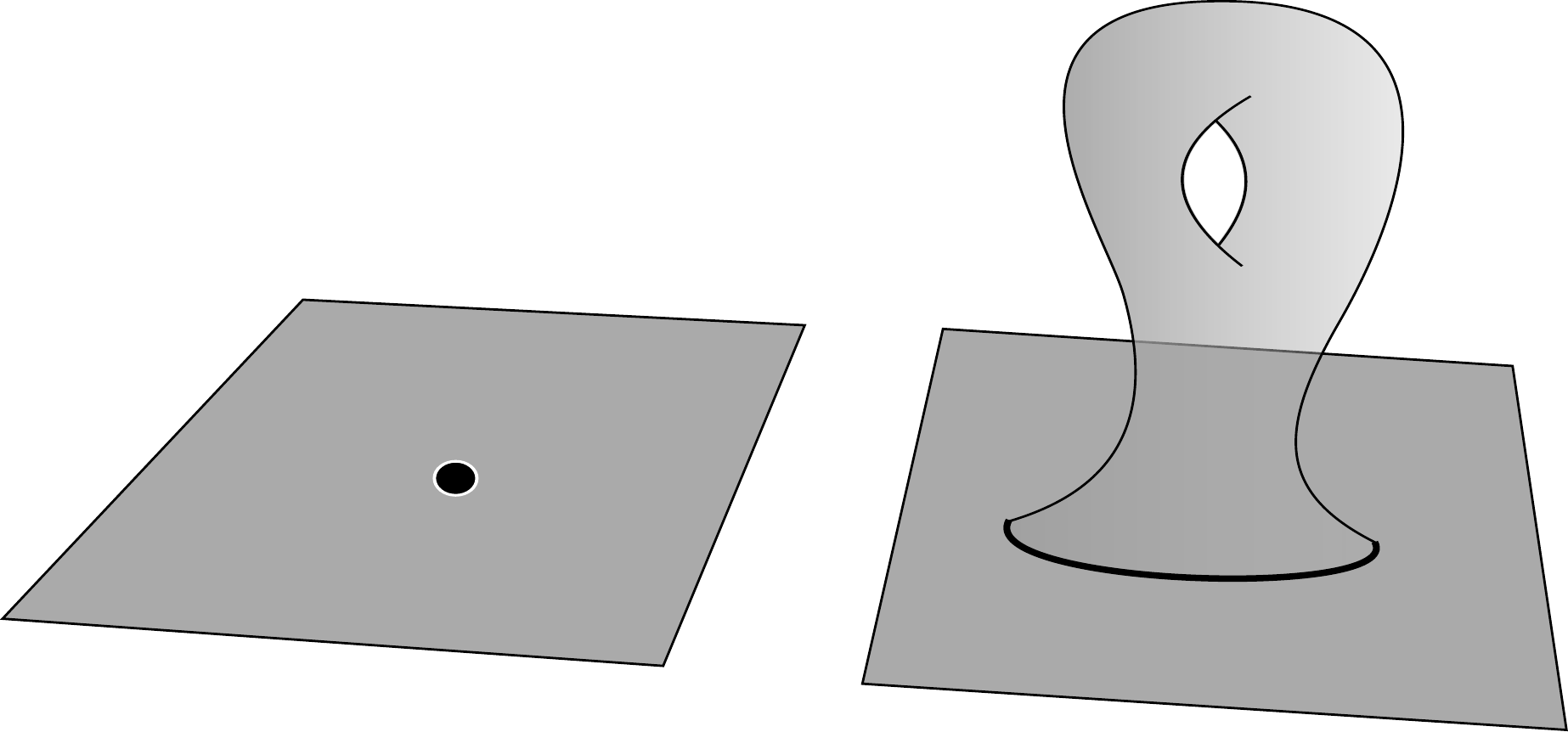}
    \end{center}
    \caption{\label{fig:dot-construction}
   Replacing a dot with a genus-$1$ surface attached along a circular seam.}
\end{figure}
 See
Figure~\ref{fig:dot-construction}. 
This construction is
functorial, from the category in which the morphisms are foams
possibly with dots, to the category of dotless foams. We extend the
definition of $J$ to foams with dots by composing with this functor.

We have now defined two different ways to extend the functoriality of
$J$, from foams without dots to foams with dots. Since it is not clear
that the definitions are the same, we take the first as standard. It
turns out that the two definitions \emph{are} the same in the more
restricted setting of the functor $\Jsharp$ defined in the next
subsection.

\subsection{Defining $J^{\sharp}$}
\label{subsec:def-Jsharp}

The objects in the category $\Cat$ on which $J$ is a functor are
required to be equipped with strong marking data, in order to avoid
reducibles. Following \cite{KM-unknot}, we introduce a variant of the
construction which applies to arbitrary webs in $3$-manifolds.
 Suppose we are
given a closed, connected, oriented $3$-manifold $Y$ with a framed basepoint $y_{0}$. Given
a web $K \subset Y$ disjoint from the basepoint, we can form the
union $K\cup H$, where the Hopf link $H$ is contained in a standard
ball centered on $y_{0}$, disjoint from $K$. The framing at $y_{0}$ is
used to put the Hopf link $H$ in a standard position. 
Let $\mu$ be
the strong marking data with $U_{\mu}$ a ball containing $H$ and
$E_{\mu}$ a bundle with $w_{2}$ dual to an arc joining the two
components of $H$ (as in Lemma~\ref{lem:Hopf-link-Rep}). We consider
$K^{\sharp}=K\cup H$ as a web in $Y$, with strong marking data $\mu$, and so we
can define
\begin{equation}
    \label{eq:general-sharp}
    \begin{aligned}
        \Jsharp(Y, K) &= J(Y, K^{\sharp} ; \mu) \\
    \end{aligned}
\end{equation}
Given a foam cobordism $(X,\Sigma)$ from $(Y_{1}, K_{1})$ to $(Y_{0},
K_{0})$ and a framed arc in $X$ joining the framed basepoints in
$Y_{1}$ and $Y_{2}$, then we can insert $[0,1]\times H$ as a foam in
$X$ along the arc in a standard way, and take standard strong marking
data $\nu = [0,1]\times \mu$ in $X$. In this way, $(X,\Sigma)$ gives
rise to a homomorphism,
\[
            \Jsharp(X,\Sigma) : \Jsharp(Y_{1},K_{1}) \to \Jsharp(Y_{0}, K_{0}).
\]
(We omit the basepoints and arcs from our notation, but this is not
meant to imply that the map $\Jsharp(X,\Sigma)$ is independent of the
choice of arc.)

\begin{definition}
    We define a category $\Cat^{\sharp}$ whose objects are pairs
    $(Y,K,y)$ consisting of a closed, oriented connected $3$-manifold
    $Y$ containing a web $K$ and framed basepoint $y$. The morphisms
    are isomorphism classes of triples $(X,\Sigma, \gamma)$, where $X$
    is a connected oriented $4$-manifold, $\Sigma$ is a foam, and
    $\gamma$ is a framed arc joining the basepoints. Thus $\Jsharp$ is
    a functor, 
\[
          \Jsharp : \Cat^{\sharp} \to (\text{Vector spaces over $\F$}).
\]
by composing ``$\mathrm{Sharp}$'' with $J$. 
\end{definition}

Sometimes we will regard $\Jsharp$ as defining a functor from the
category of webs in $\R^{3}$ and isotopy classes of foams in $[0,1]\times \R^{3}$, by
compactifying $\R^{3}$ and putting the framed basepoint at
infinity. We will then just write $\Jsharp(K)$ etc.~for a web $K$ in
$\R^{3}$.

There is a variant of this definition that we can consider. We can
define
\[
       \Isharp : \Cat^{\sharp} \to (\text{Vector spaces over $\F$})
\]
in much the same way, except that instead of using the marking $\mu$ 
with $U_{\mu}=B^{3}$ as above, we instead use the marking $\mu'$
with $U_{\mu'}$ the whole of the three-manifold $Y$. We still take $E_{\mu'}$ to be a
bundle with $w_{2}$ dual to the same arc. Thus,
\begin{equation}
    \label{eq:general-I-sharp}
    \begin{aligned}
         \Isharp(Y, K) &= J(Y, K\cup H ; \mu') .
    \end{aligned}
\end{equation}
In the case that $K$ has no
vertices, this is precisely the knot invariant $\Isharp(K)$ defined in
\cite{KM-unknot}, except that we are now using $\F$
coefficients. Over $\F$ at least, our definition extends that of
\cite{KM-unknot} by allowing embedded webs rather than
just knots and links.

There is another simple abbreviation to our notation that will be
convenient. It is a straightforward fact that $\Jsharp(S^{3},
\emptyset)$ is $\F$ (see below), so a cobordism from $(S^{3}, \emptyset)$ to
$(Y,K)$ in the category $\Cat^{\sharp}$ determines a vector in
$\Jsharp(Y,K)$. This allows us to adopt the following two conventions:

\begin{definition}\label{def:implied-ball}
    Let $X$ be compact, oriented $4$-manifold with boundary the
    connected $3$-manifold $Y$, and let $\Sigma$ be a foam in $X$ with
    boundary a web $K\subset Y$. Let a basepoint $y\in Y$ be
    given. Then we write $\Jsharp(X,\Sigma)$ for the element of
    $\Jsharp(Y,K)$ corresponding the cobordism from $(S^{3},
    \emptyset)$ to $(Y,K)$ obtained from $(X,\Sigma)$ by removing a
    ball from $X$ and connecting a point on its boundary to $y$ by an
    arc.

    Similarly, given closed, connected manifold $X$ containing a
    closed foam $\Sigma$, we will write $\Jsharp(X,\Sigma)$ for the
    scalar in $\F$ corresponding to the morphism from
    $(S^{3},\emptyset)$ to itself, obtained from $(X,\Sigma)$ by
    removing two balls disjoint from $\Sigma$ and joining them by an
    arc. We refer to $\Jsharp(X,\Sigma)$ in this context as \emph{the
    $\Jsharp$-evaluation of the closed foam.}
\end{definition}

We can also use the language of bifolds directly, rather than
manifolds with embedded webs. We will consider $3$-dimensional bifolds
$\check Y$ with a framed basepoint $y_{0}$ (in the smooth part of
$\check Y$, not an orbifold point). For such a $\check Y$, we write
the corresponding invariant as $\Jsharp(\check Y)$. Given a bifold
cobordism $\check X$ from $\check Y_{1}$ to $\check Y_{0}$ and a
framed arc joining basepoints in these, we obtain a homomorphism
\[
               \Jsharp(\check X) : \Jsharp(\check Y_{1}) \to
               \Jsharp(\check Y_{0}).
\]
We will also allow the conventions of
Definition~\ref{def:implied-ball} in this context, and so talk about
$\Jsharp(\check X)$ for an orbifold with connected boundary, or with
no boundary.

\subsection{Simplest calculations}

We have the following three simple examples of calculation of
$\Jsharp$ for webs $K$ in $\R^{3}$.

\begin{proposition}\label{prop:simple-Jsharp}
   We have the following special cases.
   \begin{enumerate}
   \item If $K$ is the empty web in $\R^{3}$, then
       $\Jsharp(K)=\F$.
   \item If the web $K\subset\R^{3}$ has an embedded bridge, then $\Jsharp(K)=0$.
   \end{enumerate}
\end{proposition}

\begin{proof}
    In both cases, we start by applying
    Lemma~\ref{lem:Hopf-link-Rep}, which tells us that the unperturbed
    set of critical points for the Chern-Simons functional for
    $(S^{3}, K\cup H , \mu)$ can be identified with
     \[
         \Rep^{\sharp}(K) = \{ \,\rho : \pi_{1}(\R^{3}\sminus K) \to \SO(3)
                    \mid \text{$\rho(m_{e})$ has order $2$ for all edges $e$}\,\}.
       \]
     In the first case of the proposition, the representation variety
      is a single point (the trivial representation of the trivial
      group), and this point is non-degenerate as a critical point of
      Chern-Simons. So $\Jsharp$ has rank $1$. 

      In the case of an embedded bridge $e$, the meridian $m_{e}$
       is null-homotopic, so there is no $\rho$ such that
      $\rho(m_{e})$ has order $2$. The representation variety is
      therefore empty, and $\Jsharp(K)=0$.
\end{proof}

\section{Excision}
\label{sec:excision}

\subsection{Floer's excision argument}

Floer's excision theorem \cite{Braam-Donaldson, KM-sutures} can be
applied in the setting of the $\SO(3)$ instanton homology $J$. We
spell out statements of the basic result for cutting and gluing along
tori (which is the original version), as well as cutting an gluing
along $2$-spheres with three orbifold points.

Let $(\check Y, \mu)$ be a $3$-dimensional bifold containing a web with strong marking
data $\mu$. We will temporarily allow the case that $\check Y$ is not
connected, in which case the marking data $\mu$ is strong only if its
restriction to each component of $\check Y$ is strong. Let $T_{1}$, $T_{2}$
be disjoint, oriented, embedded $2$-tori in $U_{\mu}$. (In particular, they lie
in the smooth part of $Y$.) We require
the additional property that the restrictions, $E_{i}$, of the marking
bundle $E_{\mu}$ to the tori $T_{i}$ have non-zero $w_{2}$. Pick
an identification between the tori, $T_{1}\to T_{2}$, and lift it to an
isomorphism between the bundles. By cutting along $T_{1}$, $T_{2}$ and
regluing (using the preferred identifications in an
orientation-preserving way), we obtain a new object $(\check Y' , \mu')$.
In $\check Y'$, we also end up with a pair of embedded tori $T_{1}'$
and $T_{2}'$.

We will be concerned with the case that $T_{1}$ and $T_{2}$ lie in
distinct components $\check Y_{1}$, $\check Y_{2}$ of $\check Y$. In
this case, cutting along $T_{i}$ cuts $\check Y_{i}$ into pieces
$\check Y_{i}^{+}$ and $\check Y_{i}^{-}$. The new bifold $\check Y'$
has components $\check Y'_{1} = \check Y_{1}^{+} \cup \check
Y_{2}^{-}$ etc.

\begin{theorem}[Floer's excision]
In the above setting, when $T_{1}$, $T_{2}$ are separating tori in
distinct components $\check Y_{1}$, $Y_{2}$, we have
\[
                J(\check Y, \mu) = J(\check Y', \mu').
\]
\end{theorem}

\begin{proof}
    The proof is a straightforward adaption of the original argument
    of Floer \cite{Braam-Donaldson}. The case of separating tori was
    given in \cite{KM-unknot}.
\end{proof}

There is a variant of the excision argument, in which the tori
$T_{1}$, $T_{2}$ are replaced by spheres $S_{1}$, $S_{2}$ each with
three orbifold points.  That is, we ask that $S_{1}$, $S_{2}$ are
2-dimensional sub-orbifolds of $Y$, meeting the web $K$ of orbifold
points in three points each. We still require that $S_{i}\setminus K$
lies in $U_{\mu}$, and require the same separation hypotheses. By
cutting and gluing, we form $(Y', \mu')$ just as in the torus
case. The same result holds:

\begin{theorem}\label{thm:excision-street}
In the above setting, when $S_{1}$, $S_{2}$ are separating $3$-pointed spheres in
distinct components $\check Y_{1}$, $Y_{2}$, we have
\[
                J(\check Y, \mu) = J(\check Y', \mu').
\]
\end{theorem}

\begin{proof}
    This adaptation of the excision theorem is contained in
    \cite{Ethan-thesis}, in a version for any number orbifold points on
    the spheres. In the case of $3$ points, the argument from the
    torus case needs no serious adaptation: the essential point is the
    existence of a unique $\SO(3)$ connection on the $2$-dimensional orbifold.
\end{proof}

\begin{remark}
    In the non-separating case, some extra care is needed to obtain
    correct statements, because of the double-covers that correspond
    to possibly different strong marking data.
\end{remark}

\subsection{Applications of excision}

If $\check Y_{1}$ and $\check Y_{2}$ are bifolds with framed
basepoints (at non-orbifold points), 
then there is a preferred connected sum $\check Y = \check Y_{1} \#
\check Y_{2}$ obtained by summing at the basepoints. We can give the
new bifold a preferred framed basepoint, on the $2$-sphere where the sum
is made. Given bifold cobordisms with framed arcs joining the
basepoints, say
\[
             \check X_{i} : \check Y_{i} \to \check Y'_{i} , \qquad (i=1,2),
\]
then we can form a cobordism
\[
               \check X : \check Y \to \check Y
\]
by summing along the embedded arcs. 

\begin{proposition} \label{prop:tensor-product-bifold}
    In the above situation, we have an isomorphism
\[
                     \Jsharp(\check Y) = \Jsharp(\check Y_{1}) \otimes \Jsharp(Y_{2}).
\]
   Furthermore, this isomorphism is natural for the maps corresponding
   to cobordisms $\check X_{i}$ as above, so that
\[
                    \Jsharp(\check X) = \Jsharp(\check X_{1}) \otimes
                    \Jsharp(\check X_{2}).
\]
    The same applies with $\Isharp$ in place of $\Jsharp$.
\end{proposition}

\begin{corollary}
    \label{cor:tensor-product}
     Suppose $K=K_{1}\cup K_{2}$ is a split web in
     $\R^{3}$, meaning that there is an embedded $2$-sphere $S$ which
     separates $K_{1}$ from $K_{2}$. Then there is an isomorphism,
\[
         \Jsharp(K) = \Jsharp(K_{1}) \otimes \Jsharp(K_{2}).
\]
    Moreover, if $\Sigma$ is a split cobordism, meaning that
    $\Sigma=\Sigma_{1}\cup \Sigma_{2}$ and $\Sigma$ is disjoint from
    $[0,1]\times S$, then
\[
         \Jsharp(\Sigma) = \Jsharp(\Sigma_{1})\otimes \Jsharp(\Sigma_{2}).
\] 
    The same applies with $\Isharp$ in place of $\Jsharp$.
\end{corollary}

These excision results will be often applied through the following
corollaries, which we state first in the language of bifolds.  Let
$\check Q$ be a closed, oriented, connected $3$-dimensional
bifold. Let $\check P_{i}$ be $4$-dimensional bifolds with boundary
$\check Q$, for $i=1,\dots, n$. In the notation of
Definition~\ref{def:implied-ball}, these have invariants
\[
             \Jsharp( \check P_{i}) \in \Jsharp(\check Q).
\]
Suppose that we also have a collection of cobordisms
\[
                   \check X_{i} : \check Y \to \check Y', \qquad i=1,\dots,n,
\]
such that $\check X_{i}$ contains in its interior an embedded copy of
$\check P_{i}$. Suppose the complements $\check X_{i} \setminus 
\mathrm{int}(\check P_{i})$ are all identified, in a way that restricts to
the identity on their common boundary components $\check Y$, $\check
Y'$ and $\check Q$.  We then have

\begin{corollary}
    If the elements of $\Jsharp(\check Q)$ defined by the bifolds $\check P_{i}$ satisfy a
    relation
     \[
                        \sum_{i=1}^{n} \Jsharp(\check P_{i}) = 0,
     \]
     then it follows also that
     \[
                           \sum_{i=1}^{n} \Jsharp(\check X_{i}) = 0
     \]
     as maps from $\Jsharp(\check Y)$ to $\Jsharp(\check Y')$.
      The same applies with $\Isharp$ in place of $\Jsharp$.
\end{corollary}

We next restate the last corollary as it applies to the case of webs. 
Let $W$ be a web in a connected $3$-manifold $Q$. For $i=1,\dots, n$, let $P_{i}$ be a
connected, oriented $4$-manifold with boundary $Q$, and let
$V_{i}\subset P_{i}$ be a foam with boundary $W$.  These determine
elements 
\[
          \Jsharp(P_{i},V_{i}) \in \Jsharp (Q,  W).
\]
Suppose we also have a collection of foam cobordisms
\[
             (X_{i}, \Sigma_{i}) : (Y, K) \to (Y',K').
\]
Suppose that each $(X_{i}, \Sigma_{i})$ contains an interior copy of
$(P_{i}, V_{i})$, and that the complements are all identical. 
We then have the following restatement:

\begin{corollary}\label{cor:local-relations-webs}
    If $\sum \Jsharp(P_{i},V_{i}) = 0$, then $\sum \Jsharp(X_{i},\Sigma_{i}) = 0$, as maps
    from $\Jsharp(Y,K) \to \Jsharp(Y',K')$. The same holds with $\Isharp$
    in place of $\Jsharp$.
\end{corollary}

\section{Calculations}

\subsection{The unknot and the sphere with dots}

\begin{proposition}\label{prop:unknot-u-map}
    For the unknotted circle, the homology $\Jsharp(K)$ has rank
    $3$. Equipped  with the
    operator $u_{e} : \Jsharp(K) \to \Jsharp(K)$ (see
    Definition~\ref{def:u-map}), it is isomorphic to
    $\F[u]/u^{3}$ as a module over the polynomial algebra. 
\end{proposition}

\begin{proof}
    The representation variety
      $\Rep^{\sharp}(K)$ is a copy of $\RP^{2}$ (as in 
       section~\ref{sec:examples-Rep}). We can choose a small,
      good perturbation $\pi$ so that $f_{\pi}$ restricts to a
      standard Morse function on $\RP^{2}$, leading to three critical
      points $\alpha_{0}$, $\alpha_{1}$ and $\alpha_{2}$.  For the
      special case of the unknot the Floer homology is $\Z/4$ graded,
      and the only $1$-dimensional moduli spaces are therefore
      $M_{1}(\alpha_{2},\alpha_{1})$ and
      $M_{1}(\alpha_{1},\alpha_{0})$. These moduli spaces approximate
      the Morse trajectories on $\RP^{2}$, so $\Jsharp(K) =
      H_{*}(\RP^{2} ; \F)$, which has rank $3$.

     For the module structure, a dot $\delta$ on the knot, the
    double-cover which corresponds to the line bundle $L_{\delta}$ is
    the non-trivial double cover of $\RP^{2}$. The calculation only involves Morse
    trajectories on $\RP^{2}$, so the result is the same as for a
    calculation of the cap product with the one-dimensional cohomology
    class acting on the homology of $\RP^{2}$.
\end{proof}

As a simple application of Corollary~\ref{cor:local-relations-webs}, we
have the following result.

\begin{proposition}\label{prop:u-cubed}
    Let $K\subset Y$ be a web, let $e$ be an edge of $K$, and $u_{e}$
    the corresponding operator on $\Jsharp(Y,K)$. (See
    Definition~\ref{def:u-map}.) Then $u_{e}^{3}=0$.
\end{proposition}

\begin{proof}
    We apply Corollary~\ref{cor:local-relations-webs}, to the case
    that $M=B^{3}$, the web $W$ is a single arc in $B^{3}$, and
    $(N,V)$ is a disk with $3$ dots. We see that it is
    sufficient to check that a disk with three dots defines the zero
    element in $\Jsharp(K)$ when $K$ is the unknot. 
    This in turn follows from Proposition~\ref{prop:unknot-u-map}.
\end{proof}

The calculation of $\Jsharp$ with its module structure for the unknot
$K$ is closely related to the evaluation of the closed foam
$S(k)$, consisting of an unknotted $2$-sphere $S$ in $\R^{4}$ or $S^{4}$,
with $k$ dots. Before continuing, it will be useful to make some
general remarks about closed foams.

A closed foam $\Sigma$ in $S^{4}$ evaluates to $0$ or $1$,
\[
           \Jsharp(\Sigma) \in \F
\]
by the convention of Definition~\ref{def:implied-ball}. Because the definition involves counting solutions in
$0$-dimensional moduli spaces, we can read off the action $\kappa$ of the relevant
solutions from the dimension formula \eqref{eq:dimension}. Taking
account of the codimension-$1$ constraints corresponding to the dots on
$\Sigma$, and writing $k$ for the number of dots, we must have
\[
 8 \kappa  +
        \chi(\Sigma) + \frac{1}{2} \Sigma\cdot\Sigma -
        \frac{1}{2}|\tau| - k = 0.
\]
This gives us some necessary conditions for a non-zero evaluation.
Since $\kappa$ is non-negative, we must have
\begin{equation}\label{eq:action-inequality}
              k \ge \chi(\Sigma) + \frac{1}{2} \Sigma\cdot\Sigma -
        \frac{1}{2}|\tau|.
\end{equation}

 We evaluate $\Jsharp(\Sigma)$ for a sphere with dots.

\begin{proposition}\label{prop:sphere-with-dots}
    Let $S(k)$ denote the unknotted $2$-sphere with $k$ dots, as a foam
    in $\R^{4}$ or $S^4$. Then $\Jsharp(S(k))=1$ if $k=2$ and is $0$ otherwise.
\end{proposition}

\begin{proof}
    The fact that $\Jsharp(S(k))=0$ for $k\ge 3$ follows from
    Proposition~\ref{prop:u-cubed}. The inequality
    \eqref{eq:action-inequality} requires $k \ge 2$ in this case. So
    $S(2)$ is the only case where the evaluation can be
    non-zero. Furthermore, for the case $k=2$, the action $\kappa$ for
    the relevant moduli spaces is zero: we are looking for flat
    connections. As in the case of the unknot considered earlier, the
    representation variety of flat connections is $\RP^{2}$, and the
    real line bundle over $\RP^{2}$ corresponding to each dot is the
    non-trivial one. So $\Jsharp(S(2))$ is obtained by evaluating the
    square of the basic $1$-dimensional class on $\RP^{2}$, and the
    answer is $1$ as claimed.
\end{proof}

If $D(k)$ is a $2$-disk in $B^{4}$ with boundary a standard circle in
$S^{3}$, regarded as a foam with $k$ dots, then a formal corollary of
the above proposition is:

\begin{corollary}\label{cor:disk-to-unknot}
    The elements $\Jsharp(D(0))$, $\Jsharp(D(1)$ and $\Jsharp(D(2))$
    are a basis for $\Jsharp(K)$, where $K\subset S^{3}$ is a standard
    circle (the unknot).
\end{corollary}

Using Corollary~\ref{cor:tensor-product}, we obtain from this a simple
criterion for testing whether an element of the Floer homology of an
unlink is zero:

\begin{lemma}\label{lem:unlink-test}
    Let $(S^{3}, K_{n})$ be an $n$-component unlink in the
    $3$-sphere. Let $(X,\Sigma_{i})$ be webs with boundary $(S^{3},
    K_{n})$. Then, in order that a relation
   \[
            \sum_{i} \Jsharp(X,\Sigma_{i})=0
   \]
   holds in $\Jsharp(S^{3}, K_{n})$, it is necessary and sufficient to
   test, for all $n$-tuples $(k_{1}, \dots, k_{n})\in \{0,1,2\}^{n}$,
   that the relation
   \[
          \sum_{i} \Jsharp(\bar X, \bar \Sigma_{i})=0,
   \]
   holds in $\F$, where $\bar X$ is obtained from $X$ by adding a
   ball, and $\bar\Sigma_{i}$ is obtained from $\Sigma_{i}$ by adding
   a standard disk with $k_{m}$ dots, $D(k_{m})\subset B^{4}$, to the
   $m$'th component of $K_{n}$.
\end{lemma}

\subsection{The theta foam and the theta web}

A \emph{theta foam} is a closed foam formed from three disks in
$\R^{3}$ meeting along a common circle, the seam. We write $\Theta$
for a typical theta-foam in $\R^{4}$, and $\Theta(k_{1}, k_{2},
k_{3})$ for the same foam with the addition of $k_{i}$ dots on the
$i$'th disk.

\begin{proposition}\label{prop:theta-evaluation}
    For the theta foam with dots, $\Theta(k_{1}, k_{2}, k_{3})$, we
    have
    \[
                  \Theta(k_{1}, k_{2}, k_{3}) = 1
     \]
    if $(k_{1}, k_{2}, k_{3}) = (0,1,2)$ or some permutation
    thereof. In all other cases the evaluation is zero. 
\end{proposition}

\begin{proof}
    In the case of the theta foam in $\R^{4}$, any solution of the
    equation has action satisfying $4\kappa\in\Z$. One can verify this
    by considering a bifold connection on $(S^{4},\Theta)$: any such
    gives rise to an ordinary $\SO(3)$ action of the $V_{4}$-cover
    $S^{4}$, where the action must be an integer.
    The relation in Proposition~\ref{prop:dot-migration} tells us that
    \[
            \Jsharp( \Theta(k_{1}, k_{2}, k_{3})) = \Jsharp( \Theta(k_{1}-1, k_{2}+1,
             k_{3})) +  \Jsharp(\Theta(k_{1}-1, k_{2}, k_{3}+1))
    \]
    as long as $k_{1}\ne 0$. So the evaluation for 
   $\Theta(k_{1}, k_{2}, k_{3})$ will be
    completely determined once we know the evaluations for the cases
    with $k_{1}=0$. So we examine $\Theta(0, k_{2}, k_{3})$. For a
    non-zero evaluation, we require $k_{i} \le 2$, because of
    Proposition~\ref{prop:u-cubed}. The inequality
    \eqref{eq:action-inequality} requires $k_{2}+k_{3}\ge 3$ in this
    case. Furthermore, $k_{2}+k_{3}$ must be odd, because $4\kappa$ is an
    integer. These constraints mean we need only
    look at $\Theta(0,1,2)$, and here the moduli space we are
    concerned with is  a moduli space of flat connections. The
    representation variety in this example is the flag manifold
    $F=\SO(3)/V_{4}$, and the three real line bundles corresponding to
    dots on the three sheets are the three tautological line bundles
    $L_{1}$, $L_{2}$, $L_{3}$ on $F$. The evaluation of
    $\Theta(0,1,2)$ is therefore equal to
     \[
                  \bigl\langle w_{1}(L_{2}) w_{1}(L_{3})^{2} , [F] \bigr\rangle,
     \]
    which is $1$.
\end{proof}

Let $K$ be a theta web. Corresponding to the three edges, there are
three operators $u_{i} : \Jsharp(K) \to \Jsharp(K)$.

\begin{proposition}
    For the theta web $K$, the instanton homology $\Jsharp(K)$ can be
    identified with the ordinary $\F$ homology of the flag manifold $F=\SO(3)/V_{4}$
    in such a way that the operators $u_{i}$ correspond to the
    operation of cap product with the classes $w_{1}(L_{i})$, where
    $L_{1}$, $L_{2}$, $L_{3}$ are the tautological line bundle on the
    flag manifold. Concretely, this means that the dimension is $6$
    and that the instanton homology is a cyclic module over the
    algebra generated by the $u_{i}$ which satisfy the relations
\[
\begin{gathered}
    u_{1} + u_{2} + u_{3} = 0 \\
         u_{1}u_{2} + u_{2}u_{3} + u_{3}u_{1}=0 \\
                u_{1}u_{2}u_{3}=0.
\end{gathered}
\]
\end{proposition}

\begin{proof}
    The representation variety for the theta web is the flag manifold,
    so the rank of the instanton homology is at most $6$. On the other
    hand, we can see that the rank is at least $6$ as
    follows. Consider a theta-foam cut into two pieces by a
    hyperplane, so as to have cobordisms $\Theta_{-}$ from
    $\emptyset$ to $K$ and $\Theta_{+}$ from $K$ to
    $\emptyset$. Putting $k_{i}$ dots on the $i$'th facet of
    $\Theta_{\pm}$ and applying $\Jsharp$, we obtain vectors
    $x_{-}(k_{1}, k_{2}, k_{3})$ in $\Jsharp(K)$ and covectors
    $x_{+}(k_{1}, k_{2}, k_{3})$. The pairings between these can be
    evaluated using the knowledge of the closed theta foam
    (Proposition~\ref{prop:theta-evaluation}). If we restrict to the
    elements $x_{\pm}(0, k_{1}, k_{2})$ with $k_{1}\le 1$ and
    $k_{2}\le 2$, then the resulting matrix of pairings is
    non-singular. So the corresponding $6$ elements $x_{-}(0, k_{1},
    k_{2})$ in $\Jsharp(K)$ are independent. The relations satisfied
    by the operators $u_{i}$ can be read off similarly.
\end{proof}

As a corollary of the above proposition and excision, we have:

\begin{proposition}\label{prop:vertex-relations}
    Let $K$ be any web in a $3$-manifold $Y$, and let $u_{1}$, $u_{2}$,
    $u_{3}$ be the operators corresponding to three edges of $K$ that
    are incident at a common vertex. Then the operators satisfy the
    same relations,
\begin{equation}\label{eq:vertex-relations}
\begin{gathered}
    u_{1} + u_{2} + u_{3} = 0 \\
         u_{1}u_{2} + u_{2}u_{3} + u_{3}u_{1}=0 \\
                u_{1}u_{2}u_{3}=0.
\end{gathered}
\end{equation}
    In particular, any monomial $u_{1}^{a}u_{2}^{b}u_{3}^{c}$ of total
    degree $4$ or more is zero.
\end{proposition}

\subsection{The tetrahedron web and its suspension}

\begin{figure}
    \begin{center}
        \includegraphics[scale=0.50]{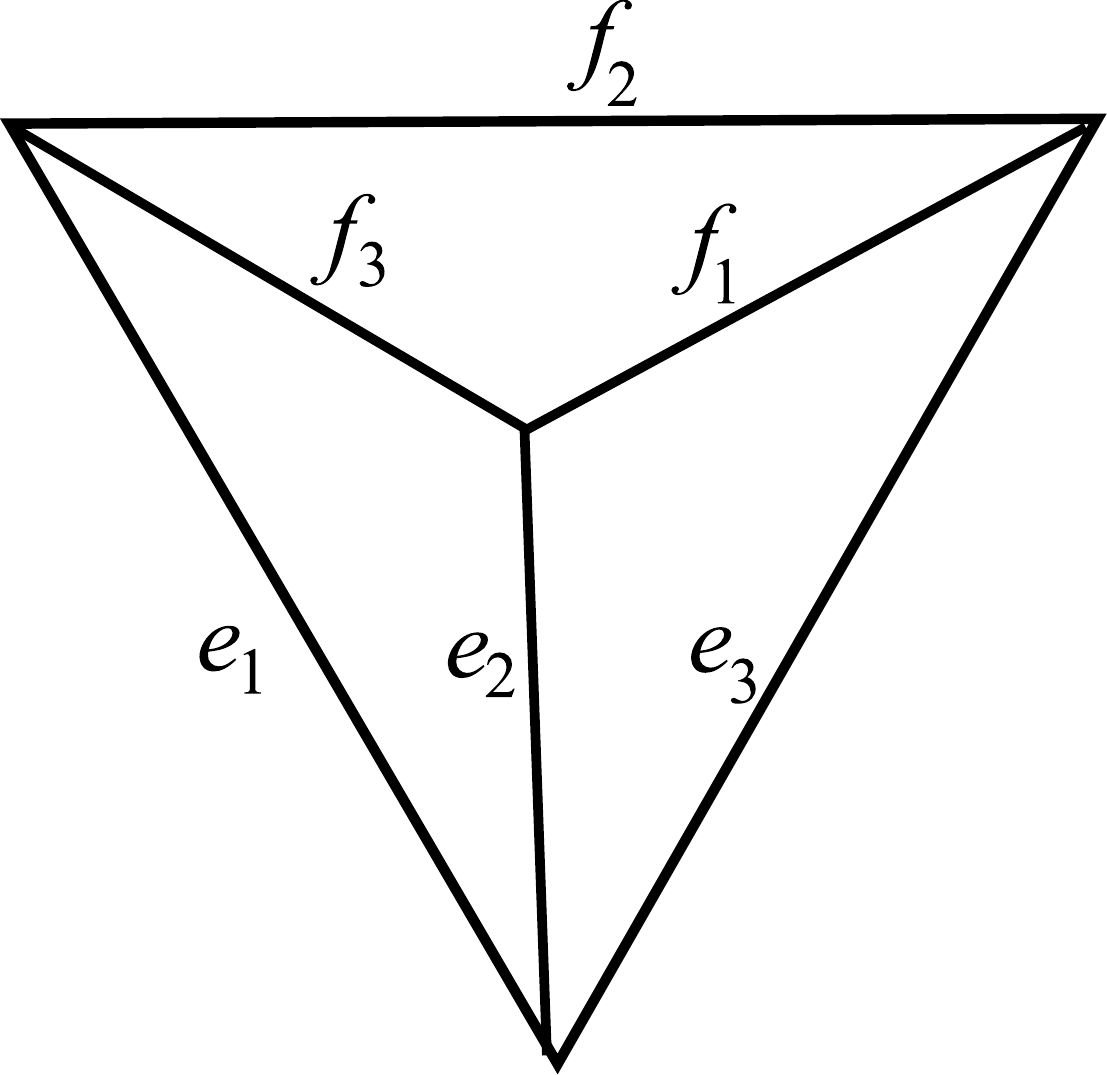}
    \end{center}
    \caption{\label{fig:tetrahedron-web}
   The tetrahedron.}
\end{figure}

Let $K$ be the tetrahedron web: the graph drawn in
Figure~\ref{fig:tetrahedron-web}, with the edges labeled as shown. Let
$T\subset B^{4}$ be the cone on $K$, regarded as a foam with one
tetrahedral point. Let $E_{i}$ and $F_{i}$ be the facets of $T$
corresponding to the edges $e_{i}$ and $f_{i}$ of $K$. Let $T(k_{1},
k_{2}, k_{3})$ denote the foam $T$ with the addition of $k_{i}$ dots
on the facet $E_{i}$. Let $S\subset S^{4}$ be the double of $T$, the
suspension of $K$ with two tetrahedral points. Let $S(k_{1}, k_{2},
k_{3})$ be defined similarly.

We can evaluate the closed foams $S(k_{1}, k_{2}, k_{3})$ just as we
did for the theta foam, and the answers are the same, as the next
proposition states.

\begin{proposition}\label{prop:S-evaluation}
    For the foam $S(k_{1}, k_{2}, k_{3})$, we
    have
    \[
                  S(k_{1}, k_{2}, k_{3}) = 1
     \]
    if $(k_{1}, k_{2}, k_{3}) = (0,1,2)$ or some permutation
    thereof. In all other cases the evaluation is zero. 
\end{proposition}

\begin{proof}
    In the case that $k=k_{1}+k_{2}+k_{3}$ is less than $3$, we obtain
    $0$. In the case that $k=3$, we obtain the evaluation of the
    corresponding ordinary cohomology class on the flag manifold, as
    in the case of the theta foam. When $k\ge 4$, the evaluation is
    zero, by an application of excision to reduce to the case of the
    theta foam.
\end{proof}

\begin{proposition}\label{prop:tetrahedron-calc}
    The instanton homology $\Jsharp(K)$ has dimension $6$. It is
    generated by the elements $\Jsharp(T(0, k_{2}, k_{3}))$ with $0\le
    k_{2}\le 1$ and $0\le k_{3}\le 2$. The operators $u_{1}$, $u_{2}$,
    $u_{3}$ corresponding to the edges $e_{1}$, $e_{2}$, $e_{3}$
    satisfy the relations of the cohomology of the flag manifold
    \eqref{eq:vertex-relations} and $\Jsharp(K)$ is a cyclic module
    over the corresponding polynomial algebra, with generator $\Jsharp(T)$.
\end{proposition}

\begin{proof}
   The representation variety of $K^{\sharp}$ is again the flag
    manifold, so the dimension of $\Jsharp(K$ is at most $6$.
    On the other hand, we can compute the pairings of the classes
    $\Jsharp(T(k_{1}, k_{2}, k_{3})$ with all their doubles, using our
    knowledge of $S(k_{1}, k_{2}, k_{3})$, and from this we see that
    the six elements $\Jsharp(0, k_{2}, k_{3})$ for $k_{2}$ and
    $k_{3}$ in the given range are independent. So the rank is exactly
    $6$. 
\end{proof}

To complete our description of $\Jsharp(K)$, we must describe the
operators corresponding to the other three edges, $f_{1}$, $f_{2}$,
$f_{3}$.

\begin{proposition}
    The operators $v_{1}$, $v_{2}$, $v_{3}$ on $\Jsharp(K)$
    corresponding to the edges  $f_{1}$, $f_{2}$,
$f_{3}$ are equal to $u_{1}$, $u_{2}$, $u_{3}$ respectively. 
\end{proposition}

\begin{proof}
    We will show that $v_{1}=u_{1}$.
    Let $S_{1}$ denote the foam $S$ equipped with a single dot on the
    facet corresponding to $f_{1}$. Let $S_{1}(k_{1}, k_{2}, k_{3})$
    denote the same foam with addition dots on the facets
    corresponding to the edges $e_{i}$. It will suffice to show that 
    \[
                  S_{1}(0, k_{2}, k_{3}) = S(1, k_{2}, k_{3})
    \]
   for all values of $k_{2}$ and $k_{3}$. As usual if
   $1+k_{2}+k_{3}\le 2$, then the evaluation is zero, while if
   $1+k_{2}+k_{3}=3$ we are evaluating an ordinary cohomology class on
   the flag manifold. Since the line bundle corresponding to a dot on
   $f_{1}$ is the same as the line bundle corresponding to a dot on
   $e_{1}$, equality does hold when $1+k_{2}+k_{3} = 3$. The only case
   still in doubt is when $1+k_{2}+k_{3}=4$, and the only potentially
   non-zero evaluations here are $S_{1}(0, 1, 2)$ or $S_{1}(0, 2, 1)$
   (which are certainly equal, by symmetry). We must show that
   $S_{1}(0, 1, 2)=0$. 

   Suppose instead that $S_{1}(0,1,2)=1$. With this complete
   information about $S_{1}(0, k_{1}, k_{2})$ we can evaluate the
   matrix elements of $v_{1}$ on our standard basis, and we find that
   $v_{1}+u_{1}=1$ as operators on $\Jsharp(K)$. But $u_{1}$ and
   $v_{1}$ are both nilpotent because $u_{1}^{3}=v_{1}^{3}=0$, so this
   is impossible.
\end{proof}

\subsection{Upper bounds for some prisms}

Let $L_{n}$ be the planar trivalent graph formed from two concentric
$n$-gons with edges joining each vertex of the outer $n$-gon to the
corresponding vertex of the inner $n$-gon. We call $L_{n}$ the
$n$-sided prism. Later we will be able to determine $\Jsharp(L_{n})$
entirely; but for now we can obtain upper bounds on the rank of
$L_{n}$ for small $n$. These upper bounds will be used later, in
section~\ref{sec:relations}, after which we will have the tools to
show that these upper bounds are exact values.

\begin{lemma}\label{lem:upper-bounds}
    The dimension of $\Jsharp(L_{n})$ is at most $12$ for $n=2$, at
    most $6$ for $n=3$, and at most $24$ for $n=4$.
\end{lemma}

\begin{proof}
    For $n=2$ and $n=3$, the representation varieties
    $\Rep^{\sharp}(L_{n})$ consists of $2$ (respectively, $1$) copy of
    the flag manifold $\SO(3)/V_{4}$. These are the orbits of the flat
    connections with image $V_{4}$, of which there are two distinct
    classes in the case of $L_{2}$. In both cases, the Chern-Simons
    function is Morse-Bott, so an upper bound for the ranks of
    $\Jsharp(L_{n})$ are the dimension of the ordinary homology of the
    representation variety, which are $12$ and $6$ respectively.

   For the case $n=4$, the representation variety is not
   Morse-Bott. The un-based representation variety $\Rep(L_{4})$
   consists of three arcs with one endpoint in common. The common
   point and the three other endpoints of the arc correspond to
   $V_{4}$ representations in $\Rep(L_{4})$. The interiors of the
   three arcs parametrize flat connections with image contained in
   $O(2)$ but not in $V_{4}$. Using an explicit holonomy perturbation,
   one can perturb the Chern-Simons function by a small function $f$
   which has critical points at the common endpoint and the three
   other endpoints of the arc, but whose restriction to the interiors
   of the three arcs has no critical points. After adding this
   perturbation, the critical set in $\Rep^{\sharp}(L_{4})$ is
   Morse-Bott and comprises
   four copies of $\SO(3)/V_{4}$. We omit the details for this
   explicit construction, because there is a more robust argument
   based on the exact triangle satisfied by $\Jsharp$ which will be
   treated in a subsequent paper \cite{KM-jsharp-triangles}.
\end{proof}

\section{Relations}
\label{sec:relations}

Sections \ref{subsec:neck-cutting}--\ref{subsec:square-reln}
establish properties of $\Jsharp$ that aid calculation. With the
exception of the triangle relation
(section~\ref{subsec:triangle-reln}), these properties (and some of their
proofs) are motivated by \cite{Khovanov-sl3}, where similar properties
of Khovanov's $\sl_{3}$ homology are either proved or (in the case of
the neck-cutting relation) taken as axioms. 

\subsection{The neck-cutting relation}
\label{subsec:neck-cutting}

The following proposition describes what happens when  a foam is
surgered along an embedded disk. The statement of the proposition is
illustrated in Figure~\ref{fig:neck-cutting}.
\begin{figure}
    \begin{center}
        \includegraphics[scale=0.40]{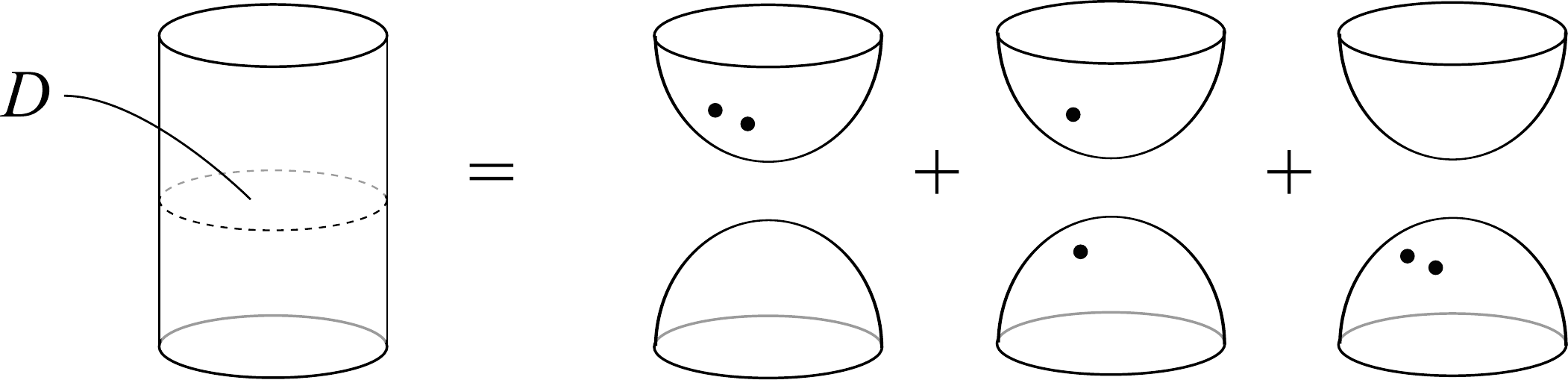}
    \end{center}
    \caption{\label{fig:neck-cutting}
   The neck-cutting relation, {\itshape cf.}~\cite{Khovanov-sl3}. The
   disk $D$ is not part of the original foam.}
\end{figure}
 See also \cite{Khovanov-sl3}.

\begin{proposition}[Neck-cutting]\label{prop:neck-cutting}
    Let $(X,\Sigma)$ be a cobordism defining a morphism in
    $\Cat^{\sharp}$. Suppose that $X$ contains an embedded disk $D$
    whose boundary lies in the interior of a facet of $\Sigma$, which
    it meets transversely. Suppose that the trivialization of the
    normal bundle to $D$ at the boundary which $\Sigma$ determines
    extends to a trivialization over the disk. Let $\Sigma'$ be the
    foam obtained by surgering $\Sigma$ along $D$, replacing the
    annular neighborhood of $\partial D$ in $\Sigma$ with two parallel
    copies of $D$. Let
    $\Sigma'(k_{1}, k_{2})$ be obtained from $\Sigma'$ by adding
    $k_{i}$ dots to the $i$'th copy of $D$. Then 
    \[
        \Jsharp(X,\Sigma) = \Jsharp(X,\Sigma'(0,2)) + \Jsharp(X,\Sigma'(1,1)) + \Jsharp(X,\Sigma'(2,0)). 
     \]
\end{proposition}

\begin{proof}
    Using the excision principle,
    Corollary~\ref{cor:local-relations-webs}, we see that it is enough
    to test this in the local case that $X$ is a  $4$-ball,
    $\Sigma$ is a standard annulus, $\Sigma'$ is a pair of disks, and
    we regard $(X,\Sigma)$ etc.~as cobordisms from $\emptyset$ to the
    unlink of $2$ components. Thus the relation to be proved is a
    relation in $\Jsharp(K_{2})$, where $K_{2}\subset S^{3}$ is the
    $2$-component unlink. To prove this relation in $\Jsharp(K_{2})$, we
    use Lemma~\ref{lem:unlink-test}. This reduces the question to the
    $\Jsharp$-evaluation of some closed foams in the $4$-sphere: the
    foams all consist of spheres $S$ with a certain number of dots, and
    the relation to be proved is
     \begin{multline*}
          \Jsharp(S(k_{1}+k_{2})) = \Jsharp(S(k_{1}))\Jsharp(S(k_{2}+2))
  \\ +  \Jsharp(S(k_{1}+1))\Jsharp(S(k_{2}+1))  +
\Jsharp(S(k_{1}+2))\Jsharp(S(k_{2}))    . 
\end{multline*}
   This relation follows directly from Proposition~\ref{prop:sphere-with-dots} 
\end{proof}

Using neck-cutting, we can evaluate closed surfaces, as long as they
are standardly embedded.

\begin{proposition}
    Let $\Sigma_{g}\subset \R^{4}$ be a genus-$g$ surface embedded as
    the boundary of a standard handlebody in $\R^{3}$. Then
    $\Jsharp(\Sigma_{1})=1$ and $\Jsharp(\Sigma_{g})=0$ for $g\ge 2$.
    If $g\ge 1$ and $\Sigma_{g}$ has one or more dots, then its
    evaluation is zero.
\end{proposition}

\begin{proof}
    Use the neck-cutting relation to reduce to the case of a sphere
    with dots.
\end{proof}

As in \cite{Khovanov-sl3}, one can use the evaluation of spheres and
theta foams together with the neck-cutting relation to obtain other
relations. In particular:

\begin{proposition}[Bubble-bursting]\label{prop:bubble-bursting}
    Let $D$ be an embedded disk in the interior of a facet of a foam
    $\Sigma\subset X$. Let $\gamma$ be the boundary of $D$, and let
    $\Sigma'$ be the foam $\Sigma' = \Sigma\cup D'$, where $D'$ is a
    second disk meeting $\Sigma$ along the circle $\gamma$, so that
    $D\cup D'$ bounds a $3$-ball. Let $\Sigma'(k)$ denote $\Sigma'$ with
    $k$ dots on $D'$. Let $\Sigma(k)$ denote $\Sigma$ with $k$ dots on
    $D$. Then we have
     \[
               \Jsharp(\Sigma'(k)) = \Jsharp(\Sigma(k-1))
     \]
      for $k=1$ or $2$, and $\Jsharp(\Sigma'(k)=0$ otherwise. 
\end{proposition}

\begin{proof}
    Let $\gamma_{1}$ be a circle parallel to $\gamma$ in
    $\Sigma\setminus D$. Apply the neck-cutting relation to the
    surgery of $\Sigma'$ along a disk with boundary $\gamma'$. The
    surgered foam is the union of a foam isotopic to $\Sigma$ with
    theta foam. 
\end{proof}

Instead of adding a disk $D'$ to $\Sigma$, we can add a standard
genus-$1$ surface $T$ with boundary $\gamma$. The resulting foam
$\Sigma\cup T$ was considered earlier in
section~\ref{subsec:dots}. Let $\Sigma(k)$ again denote $\Sigma$ with
$k$ dots in the disk. Then we have

\begin{proposition}\label{prop:tori-as-dots}
    In the above situation, $\Jsharp(\Sigma\cup T) = \Jsharp(\Sigma(1))$.
\end{proposition}

\begin{proof}
    By two surgeries, the foam $\Sigma\cup T$ becomes a split union of
    $\Sigma$, a theta foam and a standard torus. Apply the
    neck-cutting relation and the known evaluations for theta foams
    and tori.
\end{proof}

The proposition above justifies the alternative description of how to
incorporate dots on foams, from section~\ref{subsec:dots}.

\subsection{The bigon relation}

Let $K\subset Y$ be a web containing a bigon: a pair of edges
spanning standard disk in $Y$, and let $K'$ be obtained from $K$ by
collapsing the bigon to a single edge. (See
Figure~\ref{fig:bigon-maps}.)  

\begin{proposition}
    The dimension of $\Jsharp(K)$ is twice the dimension of $\Jsharp(K')$.
\end{proposition}

\begin{proof}
The proof mirrors Khovanov's argument in \cite{Khovanov-sl3}.
There four standard morphisms,
\[
\begin{aligned}
    A, B : K' &\to K \\
    C, D : K &\to K'
\end{aligned}
\]
\begin{figure}
    \begin{center}
        \includegraphics[scale=0.50]{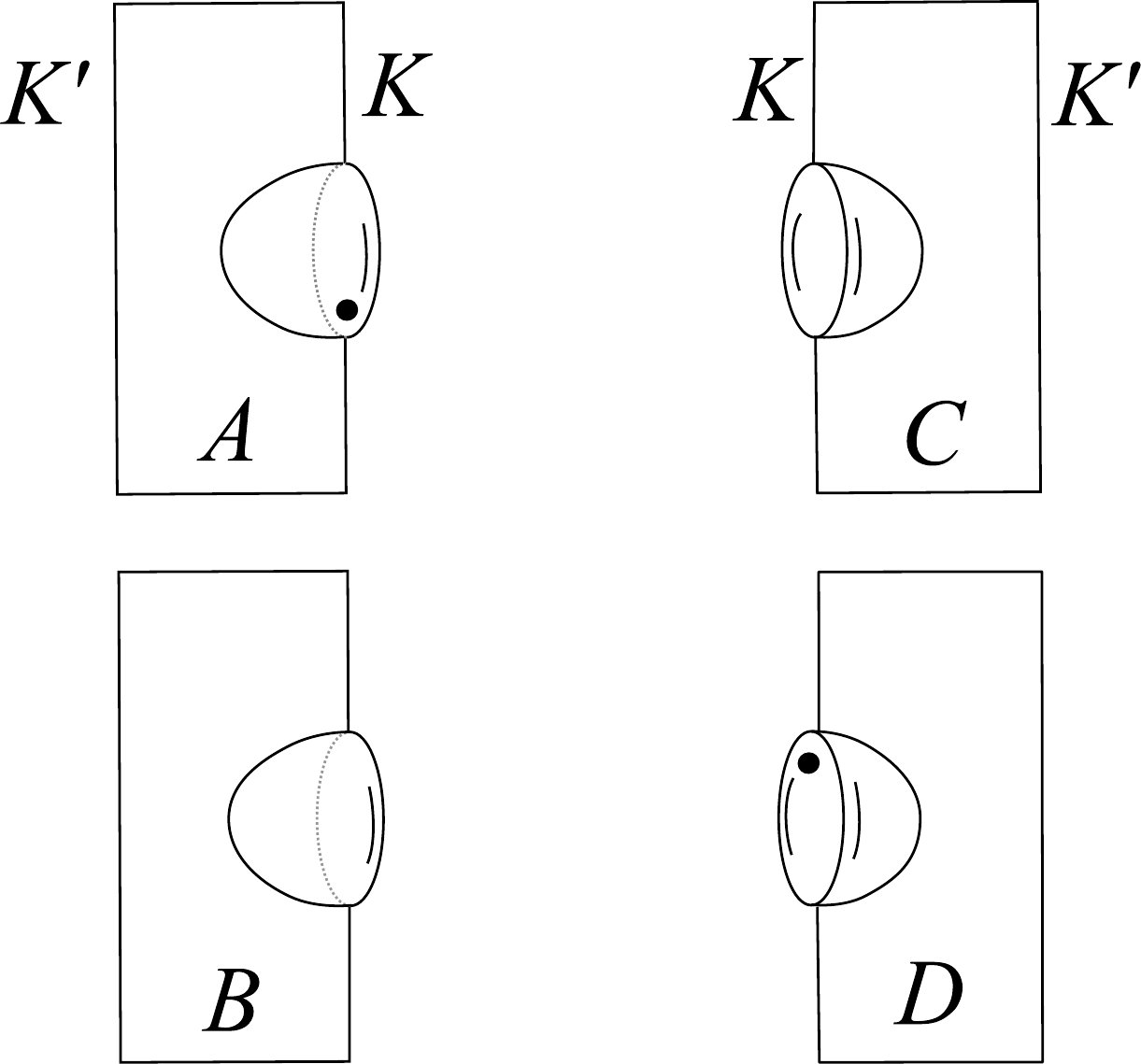}
    \end{center}
    \caption{\label{fig:bigon-maps}
   Four morphisms for the bigon relation.}
\end{figure}
as shown in Figure~\ref{fig:bigon-maps}, and these give rise to maps
\[
\begin{aligned}
   a, b : \Jsharp(K') &\to \Jsharp(K) \\
    c,d :  \Jsharp(K) &\to \Jsharp(K').
\end{aligned}
\]
The composite foams $CA$ and $DB$ are foams to which the
bubble-bursting relation applies, from which we see that $ca=1$ and
$db=1$. Similarly, $cb=da=0$. From this it follows that $a \oplus b$
maps $\Jsharp(K')\oplus\Jsharp(K')$ injectively into $\Jsharp(K)$ and
that
\[
      ac + bd : \Jsharp(K) \to \Jsharp(K)
\]
is a projection on the image of $a \oplus b$. If we can show that \[ac
+ bd=1,\] then we will be done. By the excision principle in the form
of Corollary~\ref{cor:local-relations-webs}, this is equivalent to
checking a relation
\begin{figure}
    \begin{center}
        \includegraphics[scale=0.4]{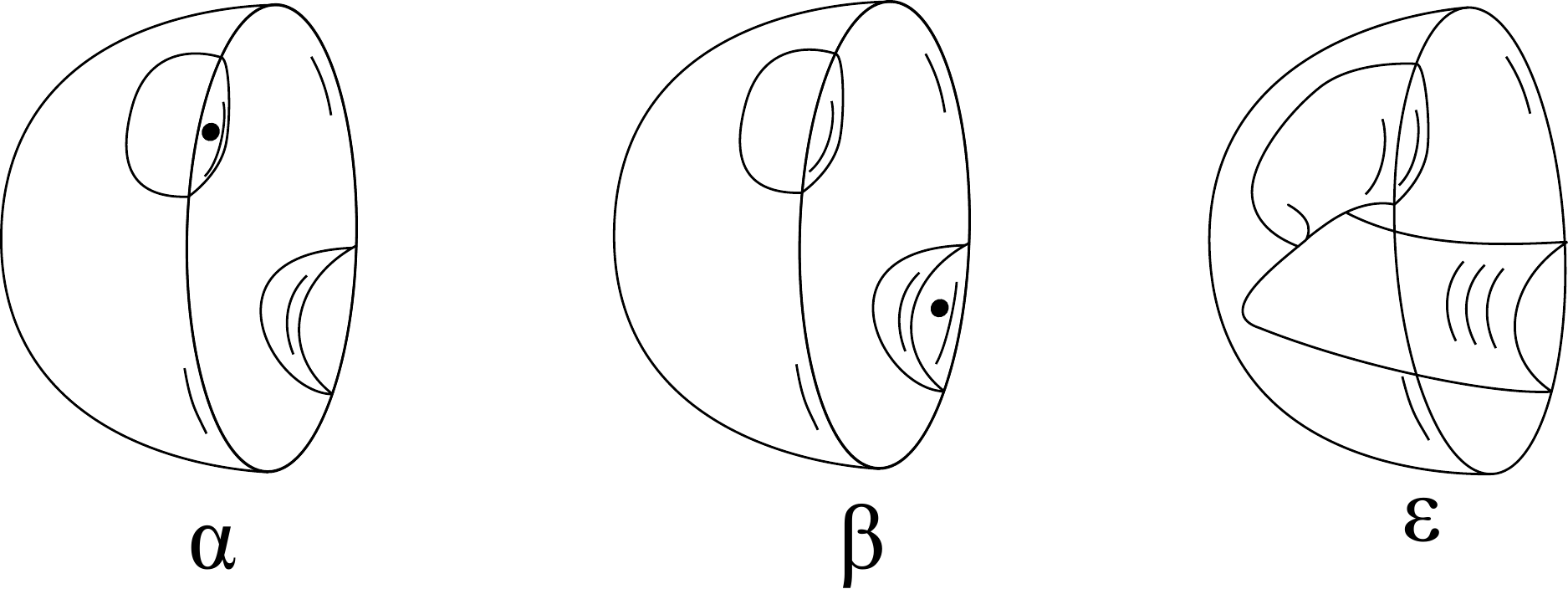}
    \end{center}
    \caption{\label{fig:bigon-composite-foams}
   The foams $\alpha$, $\beta$ and $\epsilon$.}
\end{figure}
\[
                \alpha + \beta = \epsilon
\]
in $\Jsharp(L_{2})$, where $\alpha$, $\beta$ and $\epsilon$ are the
dotted foams show in Figure~\ref{fig:bigon-composite-foams}.

In the special case that $K'$ is the theta web and $K$ is a prism
$L_{2}$, we know from Lemma~\ref{lem:upper-bounds} that $\Jsharp(K)$
has dimension at most $12$, which is twice the dimension of
$\Jsharp(K')$. So equality must hold, and in this special case we have
$ac + bd=1$. Applying this to the element
$\epsilon\in\Jsharp(K)=\Jsharp(L_{2})$, we obtain
\[
            ac(\epsilon) + bd(\epsilon)=\epsilon.
\]
\begin{figure}
    \begin{center}
        \includegraphics[scale=0.4]{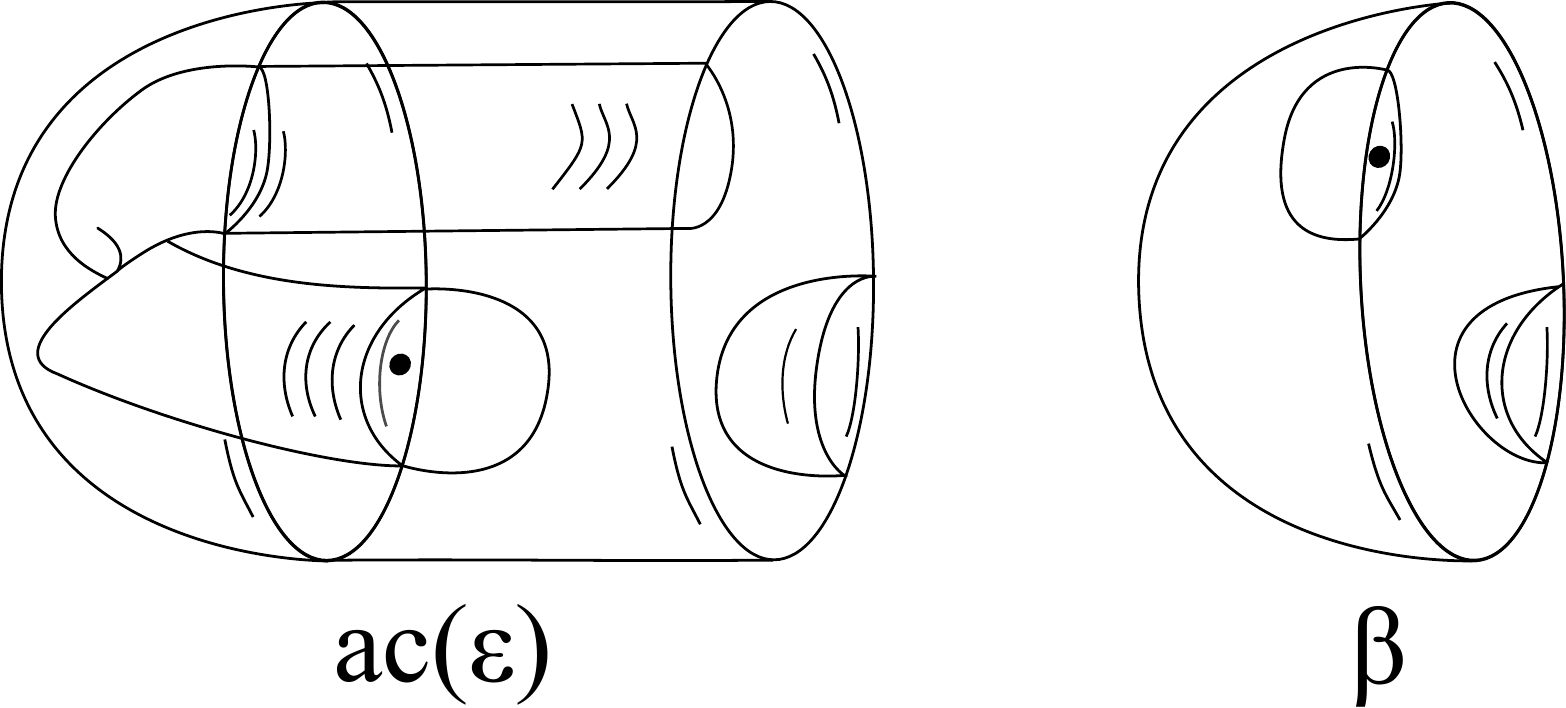}
    \end{center}
    \caption{\label{fig:ac-isotopy}
   An isotopy of foams showing that $ac(\epsilon)=\alpha$.}
\end{figure}
By simple isotopies, we see that $ac(\epsilon)=\alpha$ (see Figure~\ref{fig:ac-isotopy}) and
$bd(\epsilon)=\beta$, so the result follows.
\end{proof}

\subsection{The triangle relation}
\label{subsec:triangle-reln}

Let $K$ be a web containing three edges which form a triangle lying 
in a standard $2$-disk in $Y$. and $v$ a vertex of $K'$. Let $K'$ be obtained from
$K$ by collapsing the triangle to a single vertex, following the standard planar model shown in
Figure~\ref{fig:triangle-move}. 
\begin{figure}
    \begin{center}
        \includegraphics[scale=0.4]{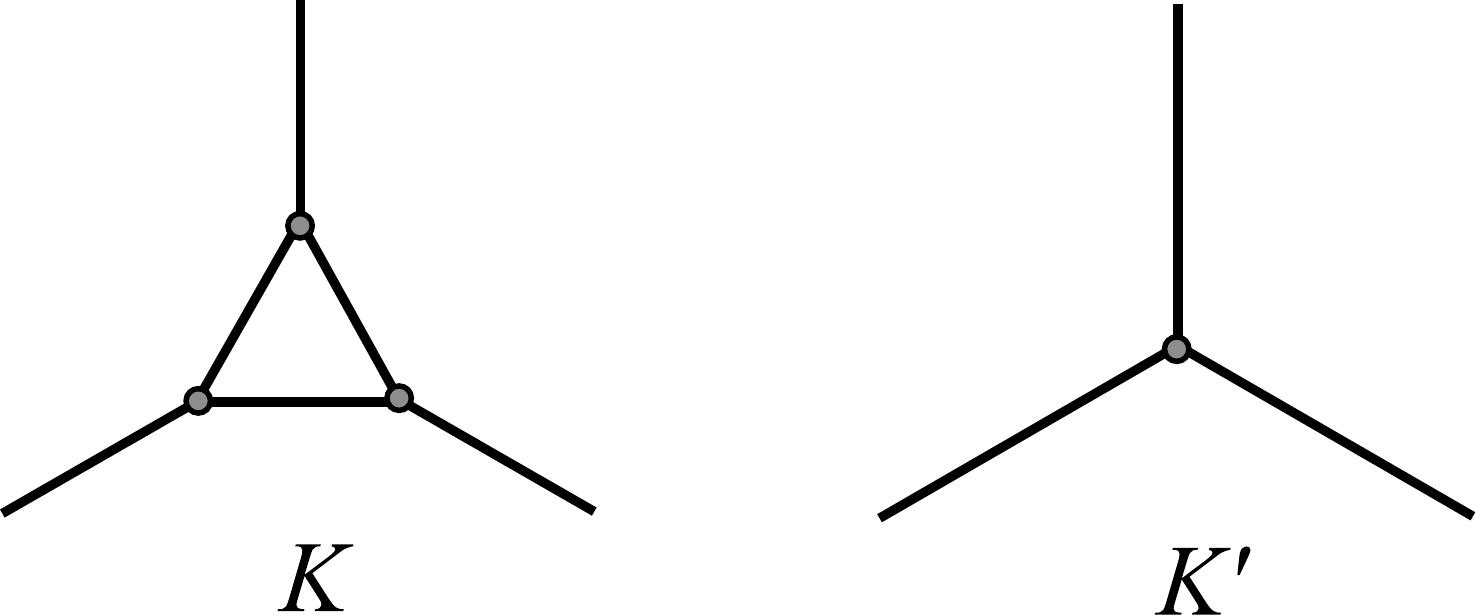}
    \end{center}
    \caption{\label{fig:triangle-move}
   The triangle move.}
\end{figure}
There are standard morphisms given by foams with one tetrahedral point each,
\[
\begin{aligned}
    A : K' &\to K \\
    C : K &\to K',
\end{aligned}
\]
and there are  resulting
linear maps
\[
\begin{aligned}
    a : \Jsharp(K') &\to \Jsharp(K) \\
    c : \Jsharp(K) &\to \Jsharp(K')
\end{aligned}
\]

\begin{figure}
    \begin{center}
        \includegraphics[scale=0.4]{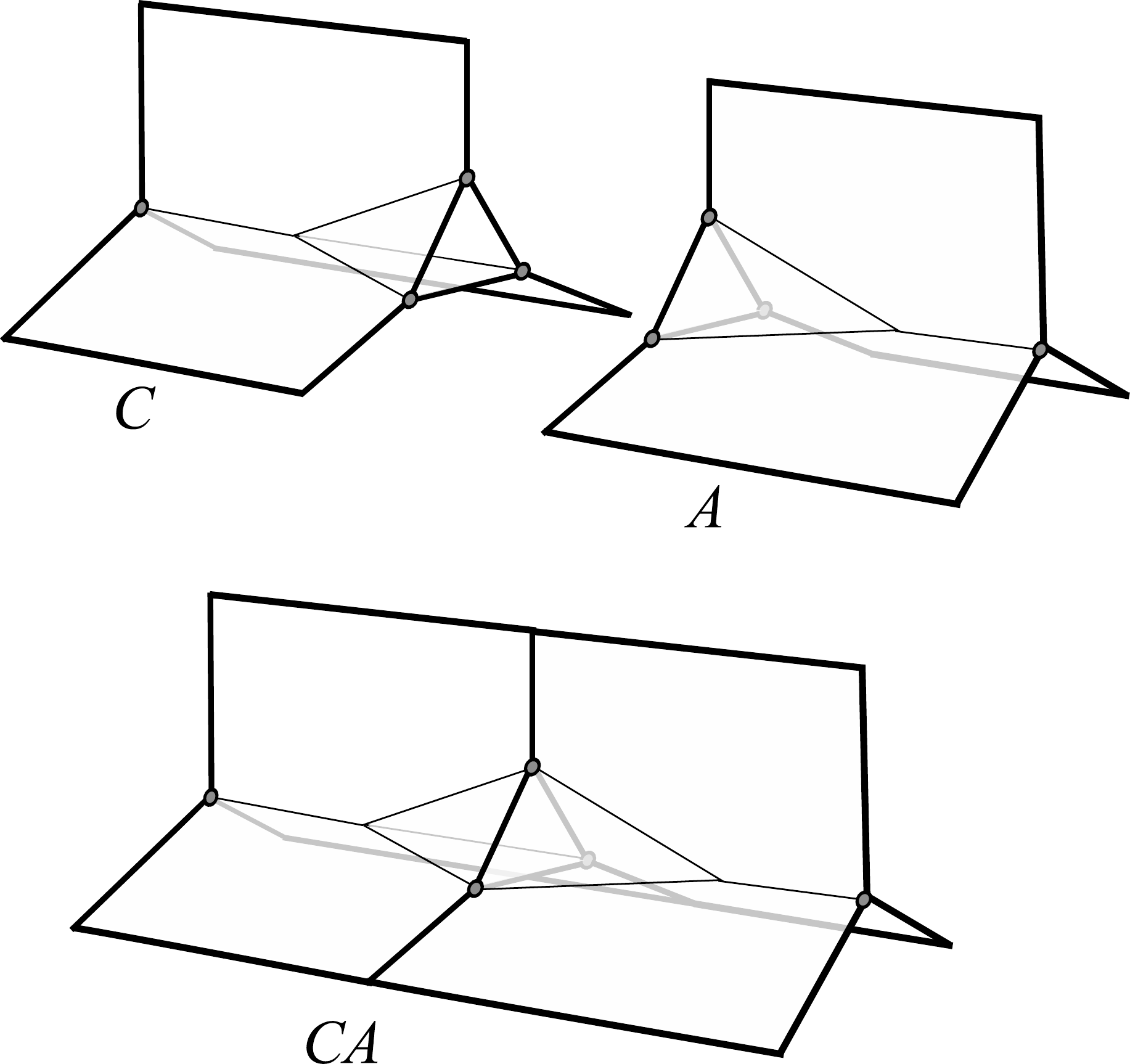}
    \end{center}
    \caption{\label{fig:composite-triangle-move}
   The foams $C$, $A$ and their composite $CA$, as a foam cobordism
   from $K'$ to $K'$.}
\end{figure}

\begin{proposition}\label{prop:triangle-relation}
    The instanton homologies $\Jsharp(K')$ and $\Jsharp(K)$ are
    isomorphic, and the maps $a$ and $c$ are mutually inverse isomorphisms.
\end{proposition}

\begin{proof}
    The proof follows the same plan as in the case of the bigon
    relation. To show that composite foam $CA$ gives the identity map
    on $\Jsharp(K')$, we apply the excision principle,
    Corollary~\ref{cor:local-relations-webs}, to the part of the
    composite shown on the left-hand side in
    Figure~\ref{fig:composite-triangle-move}. The left- and right-hand
    sides define elements in $\Jsharp$ of the theta graph $K_{\theta}$, which we
    must show are equal. We can demonstrate equality by checking that
    we have the same evaluation of closed foams when we pair both
    sides with the known basis of $\Jsharp(K_{\theta})$ coming from
    the foams $\Theta_{-}(0, k_{2}, k_{3})$. The required evaluations
    become (respectively) $\Jsharp(S(0, k_{2}, k_{3}))$, where $S$ is the cone on the
    tetrahedron web, and $\Jsharp(\Theta(0,k_{2},k_{3}))$, where
    $\Theta$ is the closed theta foam. These evaluations are equal by
    Proposition~\ref{prop:S-evaluation}
     and Proposition~\ref{prop:theta-evaluation}. This verifies that
     $ca=1$.

     To show that $ac=1$, we must show that 
      \begin{equation}\label{eq:tetrahedral-cancel}
               \alpha = \epsilon
      \end{equation}
     where $\alpha$ and $\epsilon$ are the elements in the instanton
     homology $\Jsharp(L_{3})$ corresponding to the foams in
     Figure~\ref{fig:AC-triangle-composite}. 
\begin{figure}
    \begin{center}
        \includegraphics[scale=0.4]{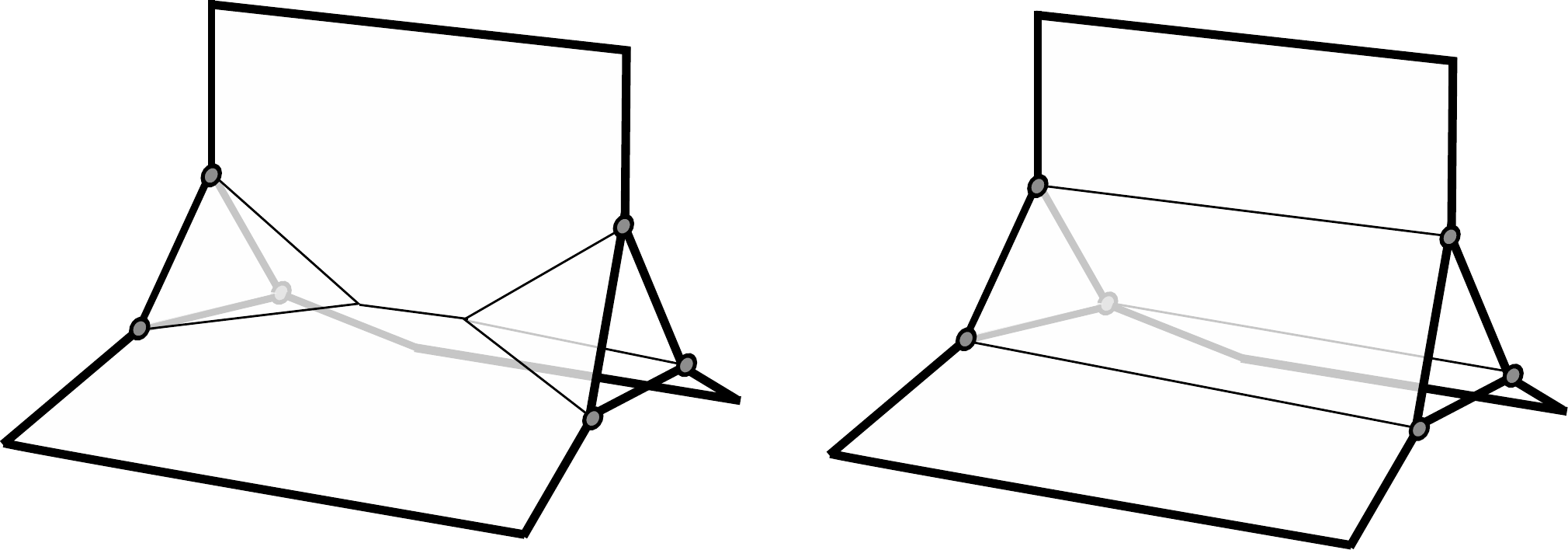}
    \end{center}
    \caption{\label{fig:AC-triangle-composite}
   The foams defining the elements $\alpha$ and $\epsilon$ in $\Jsharp(L_{3})$.}
\end{figure}
From the relation $ca=1$,
     it already follows that $AC$ is a projection onto the injective
     image of $\Jsharp(K)$ in $\Jsharp(K')$.  In the special case that
     $K'$ is the tetrahedron graph and $K$ is the prism $L_{3}$, we
     already know that $\dim \Jsharp(K) \le \dim\Jsharp(K')$ from
     Lemma~\ref{lem:upper-bounds}, so in this special case we know
     that $ac=1$ as endomorphisms of $\Jsharp(L_{3})$. Applying this
     to the element $\epsilon\in\Jsharp(L_{3})$, we obtain
     \[
           ac(\epsilon) = \epsilon.
      \]
     By an isotopy, we see that $ac(\epsilon)=\alpha$, and the proof
     is complete.
\end{proof}

We can reinterpret the central calculation in the above proof.
Let $(X,\Sigma)$ be a foam cobordism inducing a map $\Jsharp(X,\Sigma)
: \Jsharp(Y_{1}, K_{1}) \to \Jsharp(Y_{0}, K_{0})$. Suppose $x_{1}$,
$x_{2}$ are two distinct tetrahedral points in $\Sigma$ connected by a
seam-edge $\gamma$. A regular neighborhood of the arc $\gamma$ is a
4-ball meeting $\Sigma$ in a standard foam $W\subset B^{4}$ whose
boundary is the prism web $L_{3}$, the same foam that appears in the left-hand side of
Figure~\ref{fig:AC-triangle-composite}. Let $W'\subset B^{4}$ be the foam
with the same boundary shown in the right-hand side of the
figure, and let $\Sigma'$ be obtained from $\Sigma$ by replacing $W$
with $W'$.  We refer to the process of passing from $\Sigma$ to
$\Sigma'$ as \emph{canceling tetrahedral points.}

\begin{proposition}[Canceling tetrahedral points]
    In the above situation, we have $\Jsharp(X,\Sigma) = \Jsharp(X, \Sigma')$.
\end{proposition}

\begin{proof}
    By excision this reduces to the same problem as before, namely
    proving that $\alpha=\epsilon$ (equation
    \eqref{eq:tetrahedral-cancel}) in $\Jsharp(S^{3}, L_{3})$.
\end{proof}

\subsection{The square relation}
\label{subsec:square-reln}

Like the bigon and triangle relations above, our statement and proof
of the square relation follows \cite{Khovanov-sl3}, with the same sorts
of adaptations as in the earlier cases.
Let $K \subset Y$ be a web containing a square: four edges connecting
four vertices in a standard disk. Let $K'$ and $K''$ be obtained from
$K$ as shown in Figure~\ref{fig:square-relation}.

\begin{figure}
    \begin{center}
        \includegraphics[scale=0.4]{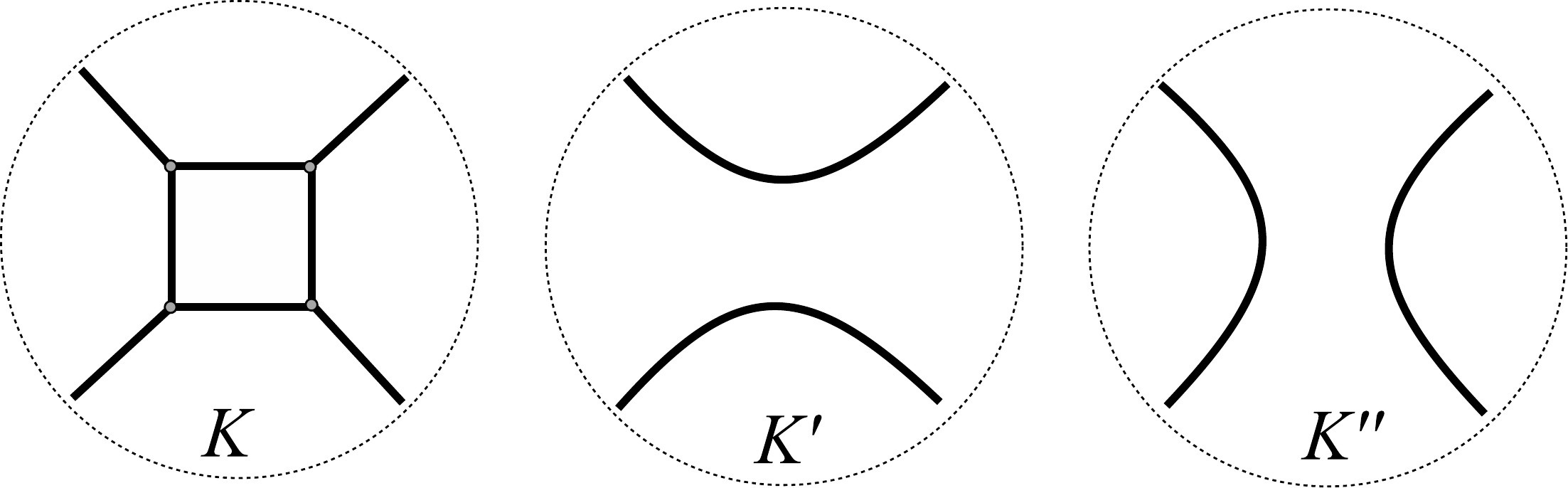}
    \end{center}
    \caption{\label{fig:square-relation}
   The webs $K$, $K'$ and $K''$ involved in the square relation:
   $\Jsharp(K)=\Jsharp(K')\oplus\Jsharp(K'')$.}
\end{figure}

\begin{proposition}\label{prop:square-relation}
    If the web $K$ contains a square and $K'$, $K''$ are as shown in
    Figure~\ref{fig:square-relation}, then
\[
                \Jsharp(K) = \Jsharp(K') \oplus \Jsharp(K'').
\]
\end{proposition}

\begin{proof}
    There are standard morphisms
\[
           \xymatrix{      K' \ar@<0.5ex>[r]^{A}  &
                           K \ar@<0.5ex>[l]^{C} \ar@<0.5ex>[r]^{D} & 
                           K'' \ar@<0.5ex>[l]^{B} }.
\]
A model for the foam $A$ is shown in Figure~\ref{fig:square-foam-A},
and the others are similar. Let $a$, $b$, $c$, $d$ be the
corresponding linear maps on $\Jsharp(K')$ etc.
\begin{figure}
    \begin{center}
        \includegraphics[scale=0.4]{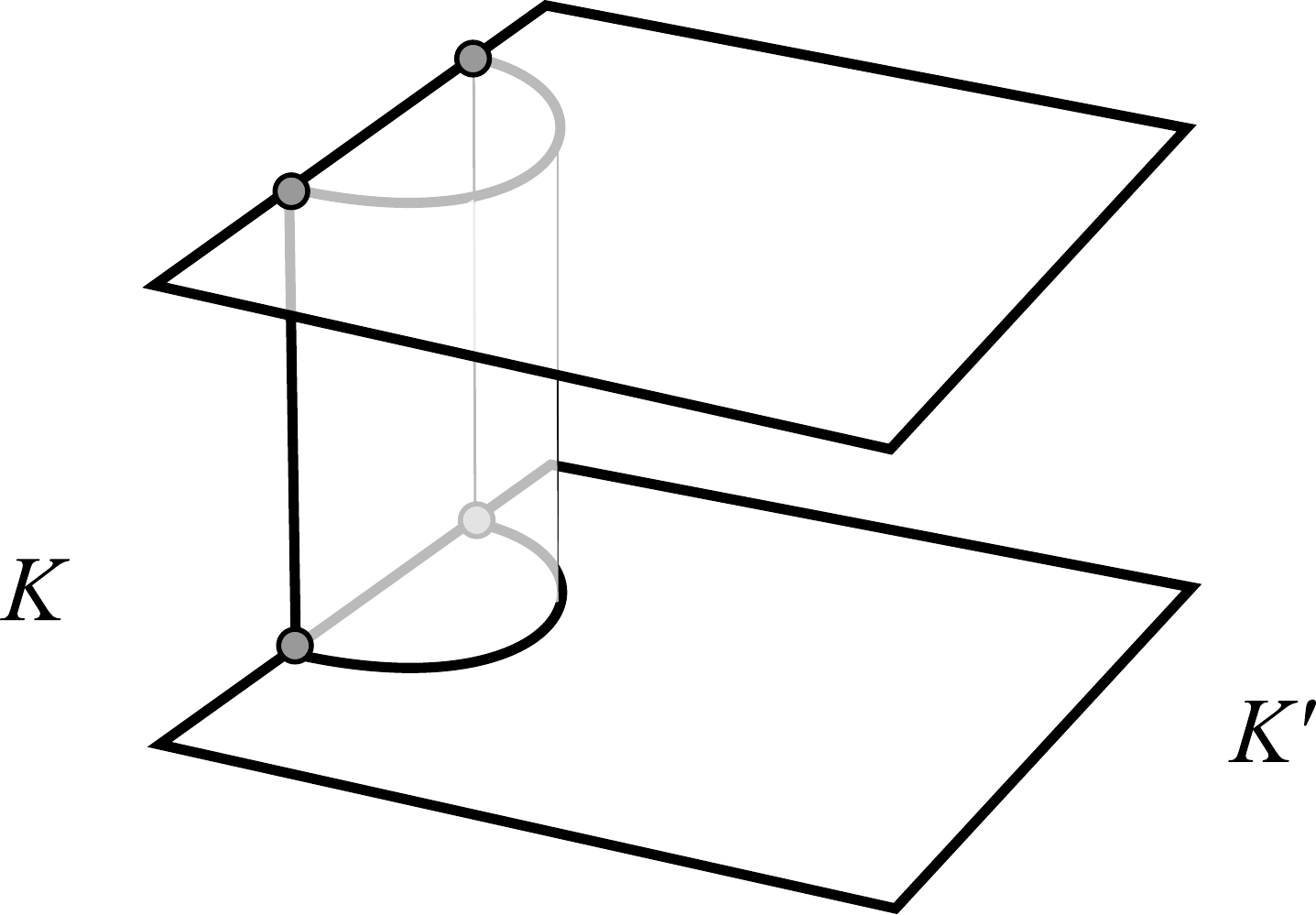}
    \end{center}
    \caption{\label{fig:square-foam-A}
   The standard cobordism from $K$ to $K'$ for the square relation.}
\end{figure}
 By an application of
neck-cutting and bubble-bursting, one sees that $ca=1$, and similarly $db=1$. By
an application of bubble-bursting, one sees that $da=0$, and similarly
$cb=0$. It follows $ac + bd$ is a projection onto the injective image
of $\Jsharp(K') \oplus \Jsharp(K'')$, and to complete the proof we
must show
\[
              ac + bd =1.
\]
In the case that $K$ is a cube, we know that $\Jsharp(K)$ has
dimension at most $24$ by Lemma~\ref{lem:upper-bounds}. The webs $K'$
and $K''$ in that case are the same, and can both be reduced to the
theta web by the collapsing of a bigon; so $\Jsharp(K')$ and
$\Jsharp(K'')$ both have rank $12$. So we learn that $ac + bd=1$ in
the case that $K$ is a cube. The general case is now proved by the
same strategy as in the bigon and triangle relations.
\end{proof}

\subsection{Simple graphs}
\label{sec:simple}

To the bigon, triangle and square relations above, we can add the
relations that hold for a `$0$-gon' and a `$1$-gon'. A $1$-gon is an
edge-loop which bounds a disk whose interior is disjoint from $K$. If
$K$ has a $1$-gon, then it has an embedded bridge, and
$\Jsharp(K)=0$. A $0$-gon is a vertexless circle bounding a disk whose
interior is disjoint from $K$. If $K'$ is obtained from $K$ by
deleting the circle, then $\dim \Jsharp(K) = 3 \dim \Jsharp(K')$, by
the multiplicative property,
Corollary~\ref{cor:tensor-product}, and
Proposition~\ref{prop:unknot-u-map}. If $\tau(K)$ denotes the number
of Tait colorings of $K$, then the relations that $\dim \Jsharp$
satisfies as a consequence of these relations for $n$-gons, $0\le n\le
4$, are satisfied also by $\tau(K)$ (as is well known and easy to
verify). It follows that if the square, triangle and bigon moves can
be used to reduce a web $K\subset \R^{3}$ to webs that are either
unlinks or contain embedded bridges, then $\dim \Jsharp(K) =
\tau(K)$. Conjecture~\ref{thm:Tait-count} is the statement that
$\dim\Jsharp(K)=\tau(K)$, and we see that a minimal counterexample
cannot have any $n$-gons with $n\le 4$.

\section{Proof of non-vanishing, Theorem~\ref{thm:non-vanishing}}
\label{sec:proof-nonvanish}

\subsection{Passing from $\Jsharp$ to $\Isharp$}
\label{subsec:J-to-I}

From a finite chain complex $(C,d)$ over $\F$ with a basis
$\Crit$, we can form a graph $\Gamma$ with vertices $\Crit$ and with edges
corresponding to non-zero matrix entries of $d$.
Let us say that a chain complex $(\tilde C, \tilde d)$ is a
\emph{double-cover} of $(C,d)$ if there is a basis $\Crit$ for $C$ and
basis $\tilde \Crit$ for $\tilde C$ such that the corresponding graph
$\tilde\Gamma$ is a double-cover of $\Gamma$.

\begin{lemma}
    If the homology of $(\tilde C,\tilde d)$ is non-zero, then so is
    the homology of $(C,d)$.
\end{lemma}

\begin{proof}
    This is a corollary of the standard mod $2$ Gysin sequence for the double
    cover, i.e. the long exact sequence corresponding to the short
    exact sequence of chain complexes
    \[
             0 \to C \to \tilde C \to C \to 0.
     \] 
    On the bases $\Crit$ and $\Crit$, the first map sends a basis
    element to the sum of its two preimages, and the second map is the
    projection.
\end{proof}

If we have a Morse function $f$ on a compact Riemannian manifold $B$,
satisfying the Morse-Smale condition, and if we pull it back to a
double-cover $\tilde B$, then the pull-back will still be
Morse-Smale. Furthermore, the associated Morse complex $(\tilde C,
\tilde d)$ will be a double-cover of $(C,d)$ in the above sense. The
same applies equally to a Morse complex of the perturbed Chern-Simons
functional on $\bonf_{l}(\check Y; \mu)$. This gives us the following
result.

\begin{lemma}
    Let $\check Y$ be a bifold with singular set $K$, 
      and suppose we have two
      different strong marking data, $\mu$ and $\mu'$, with $U_{\mu} \subset
     U_{\mu'}$ and $E_{\mu} = E_{\mu'}|_{U_{\mu}\sminus K}$. Suppose
     $J(\check Y ; \mu')$ is non-zero. Then $J(\check Y; \mu)$ is also non-zero.
\end{lemma}

\begin{proof}
    According to Lemma~\ref{lem:covering}, the map $\bonf_{l}(\check
    Y; \mu') \to \bonf_{l}(\check Y ; \mu)$ is an iterated
    double-covering of some union of components of the latter. It
    follows that the corresponding chain complex $(C', d')$ is an
    iterated double-cover of some direct summand $(C_{0}, d_{0})$ of the complex
    $(C,d)$. The previous lemma tells us that $H_{*}(C_{0}, d_{0})$ is
    non-zero. Therefore $H_{*}(C,d)$ is non-zero too.
\end{proof}

\begin{corollary}\label{cor:J-to-I}
    If $K$ is a web in $\R^{3}$ and $\Isharp(K)$
    is non-zero, then $\Jsharp(K)$ is also non-zero.
\end{corollary}

\begin{proof}
    This follows from the definition of $\Isharp$ and $\Jsharp$
    (Section~\ref{subsec:def-Jsharp}): they are defined using
    the same bifold, but $\Isharp$ is defined using a larger marking
    set.
\end{proof}

Using also the multiplicative property for split webs,
Corollary~\ref{cor:tensor-product}, we learn that, in order to prove
the non-vanishing theorem, Theorem~\ref{thm:non-vanishing}, it will
suffice to prove the following variant:

\begin{reduction}\label{prop:reduction-1}
    If $K\subset \R^{3}$ is a non-split web with no
    embedded bridge, then the
    $\F$ vector space
    $\Isharp(K)$ is non-zero.
\end{reduction}

At this point in the argument, we can dispense with $\Jsharp$.

\subsection{Removing vertices}

As mentioned in Section~\ref{subsec:def-Jsharp}, the construction
$\Isharp$ is an extension to webs of the invariant of
links (also called $\Isharp(K)$) which was defined in
\cite{KM-unknot}. We now adopt some of the notation and concepts from
\cite{KM-unknot} to continue the argument.

Let $(Y,K)$ be a $3$-manifold with an embedded web, and let
$\mu$ be marking data in which the subset $U_{\mu}$ is \emph{all} of
$Y$. In this case, let us represent $w_{2}(E_{\mu})$ as a the dual
class to a closed $1$-dimensional submanifold $w \subset Y\sminus K$,
consisting of a collection of circles and arcs joining points on edges
of $K$. Since $w$ determines the marking, we may write \[ I^{w}(Y,K)\]
for the invariant that we have previously called $J(Y,K ; \mu)$. This
is the notation used in \cite{KM-unknot}, though we are still using
$\F$ coefficients. The group is defined only if $(Y,K,w)$
satisfies a condition which is a counterpart of the strong marking condition:
a sufficient condition for this is
that there is a \emph{non-integral} surface $S$ in $Y$: i.e. a closed
orientable surface which either
\begin{enumerate}
\item has odd intersection with $K$, or
\item is disjoint from $K$ and has odd intersection with $w$.
\end{enumerate}
See
\cite[Definition~3.1]{KM-unknot}. The first condition holds, in particular, if $K$
has a vertex.
 In this notation, if $H\subset S^{3}$ is a Hopf
link near infinity in $S^{3}$ and $u$ is an arc joining its two
components, then for a web $K\subset \R^{3}$, we have
\begin{equation}\label{eq:unknot-notation}
           \Isharp(K) = I^{u}(S^{3}, K\cup H).
\end{equation}

There is another notation used in \cite[Section~5.2]{KM-unknot} that 
we also need
to adopt here. On $Y$, the obstruction to lifting an $\SO(3)$ gauge
transformation $u$ to the determinant-$1$ gauge group is an element $o(u)$ in 
$H^{1}(Y; \F)$. So the usual configuration space of connections
(the quotient by the determinant-$1$ gauge group) is acted on by
$H^{1}(Y; \F)$. Given a choice of subgroup $\phi\subset H^{1}(Y;
\F)$, we can form the quotient of the configuration space by
$\phi$ and compute the Morse theory of the perturbed Chern-Simons
functional on this quotient. The resulting group
\[
           I^{w}(Y,K)^{\phi}
\]
is defined whenever $\phi$ acts freely on the set of critical
points. A sufficient condition of this is that there be a non-integral
surface $S$ such that the restriction of $\phi$ to $S$ is zero. (See
\cite{KM-unknot} once more.) In our applications, this condition will
always be easy to verify, because there will be a $2$-sphere meeting $K$
in three points, on which $\phi$ will be zero.

We return to a general $(Y,K)$, with the marking data now represented
by $w$ as above. Let $v_{1}$, $v_{2}$ be two vertices of $K$, and let
$B_{1}$, $B_{2}$ be standard ball neighborhoods of these, each meeting
$K$ in three arcs joined at the vertex, disjoint from $w$. Let $Y^{+}$ 
be the oriented $3$-manifold obtained by removing the balls $B_{i}$ and
identifying the two $S^{2}$ boundary components. (Topologically this
is $Y\#(S^{1}\times S^{2})$.) We make the identification so that the
three points on $\partial B_{1}$ where $K$ meets $\partial (Y\sminus
B_{1})$ are identified (in any order) with the corresponding points on
$\partial B_{2}$. In this way we obtain a new pair $(Y^{+} K^{+})$, though
not uniquely. Because it is disjoint from the balls, we may regard $w$
as a submanifold also of $Y^{+}$.

\begin{proposition}\label{prop:excise-vertices}
    In the above situation, if $I^{w}(Y^{+}, K^{+})$ is non-zero, then so
    too is $I^{w}(Y, K)$.
\end{proposition}

\begin{proof}
     This is an application of excision. Let $S_{1}$ and $S_{2}$ be
     the spheres with three marked points, $\partial B_{1}$ and
     $\partial B_{2}$. If we cut $(Y,K)$ along $S_{1}$ and $S_{2}$ and
     re-glue, then we obtain a disconnected manifold with two
     components:
     \[
                  (Y' , K') = (Y^{+}, K^{+}) \cup (S^{3}, \Theta).
      \] 
     Here $\Theta$ is a standard Theta graph in $S^{3}$.  Let
     $S\subset Y^{+}$ be the new $2$-sphere where the two surfaces are
     identified. 
      Combing the
     argument of \cite[Theorem~5.6]{KM-unknot} with
     Theorem~\ref{thm:excision-street}, 
     there is an excision
     isomorphism,
     \[
     \begin{aligned}
         I^{w}(Y,K) &\cong I^{w}(Y' , K')^{\psi} \\
                    & =  I^{w}(Y^{+} , K^{+})^{\psi} \otimes I(S^{3},\Theta),
     \end{aligned}
     \]
     where the subgroup $\psi\subset H^{1}(Y'; \F)$ is the
     $2$-element subgroup generated by the class dual to $S$.

      Since the $I(S^{3}, \Theta) = \F$, we see that non-vanishing of
      $I^{w}(Y,K)$ is equivalent to non-vanishing of $I^{w}(Y^{+},
      K^{+})^{\psi}$. By the double-cover argument of
      Section~\ref{subsec:J-to-I}, this is implied by the
      non-vanishing of $I^{w}(Y^{+}, K^{+})$.
\end{proof}

If $K\subset S^{3}$ is a web, let $2n$ be the number of
vertices (which is always even). We may pair these up arbitrarily, and
apply the above topological construction $n$ times to obtain a new pair $(Y^{+} ,
K^{+})$ with no vertices at all. The manifold $Y^{+}$ will be a
connected sum of $n$ copies of $S^{1}\times S^{2}$. We call this
process \emph{excising all the vertices}.

\subsection{A taut sutured manifold}

Let us turn again to the situation of
Proposition~\ref{prop:reduction-1}. In the notation
\eqref{eq:unknot-notation}, we are looking at $I^{u}(S^{3}, K\cup
H)$. Excising all the vertices and applying
Proposition~\ref{prop:excise-vertices}, we see that it suffices to
prove  the following variant:

\begin{reduction}\label{var:K+Hopf}
    Suppose $K \subset S^{3}$ is a non-split web    
    without an embedded bridge and let $H$ and $u$ be as before.
    Let $(Y^{+}, K^{+})$ be obtained from $(S^{3},K)$ be excising all the
    vertices. Then $I^{u}(Y^{+}, K^{+} \cup H)$ is  non-zero.
\end{reduction}

Our next step deals with the Hopf link $H$. We state a lemma which
involves summing a Hopf link onto one of the components of a link $K$.

\begin{lemma}
    Let $K\subset Y$ be a web and $w\subset Y\sminus K$ a
    union  of arcs representing a choice of $w_{2}$. Let $H$ and $u$
    be a Hopf link and arc as usual, contained in a ball disjoint from
    $K$. Let $m$ be a meridian of some edge $e$ of $K$, i.e.~the boundary
    of a disk transverse to $e$, and let $v$ be an arc joining
    $m$ to $e$ in the disk. Then (with coefficients still $\F)$),
    we have
    \[
               \dim I^{w+u}(Y, K\cup H) = 2 \dim I^{w+v}(Y, K\cup m).
    \]
\end{lemma}

\begin{proof}
     We shall compare both the left- and right-hand sides of the
     equation to the dimension of the instanton homology when
     \emph{both} $m$ and $H$ are included,
    \[
                I^{w+v+u}(Y, K\cup m \cup H).
    \]

     The skein exact sequence of \cite{KM-unknot} gives an exact
    triangle
\[
               \dots  \stackrel{\partial}{\to} I^{w+u}(Y, K\cup H) \to      
             I^{w+v+u}(Y, K\cup m \cup H) \to I^{w+u}(Y,
               K\cup H) \stackrel{\partial}{\to} \dots.
\]
In the special case that
$K$ is the unknot in $S^{3}$ and $w$ is empty, we know that $ I^{w+v}(Y,
               K\cup H)$ has dimension $2$, and  $I^{w+v+u}(Y, K\cup \cup
               m\cup H)$ has dimension $4$; so in this case the connecting homomorphism
\[
               I^{w+u}(Y,
               K\cup H) \stackrel{\partial}{\to}  I^{w+u}(Y,
               K\cup H)
\]
is zero (with $\F$ coefficients). By an application of
Corollary~\ref{cor:local-relations-webs} it follows that $\partial=0$ for  \emph{all}
$K$ and $w$, and hence
    \[
               \dim I^{w+v+u}(Y, K \cup m \cup H) = 2 \dim I^{w+u}(Y, K\cup H).
    \] 

On the other hand we also have
\begin{equation}\label{eq:dim-times-4}
               \dim I^{w+v+u}(Y, K \cup m \cup H) = 4 \dim I^{w+u}(Y, K\cup m).
\end{equation}
To see this we can take the disjoint union of $Y$ with a $3$-sphere
containing two Hopf links,
\[
                         (Y, K \cup m) \amalg (S^{3}, H_{1} \cup H_{2}),
\]
and apply the excision principle, cutting and along tori to see that
the instanton Floer homology of this union is the same as that of
\[
            (Y, K \cup m \cup H) \amalg (S^{3}, H ).
\]
The formula \eqref{eq:dim-times-4} then follows from the fact that
\[
            \dim I^{u_{1}+u_{2}}(S^{3}, H_{1} \cup  \cup H_{2}) = 4 \dim I^{u}(S^{3}, H).
\]
\end{proof}

In the situation described in Variant~\ref{var:K+Hopf}, let $r$ be the
number of components of the link $K^{+}$. (This number will depend on
how we chose to make the identifications when we excised all the
vertices.) To prove non-vanishing of $I^{u}(Y^{+}, K^{+} \cup H)$, it
will suffice to prove non-vanishing of
\[
                 I^{u_{1}+\dots+ u_{r}}(Y^{+}, K^{+} \cup H_{1}\cup
                 \dots \cup H_{r})
\]
because of the tensor product rule for $\Isharp$ of disjoint
unions. By the Lemma above, the rank of this group is $2^{r}$ times
the rank of the group
\[
            I^{v}(Y^{+}, K^{+}\cup m)
\]
where $m=m_{1}\cup \dots \cup m_{r}$ is a collection of meridians, one
for each component of $K^{+}$, and $v=v_{1}+\dots+v_{r}$ is a
collection of standard arcs. 
We arrive at the
following variant which is sufficient to prove
Theorem~\ref{thm:non-vanishing}.

\begin{reduction}
    Let the web $K\subset S^{3}$ be
    non-split and without an embedded
    bridge, let $(Y^{+}, K^{+})$ be obtained by excising all  the vertices,
    let $m$ be a collection of meridians of $K^{+}$, one for each
    edge, and let $v$ be a collection of standard arcs joining each
    meridian to the corresponding component. Then
    \[
              I^{v}(Y^{+}, K^{+} \cup m )
    \]
    is non-zero.
\end{reduction}

If $K^{+}$ has only one component, then $I^{v}(Y^{+}, K^{+} \cup m)$
is exactly the group called $\Inat(Y^{+}, K^{+})$ in \cite{KM-unknot},
though over $\F$ here. It is in fact possible to choose the
identifications when excising the vertices so that $r=1$, but we will
continue without making any restriction on $r$.  

We now use excision, to remove all the components of $K^{+}$ and $m$,
in the same way that was done for the case $r=1$ in
\cite{KM-unknot}. (See in particular 
 \cite[Proposition~5.7]{KM-unknot} for essentially the same
argument that we present now.) For each component $K^{+}_{i} \subset
K^{+}$ and corresponding meridian $m_{i}$, let $T'_{i}$ and $T''_{i}$
be boundaries of disjoint tubular neighborhoods of $K^{+}_{i}$ and
$m_{i}$ respectively. Let $\bar Y^{+}$ be the $3$-manifold obtained from
$Y^{+}$ be removing all the tubular neighborhoods and gluing $T'_{i}$
to $T''_{i}$ for all $i$. The gluing is done in such a way that a
longitude of $K^{+}_{i}$ on $T'_{i}$ is glued to a meridian of $m_{i}$
on $T''_{i}$, and vice versa. In $\bar Y^{+}$, let $\bar T_{i}$ be the
torus obtained by the identification of $T'_{i}$ and $T''_{i}$. We
arrange the arc $v_{i}$, for each $i$, so that each hits both $T'_{i}$
and $T''_{i}$ in one point, so that $\bar Y^{+}$ contains closed
curves $\bar v_{i}$ transverse to each $\bar T_{i}$. Let $\bar\phi$ be the
subgroup of $H^{1}(\bar Y^{+} ; \F)$ generated by the classes dual
to the $r$ tori $\bar T_{i}$. By excision argument from the proof of  
\cite[Proposition~5.7]{KM-unknot}, we have an isomorphism
\[
        I^{v}(Y^{+}, K^{+} \cup m ) \cong I^{\bar v}(\bar Y^{+})^{\bar\phi}.
\]
(Here $\bar v$, of course, 
is the sum of the $\bar v_{i}$. Note also that there is
no link or web $K$ in $\bar Y^{+}$.) Applying the double-cover
argument again, we see that we will be done if we can show that
$I^{\bar v}(\bar Y^{+})$ is non-zero.

At this point, we appeal to the universal coefficient theorem, which
tells us that a finite complex over $\Z$ has non-trivial homology over
$\F$ if it has non-zero homology over $\Q$. So we will aim to
prove that $I^{\bar v}(\bar Y^{+} ; \Q)$ is non-zero. As in Section~5
of \cite{KM-unknot}, this last group is essentially the same as the
sutured Floer homology of an associated sutured manifold. To state it
precisely, let $M$ be complement in $Y^{+}$ of a tubular neighborhood
of $K^{+}$. Let $(M,\gamma)$ be the sutured manifold obtained by
putting two meridional sutures on each boundary component. Then we
have
\[
       \dim I^{\bar v}(\bar Y^{+} ; \Q) = 2 \dim \SHI (M,\gamma),
\]
where $\SHI$ is the sutured instanton Floer homology from
\cite{KM-sutures}. Our proof will be complete if we can show that
$\SHI(M,\gamma)$ is non-zero. According to
\cite[Proposition~7.12]{KM-sutures}, the sutured instanton homology is
non-zero provided that $(M,\gamma)$ is \emph{taut} in the sense of
\cite{Gabai}. (The proof of that Proposition rests on the existence of
sutured manifold hierarchies for taut sutured manifolds, proved in
\cite{Gabai}.) So the following variant will suffice:

\begin{reduction}
    Let the web $K\subset S^{3}$ be
    non-split and without an embedded
    bridge, let $(Y^{+}, K^{+})$ be obtained by excising all the
    vertices (so $Y^{+}$ is a connected sum of $n$ copies of
    $S^{1}\times S^{2}$), and let $(M,\gamma)$ be the sutured
    link complement. Then $(M,\gamma)$ is taut.
\end{reduction}

\subsection{Proof of tautness}

We will now show that $(M,\gamma)$ is taut. From the definition, this
means establishing:
\begin{enumerate}
\item \label{it:taut-sphere} every $2$-sphere in $M$ bounds a ball;
    and
 \item \label{it:taut-disk} no meridional curve on $\partial M$ bounds
     a disk.
\end{enumerate}
(The meridional curves are the curves parallel to the sutures.)

By its construction, $M$ contains $n$ pairs of pants $P_{1}, \dots,
P_{n}$, obtained from the standard $2$-spheres in $\# (S^{1}\times
S^{2})$ by removing the neighborhood of $K^{+}$. Write $P$ for their
union. If we cut $M$ along $P$, we obtain a $3$-manifold diffeomorphic
to the complement of a regular neighborhood of $K\subset
S^{3}$. The hypothesis that $K$ is non-split can therefore be restated
as the condition that every $2$-sphere in $M\sminus P$ bounds a ball
in $M\sminus P$. Similarly, the hypothesis that $K$ has no embedded
bridge can be stated as the condition that no meridional curve on
$(\partial M \sminus P)$ bounds a disk in $M\sminus P$.

So, for a proof by contradiction, 
it will suffice to establish the following: 

\begin{lemma}\label{lem:taut-to-taut}
If there is a sphere
$S\subset M$ contradicting \ref{it:taut-sphere}, or a disk
$D\subset M$ contradicting \ref{it:taut-disk}, then there is
another sphere $S'$ or disk $D'$ (not necessarily ``respectively'') 
contradicting \ref{it:taut-sphere} or
\ref{it:taut-disk} 
with the additional property that it is disjoint
from $P$.
\end{lemma}

\begin{proof}
    Suppose first that $D$ is a disk contradicting
    \ref{it:taut-disk}. We may assume that the boundary of $D$ is
    disjoint from $P$ and that $D$ meets $P$ transversely in a collection of
    disjoint simple closed curves. Assume there is no disk $D'$ with
    meridional boundary
    disjoint from $P$. Then the number of curves in $D\cap P$ is
    always non-zero, and we can choose $D$ to minimize this
    number.

     Any curve on $P$ either bounds a disk or is parallel to a
    boundary component of $P$. Given any collection of curves $C$ on $P$,
    there is always at least one curve among them that either bounds a
    disk  $\delta$ in $P$
    disjoint from the other components of $C$, or cobounds with
    $\partial P$ an annulus $\alpha$ which is disjoint from the other
    curves in $C$. Applying this statement to $C=D\cap P$, we let
    $\gamma$ be a component of $D\cap P$ that either bounds a disk
    $\delta$, or cobounds an annulus $\alpha$, disjoint from the other
    curves of intersection. Suppose $\gamma$ bounds a disk
    $\delta$. The curve $\gamma$ also lies on the embedded disk
    $D$. Let $A\subset D$ be the annulus bounded by $\partial D$ and
    $\gamma$. The interior of $A$ may intersect $P$, but will do so
    in a collection of curves that is smaller than $D\cap P$. From the
    union of $A$ and a parallel push-off of the disk $\delta$ (pushed
    off to be disjoint from $P$, on the same side locally as $A$), we
    form a new disk $D^{*}$ with the same boundary as $D$ but with
    fewer curves of intersection with $P$.  This contradicts the
    choice of $D$. In the case that $\gamma$ cobounds an annulus
    $\alpha\subset P$, and the interior of $\alpha$ is disjoint from
    $D$, we define $A$ as before and consider the disk formed as the
    union of $\alpha$ and $D\sminus A$. Pushing $\alpha$ off $P$ as
    before, we obtain a disk $D^{*}$ with meridional boundary having
    fewer curves of intersection with $P$ than $D$ did: the same
    contradiction.

    Now suppose there is a sphere $S\subset M$    
    contradicting \ref{it:taut-sphere}. Suppose that every
    sphere $S'$ in $M\sminus P$ bounds a ball in $M\sminus P$. Then
    $S$ must intersect $P$.
    Arrange that $S\cap P$ is transverse and let $C$ be this
    collection of curves. If a component $\gamma\subset C$ is isotopic
    in $P$
    to a boundary curve, then let $\alpha\subset P$ be the annulus
    bounded by $\gamma$ and the boundary of $P$. The curve $\gamma$
    cuts $S$ into two disks. The union of one of these disks with
    $\alpha$ is an immersed disk in $M$ with meridional boundary. By
    Dehn's lemma, there is also an embedded disk $D$ in $M$, with the same
    boundary. This takes us back to case \ref{it:taut-disk}, which we
    have already dealt with.

    There remains the case that every component of $C=S\cap P$ bounds
    a disk in $P$. Let us choose $S$ so that the number of components
    of $C$ is as small as possible. Among the components of $C$, there
    is a curve $\gamma$ that bounds a disk $\delta\subset P$ which is
    disjoint from the other curves in $C$. If we surger $S$ along
    $\delta$, we obtain two spheres $S_{1}$ and $S_{2}$. Let $C_{1}$
    and $C_{2}$ be their curves of intersecton with $P$. We have
    $C_{1} \cup C_{2} = C\sminus\gamma$. So, by our minimality
    hypothesis on $S$, it follows that both $S_{1}$ and $S_{2}$ bound
    balls, $B_{1}$ and $B_{2}$ in $M$. The union of the balls,
    together with a two-sided collar of $\delta$, is an immersed ball
    $B$ in $M$ with boundary $S$. (The ball is only immersed because
    $B_{1}$ and $B_{2}$ may not be disjoint.) An embedded sphere which bounds an
    immersed ball also bounds an embedded ball. 
    So $S$ bounds a ball in $M$,
    contrary to hypothesis.
\end{proof}

The proof of this lemma completes the proof of
Theorem~\ref{thm:non-vanishing}. It is clear from the proof that a
more general result holds, which we state here:

\begin{theorem}
Let $Y$ be a closed, oriented $3$-manifold and $K\subset Y$ an embedded
web. Suppose that the complement $K\sminus Y$ is
irreducible (every $2$-sphere bounds a ball), and that $K$ has no
embedded bridge (no meridian of $K$ bounds a disk in the
complement). Then $\Jsharp(Y,K)$ is non-zero.    
\end{theorem}

\begin{proof}
    From $K\subset Y$, we form $(Y^{+}, K^{+})$ as before by excising
    all vertices. So $Y^{+}$ is the connected sum of $Y$ with $n$
    copies of $S^{1}\times S^{2}$ and $K^{+}$ is an $r$-component link
    in $Y^{+}$. Let $(M,\gamma)$ be the sutured link
    complement, with $2$ meridional sutures on each component of the
    $r$ components of the boundary. The same sequence of variants
    applies, and we left to show again that $(M,\gamma)$ is
    taut. The proof of Lemma~\ref{lem:taut-to-taut} still applies,
    without change.
\end{proof}

\section{Tait-colorings, $O(2)$ representations and other topics}
\label{sec:further}

\subsection{$O(2)$ connections and Tait colorings of spatial and 
planar graphs}

Let $(Y,K)$ be a connected $3$-manifold containing a web $K$ with at
least one vertex. As discussed in Section~\ref{sec:examples-Rep}, an
$\SO(3)$ bifold connection $(E,A)$ has automorphism group
$\Gamma(E,A)$ which is either $V_{4}$, $\Z/2$ or trivial, according as
the corresponding connection is a $V_{4}$-connection, a fully
$O(2)$-connection, or a fully irreducible connection respectively.
We write $\Rep(Y,K)$ for the space of flat bifold
connections.

The following lemma is a straightforward observation. (See also
Lemma~\ref{lem:V4-is-Tait}.)

\begin{lemma}\label{lem:O2-implies-V}
    If $K$ is a planar graph and $\Rep(K)$ contains an $O(2)$
    connection, then $\Rep(K)$ also admits a $V_{4}$ connection, and
    $K$ therefore admits a Tait coloring.
\end{lemma}

\begin{proof}
    The are two different conjugacy classes of elements of order $2$
    in $O(2)$. There is the central element, $i$, and the conjugacy
    class of ``reflections'' that comprise the non-identity component
    of $O(2)$. If $K\subset \R^{3}$ is a (not necessarily planar) web
    and $\rho$ a fully $O(2)$ connection in $\Rep(K)$, then each edge
    of $K$ is labeled by an element that is either a reflection or
    central. Furthermore, two reflection edges meet at each trivalent
    vertex. The reflection edges therefore form  a \emph{2-factor} for
    $K$ (by definition, a set of edges with two edges incident at each
    vertex). The edges in the $2$-factor form a collection of closed
    curves.

    If the web is planar, then we can give a standard meridian
    generator for each edge, by connecting a standard loop to a point
    above the plane (as one does for the Wirtinger presentation of a
    knot group). For each edge $e$, we therefore have a well-defined
    element $\rho_{e}$ in $O(2)$. Along the loops formed by the edges
    of the $2$-factor, the elements $\rho_{e_{m}}$, $\rho_{e_{m+1}}$
    corresponding to consecutive edges differ by multiplication by
    $i$. There are therefore an even number of edges in each loop in
    the $2$-factor. There is therefore a Tait coloring of $K$, in which the
    edges in the $2$-factor are labeled alternately $j$ and $k$, and the
    remaining edges are labeled $i$.
\end{proof}

From these lemmas, we obtain an elementary reformulation of the
four-color theorem as the statement that, for any bridgeless planar web $K$, the
representation variety $\Rep(K)$ contains an $O(2)$ connection. 

For \emph{non-planar} webs in $\R^{3}$, the existence of an $O(2)$ connection does \emph{not} imply
the existence of a Tait coloring. Nor is it true that, for every web
$K$ in $\R^{3}$ without a spatial bridge, the representation variety
$\Rep(K)$ contains on $O(2)$ representation. (An example with no $O(2)$
representation can be constructed by starting with a non-split
2-component $L$ whose two components bound disjoint Seifert surfaces,
and taking $K$ to be the union of $L$ with an extra edge joining the
two components.) 

\subsection{Foams with non-zero evaluation, and $O(2)$ connections}

The remarks about $O(2)$ representations for webs in the previous
subsection are motivated in part by the following positive result
about $O(2)$ bifold connections for foams in $4$-manifolds.

\begin{proposition}\label{prop:O2-existence-foams}
    Let $\Sigma$ be a closed foam in a closed $4$-manifold $X$ (possibly
    with dots). If $\Jsharp(X,\Sigma)$ is
    non-zero, then for every Riemannian metric on the corresponding
    bifold $\check X$, there exists an anti-self-dual $O(2)$
    bifold connection.

     Furthermore, if $X$ admits an orientation-reversing
     diffeomorphism that fixes $\Sigma$ pointwise, then there exists a
     flat $O(2)$ bifold connection.
\end{proposition}

\begin{proof}
    First of all, we can immediately reduce to the case of dotless
    foams by applying Proposition~\ref{prop:tori-as-dots}.

    Let $Z$ be the product $S^{1}\times S^{3}$ containing the foam
    $S^{1}\times H$, where $H$ is the Hopf link. Let $\check Z$ be the
    corresponding bifold. The non-vanishing of $\Jsharp(X,\Sigma)$
    implies the existence, for all metrics, of an anti-self-dual
    $\SO(3)$ bifold connection on the connected sum $\check X \# \check
    Z$. (The connected sum $X\# Z$ is the manifold we obtain when we
    regard $X$ as a cobordism from $S^{3}$ to $S^{3}$ by removing two
    balls, and then gluing the two boundary components together.)
    Furthermore, these anti-self-dual connections lie in
    $0$-dimensional moduli spaces. 

    By the usual connected sum argument, either $\check X$ or $\check
    Z$ carries a solution lying in a negative-dimensional moduli
    space. The $\kappa=0$ moduli space on $\check Z$ has formal
    dimension $0$, so in fact $\check X$ must carry an anti-self-dual
    connection in a negative-dimensional moduli spaces. Furthermore,
    this must remain true for any Riemannian metric on $\check X$ and
    any small $4$-dimensional holonomy perturbation on $\check X$.  The
    following lemma therefore completes the proof of the first
    assertion in the proposition.

    \begin{lemma}
        If a moduli space on $\check X$ has negative formal dimension
        and consists only of fully irreducible connections, then a
        small holonomy perturbation can be chosen so as to make the
        moduli space empty. 
    \end{lemma}
     
    \begin{proof}
        We adapt the argument in
    \cite{Donaldson-orientations} to our case of $\SO(3)$
    connections. The map $\psi :\SO(3) \to \mathfrak{so}(3)$ given by
    $A\mapsto A-A^{T}$ is equivariant for the adjoint action of
    $\SO(3)$, and has the property that if $H\subset \SO(3)$ is any
    subgroup that is not contained in a conjugate of $O(2)$ then
    $\psi(H)$ spans $\mathfrak{so}(3)$. Using $\psi$ to construct
    perturbations as in \cite[Section 2(b)]{Donaldson-orientations},
    one shows the existence of holonomy perturbation such that the
    moduli space is regular at all fully irreducible solutions.
    \end{proof}

To prove the last statement in
Proposition~\ref{prop:O2-existence-foams}, we note that any $O(2)$
connection determines a flat real line bundle on $\xi$ $X\setminus \Sigma$
that has non-trivial holonomy on the links of some $2$-dimensional surface
$S\subset \Sigma\subset X$ comprised of a some subset of the
facets. The curvature of an anti-self-dual $O(2)$ bifold connection
can be interpreted as an anti-self-dual harmonic $2$-form, $f\in
\mathcal{H}^{-}(\xi)$, with values in $\xi$ and lying in certain lattice in $\mathcal{H}^{2}(\xi)$.
An adaptation of the usual argument for reducible connections in
instanton moduli spaces shows that, for generic metric, any such
orbifold $2$-form is zero unless $\dim \mathcal{H}^{+}(\xi) = 0$ and
$\dim \mathcal{H}^{-}(\xi) \ne 0$. However, if $X$ has an
orientation-reversing diffeomorphism fixing $\Sigma$, then these two
vector spaces have the same dimension.
\end{proof}

\begin{corollary}
    Let $\Sigma$ be a closed foam in $\R^{3}\subset\R^{4}$. If
    $\Jsharp(\Sigma)$ is non-zero, then the foam $\Sigma$ admits a
    ``Tait coloring'': a coloring of the facets of $\Sigma$ by three
    colors so that all three colors appear on the facets incident to
    each seam. 
\end{corollary}

\begin{proof}
    By the previous proposition, there is a flat $O(2)$ bifold connection
    on $(\R^{4},\Sigma)$.  We shall imitate the argument from the
    $3$-dimensional case, Lemma~\ref{lem:O2-implies-V}, to see that the
    existence of a flat $O(2)$ connection implies the existence of
    flat $V_{4}$ connection when $\Sigma$ is contained in
    $\R^{3}$.

     Whether or not $\Sigma$ is contained in $\R^{3}$, given an $O(2)$
     bifold connection and corresponding homomorphism $\rho$ from
     $\pi_{1}(X\setminus \Sigma)$, a representative of  the conjugacy
     class corresponding to the link of each facet of $\Sigma$ in $X$
     is mapped by $\rho$ either to  the central element $i$  in $O(2)$
     or to a reflection. The reflection facets constitute what we can
     call a ``2-factor'' for the foam: a subset of the facets such
     that, along each seam, two of the three facets which meet locally
     there belong to this subset. The facets belonging to the $2$-factor
     form a collection of closed surfaces in $S\subset X$.

     When $\Sigma$ is contained in $\R^{3}$, then there is a standard
     representative in $\pi_{1}(X\setminus\Sigma)$  for the link of
     each facet, so $\rho$ gives a map from facets to elements of
     order $2$ in $O(2)$, much as in the case of a planar web. So for
     each facet $f$, we have an element $\rho_{f}$. For facets in the
     $2$-factor, $\rho_{f}$ is a reflection, and if two such facets
     $f_{1}$ and $f_{2}$ are separated by a seam then $\rho_{f_{1}} =
     i \rho_{f_{2}}$. On each component of the closed surface $S$, the
     facets are therefore labeled by only $2$ distinct reflections, in
     a checkerboard fashion. We can therefore give a Tait coloring of
     the facts by coloring the rotation facets with $i$ and the
     refection facets alternately by $j$ and $k$.
\end{proof}

\begin{corollary}
    If $\Sigma$ is a closed foam in $\R^{3}$ with $\Jsharp(\Sigma)\ne
    0$, then the seams of $\Sigma$ form a bipartite $4$-valent graph.
\end{corollary}

\begin{proof}
    The proof of the previous corollary shows that the seam $\Gamma$
    is a $4$-valent graph lying on a closed surface $S\subset
    \R^{3}$, and that the facets into which $S$ is divided admit a
    coloring with two colors called $j$ and $k$.
    For each connected component $S_{m}$ of $S$, the
    corresponding subgraph of $\Gamma$ is formed as the transverse intersection
    of a collection of circles $\gamma_{m,l}$ arising as boundary
    components of the facets labeled by the central element $i\in
    O(2)$ (i.e. the facets not contained in
    $S$). The surface $S_{m}$ separates $\R^{3}$ into two
    components, and we can therefore partition the circles
    $\gamma_{m,l}$ into two types, according to the component in which
    the corresponding facet of
    $\Sigma\setminus S$ lies. Circles $\gamma$ and $\gamma'$ belonging
    to the same type do not intersect. We can now distinguish two
    different types of intersection points, as follows. After choosing an orientation
    of $S_{m}$, we can choose an oriented chart at each intersection
    point in $S_{m}$ which maps the circle $\gamma$ of the first type
    to the $x$ axis and the circle of the second type to the $y$
    axis. The first quadrant is then colored by either $j$ or $k$. In
    this way we partition the vertices of $\Gamma$ into two sets, in a
    way that exhibits $\Gamma$ as bipartite.
\end{proof}

\begin{corollary}
     Let $\Sigma\subset X$ be a closed foam in a closed $4$-manifold
     $X$ with $\Jsharp(X,\Sigma)\ne 0$. 
     Let $K \subset Y$ be the transverse intersection of $\Sigma$
     with a closed $3$-manifold $Y\subset X$. Then $\Rep(Y,K)$
     contains a flat $O(2)$ connection. 

     Furthermore, if $K$ is planar, meaning that it is contained in a
     standard $2$-disk in $Y$. Then $K$ admits a Tait coloring.
\end{corollary}

\begin{proof}
    We know from Proposition~\ref{prop:O2-existence-foams} that, for
    every metric on the bifold $\check X$, there exists an
    anti-self-dual $O(2)$ connection on $\check X$. Furthermore, the
    action of this connection has an upper bound that is independent
    of the metric. If we choose a sequence of metrics on $\check X$
    that contain cylinders on $\check Y$ of increasing length, then
    standard applications of the compactness theorems will lead to a
    flat $O(2)$ connection on $\check Y$. If $K$ is planar, then the
    $O(2)$ connection leads to a Tait coloring as before.
\end{proof}

Because of the previous corollaries, an affirmative answer to 
the following question would provide a proof of the four-color
theorem.

\begin{question}
     If $K\subset\R^{3}$ is bridgeless and planar, does it arise as $\R^{3}\cap
     \Sigma$ for some foam $\Sigma \subset \R^{4}$ with
     $\Jsharp(\Sigma)\ne 0$?
\end{question}

In the situation considered by the question above, we can regard the
evaluation of $\Jsharp(\Sigma)$ as a pairing of an element of
$\Jsharp(K)$ with an element in the dual space
$\Jsharp(-K)$, where $-K$ is the mirror image. Because we
know that $\Jsharp(K)$ is non-trivial, an affirmative answer to
the next question would also suffice.

\begin{question}\label{question:Generation-of-Jsharp}
    Let $K\subset S^{3}$ be a planar web. Is it always true that
        $\Jsharp(K)$ is generated by the elements $\Jsharp(\Sigma)$
        for foams $\Sigma\subset (\R^{4})^{+}$ with boundary $K$?
\end{question}

\begin{remarks}
    If $K\subset\R^{3}$ is spatial web that is not planar, it may be that there are
    no $O(2)$ representations in $\Rep(K)$, as noted previously. For
    such a $K$, there cannot be foam $\Sigma$ with $\Jsharp(\Sigma)\ne
    0$ such $\Sigma\cap\R^{3}=K$. So one cannot expect an affirmative
    answer to either of the two questions above if one drops the
    planar hypothesis for $K$. 

    Whether it is reasonable to conjecture
    an affirmative answer to this question for planar graphs is not at
    all clear. For example, in the case of the dodecahedral graph $K$, an
    affirmative answer would probably require that the elements
    $\Jsharp(\Sigma)$ corresponding to foams $\Sigma$ with boundary
    $K$ span a space of dimension at least $60$ (since this is the
    number of Tait colorings). Calculations by the authors failed to
    show that the span had dimension any larger than $58$.
\end{remarks}

\subsection{A combinatorial counterpart}

\label{sec:jflat}

Motivated by the partial information that we have about the evaluation
of closed foams, we can try to define a purely combinatorial counterpart to $\Jsharp$ for
planar webs $K$, by closely imitating the construction of Khovanov's
$\sl_{3}$ homology from \cite{Khovanov-sl3}. We shall describe this
combinatorial version here, though we do run into a question of
well-definedness.

Let us consider closed \emph{pre-foams}, by which we mean abstract
2-dimensional cell complexes, with finitely many cells, with a local
structure modeled on a foam, in the sense defined in this
paper. Unlike the pre-foams in \cite{Khovanov-sl3}, no orientations
are required, and we allow pre-foams to have tetrahedral points. We
allow our pre-foams to have decoration by dots.

We can try and define an evaluation, $\Jflat(S) \in \F$, for such
prefoams $S$ by applying the following rules.

\begin{enumerate}
\item The seams of the pre-foam form a $4$-valent graph. If this graph
    is not bipartite, then we define $\Jflat(S)=0$. 

\item If the seams form a bipartite graph, we can pair up the
    tetrahedral points and cancel them in pairs (Figure~\ref{fig:AC-triangle-composite})
    to obtain a new pre-foam $S'$ with no tetrahedral
    points. We declare that $\Jflat(S)=\Jflat(S')$, and so reduce to
    the case that there are no tetrahedral points.

\item When there are no tetrahedral points, the seams form a union of
    circles where three facets meet locally. If the monodromy of the
    three sheets is a non-trivial permutation around any of these
    circles, then we define $\Jflat(S)=0$. In this way we reduce to
    the case that the model for the neighborhood of each component of
    the seams is a product $S^{1}\times Y$, where $Y$ is a standard
    $Y$-shape graph.  

\item For each component of the seam having a neighborhood of the form
    $S^{1}\times Y$, we apply (abstract) neck-cutting (i.e. surgery)
    on the three circles parallel to the seam in each of the three
    neighboring facets, as in \cite{Khovanov-sl3}. In this way we
    reduce to the case that $S$ is a disjoint union of theta-foams and
    closed surfaces.

\item The theta-foams are evaluated using the rule for $\Jsharp$
    (Proposition~\ref{prop:theta-evaluation}) to leave only closed
    surfaces. The evaluation $\Jflat(S)$ for a
    sphere with two dots is $1$, for a torus with no dots is $1$, and
    for everything else is zero (including non-orientable components).
\end{enumerate}

The issue of well-definedness arises here because of the choice of how
to pair up the tetrahedral points, but the following two conjectures
both seem plausible.

\begin{conjecture}\label{conj:well-defined}
    The above rules lead to a unique, well-defined evaluation
    $\Jflat(S) \in \F$ for any pre-foam $S$.
\end{conjecture}

\begin{conjecture}
    For a closed foam $\Sigma$ in $\R^{3}$ with underlying pre-foam
    $S$, we have $\Jsharp(\Sigma)=\Jflat(S)$.
\end{conjecture}

The difficulty with establishing the second conjecture arises from the
fact that the abstract surgeries required to do neck-cutting on a
pre-foam $S$ may not always be achieved by embedded surgeries when the
pre-foam is
embedded as a foam in $\R^{3}$.

Assuming the first of the two conjectures, one can associate an
$\F$ vector space $\Jflat(K)$ to a
planar web $K$ by taking as generators the set of all foams
$\Sigma\subset (\R^{3})^{-}$ with boundary $K$ and relations
\[
           \sum [\Sigma_{i}] = 0
\]
whenever the collection of foams $\Sigma_{i}$ satisfies
\[
            \sum \Jflat (\Sigma_{i} \cup T) = 0
\]
for all foams $T\subset (\R^{3})^{+}$ with $\partial T = -K$.  This
definition imitates \cite{Khovanov-sl3}; and by extending the
arguments of that paper one can show that $\Jflat(K)$ satisfies the
bigon, triangle and square relations. In particular (still contingent
on Conjecture~\ref{conj:well-defined}), $\Jflat(K)$ is isomorphic to
$\Jsharp(K)$ at least for simple planar graphs in the sense of 
section~\ref{sec:simple}.

If both of the above conjectures are true, then $\Jflat(K)$ is a
subquotient of $\Jsharp(K)$: it is the subspace of $\Jsharp(K)$
generated by the foams with boundary $K$, divided by the annihilator
of the dual space generated by the foams with boundary $-K$.

\begin{remark}
    Assuming Conjecture~\ref{conj:well-defined}, the authors examined
    $\Jflat(K)$ in the case of the dodecahedral graph
    \cite{KM-dodecahedron-Tutte}, and showed that its dimension in
    this case is at least $58$. See also the remarks at the end of the
    previous subsection.
\end{remark}

\subsection{Gradings}

Although it plays no real role in the earlier parts of this paper, we
consider now whether there is a natural $\Z/d$ grading on the homology
group $\Jsharp(K)$, or more generally on $J(Y,K; \mu)$, for a web
$K\subset Y$ with strong marking data $\mu$. To formulate the question
precisely, we write $\check Y$ as usual for the corresponding
bifold, and $\bonf_{l}(\check Y;\mu)$ for the space of marked bifold
connections. For any closed loop $\zeta$ in $\bonf_{l}(\check Y;\mu)$
there is a corresponding spectral flow $\sflow_{\zeta}$; and if
$\sflow_{\zeta}\in d\Z$ for all loops $\zeta$, then we can regard
$J(Y,K; \mu)$ as $\Z/d$ graded, in the usual way. With this
understood, we have the following result.

\begin{proposition}
    If $K\subset Y$ has no vertices, or if the abstract graph
    underling $K$ is bipartite, then $J(Y,K;\mu)$ may be given a
    $\Z/2$ grading. If the graph is not bipartite, then in general
    there may exist closed loops with spectral flow $1$, and there is
    no non-trivial $\Z/d$ grading for any $d$. 
\end{proposition}

\begin{proof}
    The spectral flow $\sflow_{\zeta}$ is equal to the index $d(E,A)$
    for the corresponding $4$-dimensional bifold connection on
    $S^{1}\times \check Y$.
    By the formula \ref{prop:dimension}, 
    this in turn is equal to $8\kappa$, where $\kappa$ is the
    action of $(E,A)$. 

   If $K$ has no vertices, then the symbol of the linearized equations
   with gauge fixing on the bifold $S^{1}\times Y$ is homotopic to the
   symbol of a complex-linear operator, for we can choose an
   almost-complex structure on $S^{1}\times Y$ such that the
   submanifold $S^{1}\times K$ is almost complex: the appropriate
   complex-linear operator is then the operator $\bar\partial +
   bar\partial^{*}$ coupled to the complexification of $E$ on the
   complex orbifold. It follows that the real index $d(E,A)$ is even
   in this case.

   In the general case, let $e$ be an edge of $K$, let $m_{e}$ be a
   meridional circle linking $e$, and consider the torus $S^{1}\times
   m_{e}$ in $S^{1}\times (Y\setminus K)$. We can evaluate $w_{2}(E)$
   on $T_{e}$, and we call the result $\eta(e)$:
   \[
   \begin{aligned}
       \eta(e) &= \bigl\langle w_{2}(E) , [S^{1}\times m_{e}]
       \bigr\rangle \\
             &\in \{0,1\}.
   \end{aligned}
   \]
   If edges $e_{1}$,
   $e_{2}$, $e_{3}$ meet at a vertex $v$, then the sum of the
   meridians $[m_{e_{1}}] + [m_{e_{2}}]  + [m_{e_{3}}]$ is zero in
   homology, so
   \[
            \eta(e_{1}) + \eta(e_{2}) + \eta(e_{3}) = 0.
   \]
   It follows that, at each vertex $v$, the number edges
   which are locally incident at $v$ and have $\eta(e)=1$ is either
   $0$ or $2$. Let us say that a vertex $v$ of $K$ has \emph{Type I}
   if all incident edges have $\eta(e)=0$, and \emph{Type II} if two
   of the incident edges have $\eta(e)=1$.  

   \begin{lemma}
       If $N$ is the number of vertices of $K$ with Type II, then the
       action $\kappa$ is equal to $N/8$ modulo $(1/4)\Z$.
   \end{lemma}

   \begin{proof}[Proof of the lemma] Let $v$ be any vertex of $K$, let
       $B_{v}$ be a ball neighborhood of $v$, and consider the
       restricion of the bifold connection subset
       $S^{1}\times B_{v}$ in $S^{1}\times \check Y$. Using radial
       parallel dilation in the $B_{v}$ directions, we can alter
       $(E,A)$ so that the bifold connection is flat on $S^{1}\times
       B_{v}$. The holonomy group of $(E,A)$ on the slice $B_{v}$ is the Klein
       $4$-group is $V_{4}={1,a,b,c}$. On $S^{1}\times B_{v}$, the
       flat connection is determined by the holonomy around four
       loops: first the three loops which are the meridians of the
       three edges in $B_{v}$, and fourth the loops $S^{1}\times p$
       for some point $p$ in $B_{v}\setminus K$. After re-ordering the
       edges, the first three holonomies are $a$, $b$ and $c$, while
       the holonomy around $S^{1}\times p$ is either $1$, $a$, $b$ or
       $c$. If the holonomy around $S^{1}\times p$ is $1$, then $v$
       has Type I, otherwise it has Type II. 

       We see from this description that if $v_{1}$ and $v_{2}$ are
       vertices of the same type (either both Type I or both Type II),
       then, after modifying the bifold connection a little, we can
       find an identification of the neighborhoods \[ \tau : S^{1}\times
       B_{v_{1}} \to S^{2}\times B_{v_{2}} \] such that the bifold
       connections $(E,A)$ and $\tau^{*}(E,A)$ are isomorphic on
       $S^{1}\times B_{v_{1}}$. We can then surger $(Y,K)$ by removing
       the two balls and identifying the boundaries using $\tau$, and
       we obtain a new bifold connection $(E', A')$ for the new pair
       $(Y', K')$. The new bifold connection has the same action as
       the original. If the number of Type II vertices is even, then
       the number of Type I vertices is even also (because any web has
       an even number of vertices in all); so we can then pair up
       vertices of the same type, and apply the above surgery
       repeatedly to arrive at the case where there are no
       vertices. When there are no vertices, we have already seen that
       the index $d$ is even, which means that $\kappa$ belongs to
       $(1/4)\Z$.

       To complete the proof of the lemma, it is sufficient to exhibit
       a single example of a bifold connection $(E,A)$ on some $S^{1}\times
       \check Y$ whose action is $1/8$. So consider the standard
       $1$-instanton, as an anti-self-dual $\SO(3)$-connection on $\R\times
       S^{3}$, invariant the action of $SO(4)$ on $S^{3}$. Let
       $\Gamma\subset \SO(4)$ be the group of order $8$ generated by
       the reflections in the coordinate $2$-planes of $\R^{4}$. The
       standard $1$-instanton descends to an anti-self-dual bifold
       connection $(E,A)$ on $\R\times (S^{3}/\Gamma)$. The bifold
       $\check Y=S^{3}/\Gamma$ corresponds to the pair $(S^{3},K)$,
       where the web $K$ is the $1$-skeleton of a tetrahedron. If we
       modify the connection to be flat near $\pm \infty$, we can glue
       the two ends to obtain a bifold connection on
       $S^{1}\times\check Y$ with action $1/8$.       
   \end{proof}

   We return to the proof of the proposition. Since two edges are
   incident at each vertex of Type II in $K$, the edges $e$ with
   $\eta(e)=1$ form cycles. If the graph is bipartite, then every
   cycle contains an even number of vertices, and it follows that
   there are an even number of vertices of Type II. By the lemma, the
   action is therefore in $(1/4)\Z$ in the bipartite case, and so the
   index $d(E,A)$ is even. By contrast, in the non-bipartite case, the
   example of the tetrahedron web in the proof of the lemma above
   shows that the action may be $1/8$, and the spectral flow
   $\sflow_{\zeta}$ around a closed loop $\zeta$ may therefore be $1$.
\end{proof}

\subsection{The relation between $\Isharp$ and $\Jsharp$}

For a web $K$ in $\R^{3}$, we have already considered the relation
between $\Isharp(K)$ and $\Jsharp(K)$ in a very simple form in
Corollary~\ref{cor:J-to-I}. The essential point in the proof of that
corollary was the mod-$2$ Gysin sequence relating the homology of a
double-cover to the homology of the base, as a mapping-cone.  In the
proof of Corollary~\ref{cor:J-to-I}, passing from $\Jsharp(K)$ to
$\Isharp(K)$ involves taking iterated double-covers of the spaces of
marked bifold connections, so the complex that computes $\Isharp(K)$
can be constructed as an iterated mapping cone. In this section we
will describe the complex a little more fully, as a ``cube''. 

Consider first the situation of a finite simplicial complex $S$ with a
covering  $f: \tilde S \to S$ which is the quotient map for an action
of the group $G=(\Z/2)^{n}$. We use multiplicative notation for the
group. We write $\F[G]$ for the group ring of $G$ over the field $\F$
of $2$ elements. 
With coefficients $\F$, we
can describe the homology of $\tilde S$ almost tautologically as the
homology of the base $S$ with coefficients in a local system
$\Gamma$ whose stalk is the vector space $\F[G]$  at
each point. Following \cite{Papadima-Suciu} we can construct a
spectral sequence abutting to this homology group as follows.

The ring $\F[G]$ contains an ideal $\cI$, the \emph{augmentation
  ideal}
\[
      \cI  = \Bigl\{ \sum_{g\in G} \lambda_{g} g \Bigm |  \sum_{g\in G}
      \lambda_{g}=0 \,\Bigr \}.
\]
The powers of the augmentation ideal define a filtration $\{\cI^{m}\}$
of the group ring. We write $\Gr$ for the associated graded:
\[
\begin{aligned}
    \Gr &=  \bigoplus_{m} \Gr_{m}\\
        &=   \bigoplus_{m} \cI^{m} /\cI^{m+1} .
\end{aligned}
\]
The dimension of the vector space $\Gr_{m}$ is $\begin{pmatrix} n \\ m \end{pmatrix}$.

For any $g\in G$, the element $1+g$ belongs to $\cI$, so
multiplication by $1+g$ lowers the filtration level. So multiplication
by $g$ induces the identity operator on the associated graded
$\Gr$. The filtration of $\cI$ of $\F[G]$ gives rise to a filtration
of the local system $\Gamma$ on $S$, so there is also an associated graded object
$\Gr_{*}(\Gamma)$. The local system $\Gr_{*}(\Gamma)$ on $S$ is trivial,
because the monodromy elements $g\in G$ act trivially on $\Gr$.

Consider now the cohomology of $S$ with coefficients in $\Gamma$. (We
choose cohomology instead of homology for a while, for the sake of the
exposition below.) The
cochain complex $C^{*}(S;\Gamma)$ has a filtration by subcomplexes
$C^{*}(S;\cI^{m})$ and there is an associated spectral sequence. The
$E_{1}$ page of the spectral sequence consists of the homology groups
\[
              H^{*}(S; \Gr_{*}(\Gamma));
\]
and since the local system $\Gr_{*}(\Gamma)$ is trivial, this is
simply
\[
       H^{*}(S; \F) \otimes \Gr_{*} \cong H^{*}(S;\F) \otimes \F^{2^{n}}.
\]
The individual summands here are
\[
  H_{*}(S; \F) \otimes \Gr_{m} \cong H_{*}(S;\F) \otimes
  \F^{\left( {n\atop m} \right)}.
\]
We therefore have:

\begin{proposition}[\cite{Papadima-Suciu}]
    There is a spectral sequence whose $E_{1}$ page is $H^{*}(S; \F)
    \otimes \Gr$ and which abuts to $H^{*}(\tilde S; \F)$.
\end{proposition}

After selecting a generating set $x_{1}, \dots, x_{n}$ for $G$, we can
unravel the picture a little further. 
For any subset $A\subset \{1,\dots,n\}$, let $G_{A}\le G$ be the
subgroup generated by the elements $x_{i}$ with $i\in A$, and let
\[
\begin{aligned}
    \xi_{A} &= \sum_{g\in G_{A}} g \in \F[G].
\end{aligned}
\]
As $A$ runs through all subsets, the elements $\{\xi_{A}\}_{A}$ form a
basis for $\F[G]$. In this basis, the multiplication by $x_{i}$ has
the form
\[
           x_{i} \xi_{A} =
           \begin{cases}
               \xi_{A}, & i\in A ,\\
               \xi_{A} + \xi_{A\cup\{i\}}, & i \not\in A.
           \end{cases} 
\]
For a general element of $G$ of the form $x_{I} =
x_{i_{1}}x_{i_{2}}\dots x_{i_{k}}$ (the product of $k$ distinct
generators), we similarly have
\begin{equation}\label{eq:xIxiA}
         x_{I} \xi_{A} = \sum_{ A \subset A' \subset A \cup I} \xi_{A'}.
\end{equation} 
In terms of the basis $\xi_{A}$, we can describe the augmentation
ideal and its powers by
\[
    \cF^{m} = \text{span}\{ \, \xi_{A} \mid |A| \ge m \,\}.
\] 
The formula \eqref{eq:xIxiA} shows that, when $|A|=m$,
\begin{equation}\label{eq:d-filtered}
            x_{I} \xi_{A} =  \xi_{A} \pmod{ \cF^{m+1}}.
\end{equation}

Consider again the local system $\Gamma$ on $S$, and regard $\Gamma$
as having trivialized stalks:
so at each vertex of $S$ the stalk is $\F[G]$, and to each $1$-simplex
$\sigma$ is a assigned multiplication by $x_{I}$ for some
$I=I(\sigma)$. From the filtration $\{\cI^{m}\}$ of $\F[G]$, we obtain
a filtration of the cochain complex $C^{*}(S; \Gamma)$, as discussed
earlier. 
If we write
$d_{\Gamma}$ for the coboundary operator, we can compare $d_{\Gamma}$
to the ordinary coboundary operator $d$ with trivial constant
coefficients $\F[G]$. We use the basis $\xi_{A}$ to write
\[
            C^{*}(S; \Gamma) = \bigoplus_{A} C^{*}(S; \F)
\]
so that
\[
              d_{\Gamma} = \sum_{B\supset A} d_{A,B},
\]
where $d_{A,B}$ is  the component from the summand indexed by $A$ to
the summand indexed by $B$. For each $1$-simplex $e$ and each $i\in
\{1, \dots, n\}$ let
\[
         u_{i}(e)=
         \begin{cases}
             1 ,& i\in I(e) \\
             0 ,& i\not\in I(e).
         \end{cases}
\]
These $1$-cochains $u_{i}$ on $S$ are cocycles representing the  cohomology classes in $H^{1}(S;\F)$
which are the dual basis to the elements $x_{i}\in G \cong H_{1}(S;\F)$.
 For $I=\{i_{1},\dots i_{k}\}$, define a cochain $U_{I}$ by
\[
         U_{I}(e) = u_{i_{1}}(e) u_{i_{2}}(e) \dots u_{i_{k}}(e).
\]
(This could equivalently be written as 
\[ U_{I} =
u_{i_{1}}\cupprod_{1} u_{i_{2}} \cupprod_{1} \dots \cupprod_{1} u_{i_{k}}
\]
where $\cupprod_{1}$ is the reduced product of \cite{Steenrod}.) From
\eqref{eq:xIxiA}, we obtain the formulae
\[
       d_{A,A} = d,
\]
and for $B \supsetneq A$,
\[
         d_{A,B} \alpha = U_{B\setminus A} \cupprod \alpha.
\] 

If we think of the subsets $A$ as indexing the vertices of a cube,
then we have a complex computing $H^{*}(\tilde S;\F)$ which consists
of a copy of $C^{*}(S;\F)$ at each vertex. In the spectral sequence
corresponding to the filtration $\cF^{m}$, the $E_{0}$ page is a
direct sum of copies of the complex $(C^{*}(S), d)$, one at each
vertex of the cube. The $E_{1}$ page has a copy of $H^{*}(S;\F)$ at each
vertex of the cube and the differential on $E_{1}$ is the sum of maps
corresponding to edges of the cube: each such map is multiplication by
a class $[u_{i}]$ in $H^{1}(S;\F)$.

To describe essentially the same situation in using Morse homology in
the context of $\Jsharp$ and $\Isharp$, let $K\subset \R^{3}$ be a
web, let $K^{\sharp}=K\cup H$ be the web in $S^{3}$ obtained by adding
a Hopf link $H$ near infinity. Recall that the groups $\Jsharp(K)$
and $\Isharp(K)$ are defined using two different marking data $\mu$
and $\mu'$ for bifold connections on $(S^{3}, K\cup H)$. If $\bonf$
and $\bonf'$ denote the corresponding spaces of bifold connections,
then there is a map (see Lemma~\ref{lem:covering})
\[
           f: \bonf' \to \bonf
\]
which is a covering of a connected component of $\bonf$. The covering
group is the group \[ G=H^{1}(\R^{3}\sminus K;\F), \] and the image of $f$ consists of
bifold connections $(E,A)$ such that $w_{2}(E)$ is zero on
$\R^{3}\setminus K$ (regarded as a subset of $S^{3}\setminus
K^{\sharp})$). Because we wish to write the group $G$
multiplicatively and $H^{1}(\R^{3}\sminus K;\F)$ additively, we write
$x^{h}\in G$ for the element corresponding to $h\in
H^{1}(\R^{3}\sminus K ; \F)$. We also write $x^{h}$ for the
corresponding linear operator on $\F[G]$, given by translation. 
The marking data $\mu$ ensures in any case that
$w_{2}(E)$ is zero on the sphere are infinity in $\R^{3}$; and if $K$
is connected, this means that $f$ is surjective. For connected webs,
we are therefore in the situation that $\bonf$ is the quotient of
$\bonf'$ by the action of $G$.

At this point, for \emph{connected} webs $K$, we can regard $\Isharp(K)$ as
the Floer homology of the perturbed Chern-Simons functional on $\bonf$
with coefficients in a local system with fiber $\F[G]$. Using the
filtration of $\F[G]$ by the powers $\cI^{m}$ of the augmentation
ideal as above, we obtain a filtration of the complex for
 $\Isharp(K)$, and hence a spectral sequence:

 \begin{proposition}
  For a connected web in $\R^{3}$, there is a spectral sequence whose  
$E_{1}$ page is the vector space
\[
      E_{1} = \Jsharp(K) \otimes \Gr,
\]
graded by the grading of $\Gr$, and which abuts to $\Isharp(K)$.
 \end{proposition}

After choosing a basis for $H^{1}(\R^{3}\sminus K;\F)$ (i.e.~an
$n$-element generating set for $G$), we can describe the $E_{1}$ page
as a direct sum of $2^{n}$ copies of $\Jsharp(K)$, indexed by the
vertices $A$ of a cube of dimension $n$.

In the Morse theory picture that underlies $\Isharp$, we can describe
the total differential on the $E_{0}$ page. Let $(C^{\sharp}(K), d)$
be the instanton Morse complex which computes $\Jsharp(K)$, whose
basis is the set of critical points of the perturbed Chern-Simons
functional on the corresponding space of bifold connections
$\bonf^{\sharp}$. Fix a basis for $H^{1}(\R^{3}\sminus K ; \F)$ (or
equivalently a dual basis $\{\gamma_{i}\}$  in $H_{1}(\R^{3}\sminus K; \F)$),
so
that the $E_{0}$ page of the spectral sequence can be written as a
direct sum of $2^{n}$ copies of the complex $C^{\sharp}(K)$ indexed by
the vertices $A$ of the cube:
\[
 E_{0} = \sum_{A\subset \{1,\dots,n\}} C^{\sharp}(K)_{A}.
\]

For each basis element $\gamma_{i}$, there is a corresponding chain
map
\[
    u_{i} :  C^{\sharp}(K) \to  C^{\sharp}(K).
\]
Its matrix entries can be defined by first choosing a suitable
codimension-$1$ submanifold $V_{i} \subset \bonf^{\sharp}$ representing
the corresponding element of $H^{1}(\bonf^{\sharp})$ and then defining
the matrix entry from $\alpha$ to $\beta$ to be count of index-$1$
instanton trajectories from $\alpha$ to $\beta$ whose intersection
with $V_{i}$ is odd. More generally, for each subset $I\subset
\{1,\dots,n\}$, there is a map (not a chain map in general),
\[
         U_{I} :  C^{\sharp}(K) \to  C^{\sharp}(K)
\]
obtained by counting trajectories having odd intersection with
$V_{i}$ for every $i\in I$. (Thus $U_{I}=d$ in the case that $I$ is empty.) 
The complex that computes $\Isharp(K)$ is then
$(E_{0}, D)$ where $D$ has components as follows. For every pair of
subsets $A\subset B$, there is  a component
\[
       D^{A,B} : C^{\sharp}(K)_{B} \to C^{\sharp}(K)_{A}
\] 
given $D^{A,B}=U_{B\sminus A}$. 

\subsection{Replacing $\SO(3)$ by $\SU(3)$}

One can construct a variant of $\Jsharp(K)$ using the group $\SU(3)$
in place of $\SO(3)$. Given a web $K$ in a $3$-manifold $Y$, form the
disjoint union $K \cup H$, where $H$ is a Hopf link contained in a
ball $B\subset Y$ Consider $(Y, K\cup H)$ as defining a bifold, just as in the
$\SO(3)$ case. Fix a $U(3)$ bifold bundle with a reduction to $\SU(3)$
outside the ball $B$, whose determinant is non-trivial on the
periphal torus of the components of $H$. (This non-triviality of the
bundle plays the same role as the strong marking data in the $\SO(3)$
case.)
The projectively flat connections with
fixed determinant on such a bifold bundle correspond to
representations 
\[
         \rho : \pi_{1}(Y\sminus K, y_{0}) \to \SU(3)
\]
satisfying the constraint that, for each edge $e$ of $K$ and any
representative $m_{e}$ for the conjugacy class of the meridian of $e$,
the element $\rho(m_{e})$ has order $2$ in $\SU(3)$. (Compare
Definition~\ref{def:Rsharp}.) The $\SO(3)$ representation variety
$\Rep^{\sharp}(Y,K)$ sits inside the larger $\SU(3)$ representation variety as
the fixed-point set of complex conjugation acting on the matrix
entries. 

One can go on to construct the corresponding instanton homology
groups, which we temporarily denote by  $L^{\sharp}(K)$. Unlike the $\SO(3)$ case, this can now be done with integer
coefficients. For webs in $\R^{3}$, with
$\F$ coefficients, it turns out that the dimension $L^{\sharp}(K; \F)$ is equal to the number of Tait
colorings of $K$. In particular, it is independent of the spatial
embedding of the web. A proof in the $\F$ case can be given using
the ideas of \cite{KM-jsharp-triangles}. The same result holds with
rational coefficients, but with a different proof. One can imagine
that a Smith-theory argument \cite{Smith} might establish an
inequality \[ \dim
L^{\sharp}(K; \F)\ge \dim \Jsharp(K) \] for planar webs $K$, in
which case the four-color theorem would follow from
Theorem~\ref{thm:non-vanishing}. However, there is not a version of
Smith theory that applies to these variants of Floer homology, and one
must use the hypothesis of planarity.

\bibliographystyle{abbrv}
\bibliography{fcs}

\end{document}